\documentclass{amsart}
\usepackage[utf8]{inputenc}
\usepackage{fourier}
\usepackage{graphicx}
\usepackage{physics}
\usepackage{color, amsmath, amssymb}
\usepackage{dsfont}
\usepackage{listings}
\usepackage{tikz-cd}
\usetikzlibrary{trees}

\usepackage{amsthm}
\usepackage{matlab-prettifier}
\usepackage{comment}
\usepackage{hyperref}
\hypersetup{
    colorlinks=true,
    linkcolor=blue,
    filecolor=magenta,      
    urlcolor=cyan,
    pdftitle={Overleaf Example},
    pdfpagemode=FullScreen,
    }

\usepackage{algorithm}
\usepackage{algorithmic}

\numberwithin{equation}{section}
\newtheorem{thm}{Theorem}[section]

\newtheorem{lem}[thm]{Lemma}

\newtheorem{cor}[thm]{Corollary}
\theoremstyle{definition}

\newtheorem{prop}[thm]{Proposition}
\theoremstyle{definition}

\newtheorem{defn}[thm]{Definition}
\theoremstyle{definition}

\theoremstyle{definition}
\newtheorem{rem}[thm]{Remark}
\newtheorem{remark}[thm]{Remark}

\newcommand{\e}{\bf{e}}

\newcommand{\cS}{\mathcal{S}}
\newcommand{\f}{\phi}
\newcommand{\ra}{\rightarrow}
\newcommand{\sg}{\sigma}
\def\du{\bigoplus}
\DeclareMathOperator{\Int}{Int}

\newcommand{\N}{\mathbb N}
\newcommand{\R}{\mathbb R}
\newcommand{\Z}{\mathbb Z}
\newcommand{\C}{\mathbb C}
\newcommand{\one}{\mathbf{1}}

\newcommand{\D}{\mathbb D}
\newcommand{\E}{\mathbb E}

\renewcommand{\L}{\mathcal L}
\newcommand{\fs}{\mathcal{S}}
\newcommand{\cfe}{\mathcal{C} \mathcal{F}_E}

\newcommand{\dist}{\text{dist}}
\newcommand{\diam}{\text{diam}}

\newcommand{\err}{\text{err}}

\newcommand{\Fin}{ \text{Fin}}
\newcommand{\cH}{\mathcal{H}}

\newcommand{\om}{\omega}

\newcommand{\normp}[1]{\left\lVert#1\right\rVert_P}

\title{Rigorous Hausdorff dimension estimates for conformal fractals}
\author{Vasileios Chousionis}
\address{Department of Mathematics, University of Connecticut}
\email{vasileios.chousionis@uconn.edu}

\author{Dmitriy Leykekhman}
\address{Department of Mathematics, University of Connecticut}
\email{dmitriy.leykekhman@uconn.edu}

\author{Mariusz Urba\'nski}
\address{Department of Mathematics, University of North Texas}
\email{mariusz.urbanski@unt.edu}

\author{Erik Wendt}
\address{Department of Mathematics, University of Connecticut}
\email{erik.wendt@uconn.edu}

\thanks{V.~C.\ was supported by Simons Foundation Collaboration grant 521845 and by  NSF grant 2247117. V.C. and D.L. were supported by NSF grant 2513639. M.U. was supported by Simons Foundation Collaboration grant 581668.}
\date{\today}

\begin{document}

\begin{abstract} We develop a versatile framework which allows us to rigorously estimate the Hausdorff dimension of finitely irreducible maximal conformal graph directed Markov systems in $\mathbb{R}^n$ for $n \geq 2$. Our method is based on piecewise linear approximations of the eigenfunctions of the Perron-Frobenius operator via a finite element framework for discretization. One key element in our approach is obtaining bounds for the derivatives of these  eigenfunctions, which,  besides being essential for the implementation of our method, are of independent interest.
\end{abstract}

\maketitle
\tableofcontents

\section{Introduction}

 Understanding and determining the Hausdorff dimension of various and diverse attractors has played a crucial role in advancing the fields of fractal geometry and dynamical systems. In particular, one of the most influential results in iterated function systems, due to Hutchinson \cite{hutch}, asserts that if $\mathcal{S}=\{\phi_i\}_{i =1}^k$ is a set of similitudes which satisfies the open set condition, and $J$ is the unique compact set such that $J=\cup_{i=1}^k \phi_i(J)$ (frequently called the \textit{limit set} or the \textit{attractor} of $\mathcal{S}$), then 
$\dim_{\mathcal{H}}(J)$ is the parameter $ t \in [0,\infty)$ so that 
\begin{equation}
\label{eq:hutch}
\sum_{i =1}^k r_i^t = 1,
\end{equation}
where $r_i \in (0,1)$ are the contraction ratios of the maps $\phi_i$.

The dimension theory of \textit{conformal iterated function systems} (CIFS) is much more complex. In \cite{MUCIFS} Mauldin and the third named author employed thermodynamic formalism to determine the Hausdorff dimension of limit sets of CIFSs. Under the standard assumptions, the Hausdorff dimension of the limit set is Bowen's parameter
\[
    t^*=\inf\{t\geq0:P(t)\leq0\},
\]
where $P$ is the corresponding topological pressure; see Section~\ref{sec:prelim}.
For a regular system, $t^*$ is the unique zero of $P$. Finite systems are regular,
whereas a countably infinite system need not have a pressure zero. This approach
traces back to the fundamental work of Rufus Bowen \cite{bowen}.
Using Hutchinson's formula \eqref{eq:hutch} one can determine the Hausdorff dimension of self similar sets with very high precision. However, due to the complexity of the pressure function, obtaining rigorous and effective estimates for the Hausdorff dimension of self-conformal sets is significantly subtler.

Consider, for example the set of irrational numbers whose continued fraction expansion can only contain digits from a prescribed set $E \subset \N$, i.e.
$$J_E=\left\{ [e] : e \in E^{\N}\right\}
\quad\mbox{ where } \quad
[e]=[e_1, e_2, \ldots] = \cfrac{1}{e_1 + \cfrac{1}{e_2 + \ldots}}.$$
Quite conveniently, the set $J_E$ is the limit set generated by the family
$\cfe=\{\phi_e:[0,1] \ra [0,1]\}_{e \in E},$
where 
$\phi_{e} (x) = \frac{1}{e+x}.$
If $1\in E$, this family is not uniformly contracting on $[0,1]$, since
$|\phi_1'(0)|=1$. Passing to all two-letter compositions gives a uniformly
contracting CIFS with the same limit set: indeed,
$|(\phi_a\circ\phi_b)'(x)|=(a(b+x)+1)^{-2}\leq1/4$.
Estimating $\dim_{\cH}(J_E)$ for $E \subset \N$ is of particular historical and contemporary interest. The problem first appeared in Jarnik's work \cite{jarnik} during the late 1920s in relation to \textit{Diophantine approximation} and \textit{badly-approximable numbers}. Specifically, Jarnik obtained dimension estimates when $E=\{1,2\}$. Jarnik's result was subsequently improved and extended by many authors \cite{bu2, bu1, cu1, cu2, cu3,  Falk, FN_high_order, good, hentex, hen1, henbook, hen2, jenktexan, JenkinsonO_PollicottM_2001, poli,  Pollicott_Vytnova_2022}. In particular, the work of Pollicott and his collaborators has been very influential in the field and led to some impressively accurate estimates. Notably, Pollicott and Vytnova in \cite{Pollicott_Vytnova_2022} were able to rigorously estimate $\dim_{\cH} (J_{\{1,2\}})$ with an accuracy of 200 digits. They used a different technique based on Chebyshev–Lagrange polynomial approximations with positive-test-function/min–max inequalities. Additionally, rigorous bounds for $\dim_{\cH}(J_E)$ were needed in a seminal work by Kontorovich and Bourgain \cite{kontor} and subsequent work of Huang \cite{huang} to prove an almost everywhere version of Zaremba's Conjecture. More precisely, lower bounds for $\dim_{\cH}(J_{\{1,2,\dots, 50\}})$ and $\dim_{\cH}(J_{\{1,2,\dots, 5\}})$ were respectively employed in \cite{kontor} and \cite{huang}. These bounds were justified rigorously in \cite{polcont} and they also follow from \cite{Falk}. In \cite{CLU2} the first three authors obtained rigorous estimates of $\dim_{\cH}(J_E)$ for a wide variety of subsets $E \subset \N$, such as the primes, various powers, arithmetic progressions, etc. These estimates played a crucial role in the study of the dimension spectrum of continued fractions with restricted digits in \cite{CLU2}, and they were also recently used in \cite{DasSim}.

We will focus our attention on dimension estimates of limit sets in the general framework of {\em conformal  graph directed Markov systems} (CGDMS). For the moment, we will only describe CGDMSs briefly and we will discuss them in more detail in Section \ref{sec:prelim}. A CGDMS in $\R^n$ is structured around a directed multigraph $(E,V)$ with a countable set of edges $E$ and a finite set of vertices $V$, and an incidence matrix $A:E \times E \ra \{0,1\}$. Each vertex $v \in V$ corresponds to a pair of sets $(X_v,W_v), \ X_v, W_v \subset \R^n$ such that $X_v$ is compact and connected, $W_v$ is open and connected and $X_v \subset W_v$. For each edge $e \in E$ there exists a contracting map $\phi_{e}:X_{t(e)} \ra X_{i(e)}$ which extends to a $C^1$ conformal diffeomorphism from $W_{t(e)}$ into $W_{i(e)}$. The incidence matrix $A:E \times E \ra \{0,1\}$ determines if a pair of these maps is allowed to be composed. A CGDMS is called \textit{maximal} when $t(a)=i(b)$ if and only if $A_{a,b}=1$; i.e. all possible compositions are admissible. 

Under two natural conditions, \textit{finite irreducibility} (which can be thought of as a non-degeneracy condition of the graph) and \textit{Open Set Condition (OSC)} (which is a natural separation condition ensuring limited overlap) Mauldin and the third-named author in \cite{Mauldin_Urbanski_2003} developed a rich and robust dimension theory of CGDMSs, see also \cite{CTU, polur, MRUvol3, KU} for related recent advances. Limit sets of maximal and finitely irreducible CGDMSs encompass a diverse range of geometric objects, including limit sets of Kleinian groups, Schottky-type systems, Apollonian circle packings, as well as self-conformal and self-similar sets. This diversity justifies our focus on studying dimension estimates within the unified framework of CGDMSs. 

Several approaches have been developed over the years to computationally estimate the Hausdorff dimensions of limit sets. The method we employ in this paper is based on discretization of the Perron-Frobenius (transfer) operator. The method takes its origin from the work of Ulam \cite[pages 74-75]{Ulam_1960}, where he proposed to use step functions on a partition into nonoverlapping subintervals as an approximation function space. In \cite{LiTY_1976} it was shown that the method converges for $C^2$ maps in one dimension and later the convergence was extended to higher dimensions in \cite{ding1996finite}. In the seminal  work \cite{mcmullen}, McMullen applied the previous ideas  to obtain  Hausdorff dimension estimates up to 4-5 digits of accuracy for a wide variety of conformal fractals, such as the Apollonian gasket, Schottky groups, and various Julia sets; however the results lack additional arguments to make them rigorous. Although the step function method is quite versatile, due to its low order approximation it is rather inefficient. Naturally, one may consider higher order polynomial approximations. Thus,
Falk and Nussbaum \cite{Falk,FalkRS_NussbaumRD_2016b} showed that in the case of CIFSs arising from continued fractions, both real and complex,  the approximation can be extended to the piecewise linear case and since linear interpolation is positivity preserving (linear interpolation of positive functions is positive), Lemma \ref{lemma: monotonicity} (see Section \ref{sec: numerical}) provides the mathematical justification for the approximation. The method is very flexible, offers a wide range of generalizations and extensions (see Section \ref{sec: numerical} for description of the method using the finite element framework) and the corresponding partition does not need to be Markov as in the case of the step functions proposed by Ulam \cite{Ulam_1960}. 
Falk and Nussbaum subsequently developed a rigorous
higher-order method using continuous piecewise polynomial
approximations \cite{FN_high_order}.
Although the resulting matrices need not have nonnegative
entries, they establish a form of positivity with respect
to a suitable invariant cone. This ``hidden positivity'',
together with explicit eigenfunction derivative bounds,
yields rigorous spectral radius and dimension bounds.
Their analysis treats systems on subsets of $\R$.
Our construction uses piecewise linear finite elements on
simplicial meshes in higher dimensions, retaining
nonnegative comparison matrices.

We also would like to mention that there is an alternative approach to discretizing the transfer operator using global polynomial approximations \cite{Babenko_1978, Pollicott_Vytnova_2022}. 
Vytnova and Wormell \cite{Vytnova_Wormell_2025}
develop an approximation framework in multiple dimensions.
Their Theorem~3.3 gives explicit bounds for finite-rank
Chebyshev--Lagrange approximations of transfer operators
in multiple dimensions. The hypotheses require analytic
extensions of the branches mapping a common complex domain
into a strictly smaller domain, together with summable
bounds on the analytic weights. Under these assumptions,
the approximation error decays exponentially in the
polynomial degree, with dimension-dependent factors.
Our finite element construction uses local interpolation
and explicit eigenfunction derivative bounds to obtain
second-order error estimates in the mesh diameter.
The distinction therefore concerns the approximation
spaces and quantitative hypotheses, rather than the
availability of a multidimensional approximation theory.
Rigorous bounds are obtained using the positive-operator
comparison principle discussed below; see, for example,
\cite[Lemma 3.1]{Pollicott_Vytnova_2022}.
The system-dependent work lies in constructing suitable
test functions and rigorously bounding the resulting ratios,
rather than in the abstract comparison principle itself. For certain problems this approach yields very impressive results with many digits of accuracy; see for example the aforementioned paper \cite{Pollicott_Vytnova_2022}, where highly accurate estimates are obtained for several one-dimensional continued fractions subsystems, and the recent paper of Vytnova and Wormell \cite{Vytnova_Wormell_2025} where the Hausdorff dimension of the Apollonian gasket is estimated with high precision.
We also would like to point out that there exists a large body of literature on approximation of Perron-Frobenius and Koopman operators not in the context of Hausdorff dimension estimates, see e.g. \cite{arbabi2017ergodic, bandtlow2020lagrange,colbrook2024rigorous, dellnitz1999approximation, junge2009discretization, klus2015numerical}.

Our goal in this paper is not to obtain the sharpest results possible, but to develop a versatile method that can provide rigorous and reasonably accurate Hausdorff dimension estimates for a very broad family of conformal fractals. Our approach relies on piecewise linear approximations of the eigenfunctions of the following \textit{Perron-Frobenius operator}. Given any maximal and finitely irreducible CGDMS $\cS$ we define the Perron-Frobenius operator
\begin{equation*}
 F_t: C(X) \ra C(X), \quad \quad F_t(g)(x) = \sum_{e \in E} \|D\phi_e(x)\|^t g(\phi_e(x)) \chi_{X_{t(e)}}(x).
\end{equation*}
Here $\chi_A$ denotes the characteristic function of $A$.
Each summand is understood to be zero outside its branch
domain, without evaluating the branch there.
The set $X=\du_{v\in V}X_v$, and $t>0$ is any parameter
such that $P(t)$, the topological pressure of the system
evaluated at $t$, is finite.
For any such parameter $t$, there exists a unique continuous function $\rho_t: X \rightarrow [0,\infty)$, normalized by $\int_X\rho_t\,dm_t=1$, so that
\begin{equation}\label{eq: F_t rho}
F_t (\rho_t) = e^{P(t)} \rho_t,
\end{equation}
see Section \ref{sec:rn} for more details. 

For a regular CGDMS with positive Hausdorff dimension $t^*$, we have
$P(t^*)=0$, and \eqref{eq: F_t rho} identifies $t^*$ as the parameter for
which $F_{t^*}$ has leading eigenvalue $1$. More generally, certified bounds
placing the spectral radius above $1$ at one finite-pressure parameter and
below $1$ at a larger parameter prove that the dimension lies between them.
Such a sign change also establishes regularity, rather than presupposing it.
If an infinite system is irregular, one must instead use the definition of
Bowen's parameter; a pressure-root search is not available. Smooth
eigenfunctions with effective derivative bounds make the spectral comparison
approach particularly useful for rigorous dimension estimates.

A common principle underlying the methods of Falk--Nussbaum
\cite{Falk,FalkRS_NussbaumRD_2016b},
Pollicott--Vytnova \cite{Pollicott_Vytnova_2022},
Vytnova--Wormell \cite{Vytnova_Wormell_2025},
and the present paper is the following positive-operator comparison.
For a bounded positive operator $L:C(X)\to C(X)$ on a compact
space $X$ and any strictly positive function $f\in C(X)$,
\[
    \inf_{x\in X}\frac{Lf(x)}{f(x)}
    \leq r(L)
    \leq
    \sup_{x\in X}\frac{Lf(x)}{f(x)}.
\]
For the transfer operators considered here,
$r(F_t)=e^{P(t)}$, so certified comparisons with $1$ give
pressure bounds and hence Hausdorff dimension bounds.
The approaches differ in how suitable approximations are
constructed and their errors are controlled. Our method uses
positivity-preserving piecewise linear finite element
interpolation, explicit derivative bounds for the eigenfunction,
and, for infinite alphabets, computable tail estimates.
These ingredients produce nonnegative comparison matrices
whose spectral radii bound $r(F_t)$. The finite-dimensional
version of the same positive-operator principle,
Lemma~\ref{lemma: monotonicity}, then provides the spectral
comparisons used in our certification.

Inspired by the work of Falk and Nussbaum \cite{Falk,FalkRS_NussbaumRD_2016b}, we provide a method for bounding transfer-operator spectral radii at finite-pressure parameters for maximal and finitely irreducible CGDMSs in $\R^n$, $n\geq2$. 
The framework requires quantitative verifications for each
system. These include conformal extension neighborhoods
with explicit positive margins, effective derivative bounds
for the eigenfunction extension on the computational domain,
and a compatible mesh whose domain contains $X$ and is
preserved by the branches. For infinite alphabets, an explicit
uniform estimate for the omitted tail is also required.
Thus, the framework provides a general finite element
template; obtaining a certified dimension enclosure depends
on verifying these conditions and the required spectral
inequalities for the particular system.

Two-sided dimension estimates via a pressure sign change apply to regular systems, with regularity itself certified by the sign change. In this approach, we focus our attention on the action of $F_t$ on piecewise linear approximations of the analytic extension $\tilde{\rho}_t$ of the eigenfunction $\rho_t$ on some mesh domain $X^h\supseteq X$. Provided that
$h$ is small and good estimates for the second derivatives of $\tilde{\rho}_t$ are available, we have accurate piecewise linear approximations of $\tilde{\rho}_t$ and our method yields rigorous Hausdorff dimension estimates with several digits of accuracy even for limit sets in $\R^n$ for $n >2$.

As mentioned earlier, our strategy depends on certain derivative bounds for the eigenfunctions $\rho_t$ of the Perron-Frobenius operator $F_t$. Falk and Nussbaum obtained such bounds for second order derivatives in the case of CIFSs defined via real and complex continued fraction algorithms using some very technical arguments (especially in the case of complex continued fractions). In Section \ref{sec:derbound}  (Theorem \ref{thm:derivative_bound}) we prove that the eigenfunctions $\rho_t$  admit real analytic extensions and they satisfy the desired inequalities for derivatives of all orders. More precisely if $\cS$ is a maximal and finitely irreducible CGDMS in $\R^n, n \geq 2,$ then for any multi-index $\alpha$:
\begin{enumerate}
\item \label{eq:introderbound1} There exists a computable constant $C_1(t,\alpha)>0$ such that if $\cS$ consists of M\"obius maps: 
\begin{equation*}
        |D^\alpha \rho_t (x)| \leq {C_1(t,\alpha)} n^{|\alpha|/2}\eta_\cS^{-|\alpha|}  \, \rho_t(x), \quad \forall x \in X.
    \end{equation*} 
\item
\label{eq:introderbound2} There exists a computable constant ${C_2(t,\alpha)}>0$ such that if $n=2$: 
\begin{equation*}
        |D^\alpha \rho_t (x)| \leq {C_2(t,\alpha)} \eta_\cS^{-|\alpha|} \, \rho_t(x), \quad \forall x \in X.
    \end{equation*} 
\end{enumerate}
Here $\eta_\cS=\min_{v\in V}\dist(X_v,\partial W_v)$.
Besides being key ingredients in our methods, we consider that these derivative bounds have  independent value and they might also find applications in other related problems. The proof of Theorem \ref{thm:derivative_bound}, which is quite short and streamlined, employs complexification and some basic tools from the theory of several complex variables. We also stress that the open set condition is not required for Theorem \ref{thm:derivative_bound}, i.e. for \eqref{eq:introderbound1} and \eqref{eq:introderbound2}. 

In Sections \ref{sec: numerical} and \ref{sec:app}  we discuss a selection of CGDMSs to which our method applies. Due to length considerations we decided not to include an exhaustive list of applications, but we focused on examples which highlight the versatility of our method. We gather our estimates in Table \ref{tableintro}.

We pay particular attention to CIFSs which are defined by continued fraction algorithms. We rigorously estimate the Hausdorff dimension of limit sets of CIFSs defined by complex continued fractions, earlier considered in \cite{FalkRS_NussbaumRD_2016b}, and for the first time, we also provide estimates for the complex continued fraction system whose alphabet is the set of Gaussian primes. We also introduce a CIFS modeled on higher dimensional continued fraction algorithms and we provide the first dimension estimates for the limit set of the three-dimensional continued fraction system. To the best of our knowledge this is the first example of a genuine $3$-dimensional CIFS (meaning that the generating conformal maps are defined in $\R^3$, they are not similarities, and the limit set is not contained in any lower dimensional affine subspace of $\R^3$) where a rigorous numerical method is applied in order to estimate the Hausdorff dimension of its limit set. 

We also discuss how our method can be applied to limit sets of systems defined by quadratic perturbations of linear maps. We included this example in order to highlight the fact that our method can also be applied to systems which do not consist of M\"obius maps.  We do stress that Falk and Nussbaum considered systems in $\R$ consisting of power perturbations of linear maps in \cite[Sections 3.3 and 5]{Falk}.

Since our method applies to CGDMSs as well as CIFSs, we also consider Schottky-type conformal systems and estimate the Hausdorff dimensions of their limit sets.

Finally, we provide rigorous estimates for the Hausdorff dimension of the Apollonian gasket, and for the Hausdorff dimension of several limit sets of its subsystems. Although there exist several non-rigorous estimates for the Hausdorff dimension of the Apollonian gasket \cite{mcmullen,baifinch}, before 2025 there was only one rigorous estimate, due to Boyd \cite{boyd}. As mentioned earlier, Vytnova and Wormell \cite{Vytnova_Wormell_2025} in 2025 obtained rigorous and very accurate (up to 128 digits) estimates for the Hausdorff dimension of the Apollonian gasket. While our method applied to the gasket yields estimates that are notably less accurate than those achieved by Vytnova and Wormell, it offers the advantages of ease of implementation and high flexibility. These attributes allow us to derive rigorous and effective estimates for the Hausdorff dimensions of various subsystems of the Apollonian gasket. This is crucial in our work \cite{CLUW_2025} where we identify the gasket's dimension spectrum, and we rely on rapid, rigorous and effective estimates for a broad range of its subsystems.

We summarize our numerical findings in the following table.

\begin{table}[ht]
\caption{INTLAB Certified Hausdorff dimension intervals for various examples.}
\label{tableintro}
\centering
\vspace{.1in}
\begin{tabular}{|c|c|}
\hline
{Example}           & certified interval \\
\hline
$2D$ Continued fractions with 4 generators & {$[1.149574, 1.149582]$} \\
\hline
$2D$ Continued fractions  & {$[1.8556,1.8568]$}  \\
\hline
$2D$ Continued fractions on Gaussian primes &  {$[1.5100,1.5152]$}  \\
\hline
$3D$ Continued fractions with 5 generators & {$[1.4475,1.4491]$} \\
\hline
$3D$ Continued fractions & $[2.566, 2.633]$ \\
\hline
A quadratic $abc$-example & {$[0.63167, 0.63198]$} \\
\hline
2D Schottky-type system & $[0.295495,0.295598]$  \\
\hline
3D Schottky-type system & $[0.3849, 0.3967]$  \\
\hline
12-map Apollonian subsystem & $[1.1136,   1.1145]$\\
\hline
Apollonian gasket & $[1.3047,    1.3081]$ \\
\hline
Apollonian subsystem with 3 infinite families & $[1.0732,    1.0784]$ \\
\hline
\end{tabular}
\label{table: Hausdorff dimension estimates for varous example}
\end{table} 

Table \ref{tableintro} illustrates the generality of our method by providing several rather distinct examples, for which the Hausdorff dimensions are computed with varying degrees of accuracy. The accuracy of the computations depends mainly on the size of the alphabet and the size of the discrete problem (see Section \ref{sec: numerical} for more details). Naturally, the largest and the most computationally intensive problem is 3D Continued fractions on an infinite lattice while a quadratic $abc$-example is the smallest. Our main objective in this paper is to demonstrate that Hausdorff dimensions of a very broad family of conformal fractals are effectively computable. We did not pursue the avenue of giving the best results possible, which we plan to do in future work where we will explore the computational boundaries of our method. 

\section{Preliminaries}
\label{sec:prelim}
In this section we provide all the necessary background on conformal graph directed Markov systems and their thermodynamic formalism. We pay special attention to eigenfunctions of a Perron-Frobenius operator which play a crucial role in our approach.
\subsection{Conformal graph directed Markov systems.}\label{subsec:cgdms}
\begin{defn}A {\it graph directed Markov system} (GDMS) \index{GDMS}
\label{definition: graph}
\begin{equation}
\label{eq:gdms}
\cS= \big\{ V,E,A, t,i, \{X_v\}_{v \in V}, \{\f_e\}_{e\in E} \big\}
\end{equation}
consists of
\begin{enumerate}
\item a directed multigraph $(E,V)$ with a countable set of edges $E$, which we will call the {\em alphabet} of $\cS$, and a finite set of vertices $V$,

\item an incidence matrix $A:E \times E \ra \{0,1\}$,

\item two functions $i,t:E\ra V$ such that $t(a)=i(b)$ whenever $A_{ab}=1$,

\item a family of non-empty compact metric spaces $\{X_v\}_{v\in V}$,


\item  a family of injective contractions $$ \left\{\phi_e:X_{t(e)}\to X_{i(e)}\right\}_{e\in E}$$ such that every $\phi_e,\, e\in E,$ has Lipschitz constant no larger than $s$ for some $s \in (0,1)$.
\end{enumerate}
\end{defn}
When it is clear from context we will  use the simpler notation $\cS=\{\f_e\}_{e \in E}$ for a GDMS.
We will always assume that the alphabet $E$ is not a singleton and for every $v \in V$ there exist $e,e' \in E$ such that $t(e)=v$ and $i(e')=v$. GDMSs with finite alphabets  will be called {\it finite}\index{GDMS!finite}. 

\begin{remark}
\label{ifsosc}
When $V$ is a singleton and for every $e_1,e_2 \in E$, $A_{e_1e_2}=1$ if and only if $t(e_1)=i(e_2)$, the GDMS is called an {\it iterated function system} (IFS). 
\end{remark}

We will use the following standard notation from symbolic dynamics. For every $\om \in E^*:=\bigcup_{n=0}^\infty E^n$, we denote by $|\om|$  the unique integer
$n \geq 0$ such that $\om \in E^n$, and we call $|\om|$ the {\em length} of
$\om$. We also set $E^0=E_A^0=\{\emptyset\}$. For $n \in \N$ and $\om \in
E^\mathbb{N}$, we let
$$
\om |_n:=\om_1\ldots \om_n\in E^n.
$$
If $\tau \in E^*$ and $\om \in E^* \cup E^\mathbb{N}$, then
 $$\tau\om:=(\tau_1,\dots,\tau_{|\tau|},\om_1,\dots).$$
For $\om,\tau\in E^{\mathbb N}$, the longest
initial block common to both $\om$ and $\tau$ will be denoted by $\omega\wedge\tau  \in
E^{\mathbb N}\cup E^*$. The \textit{shift map}
$$
\sg: E^\mathbb{N} \ra E^\mathbb{N}
$$
is given by the formula
$$
\sg\left( (\om_n)^\infty_{n=1}  \right) =  \left( (\om_{n+1})^\infty_{n=1}  \right).
$$
For a matrix $A:E \times E \to \{0,1\}$ we let
$$
E^\mathbb{N}_A
:=\{\om \in E^\mathbb{N}:  \,\, A_{\om_i\om_{i+1}}=1  \mbox{ for
  all }\,  i \in \N
\},
$$
and we call its elements {\it $A$-admissible} (infinite) words. We also set
$$
E^n_A
:=\{\omega \in E^n:  \,\, A_{\om_i\om_{i+1}}=1  \,\, \mbox{for
  all}\,\,  1\leq i \leq
n-1\}, \quad n \in \N,
$$
and
$$
E^*_A:=\bigcup_{n=0}^\infty E^n_A.
$$
The elements of $E_A^\ast$ are called $A$-admissible (finite) words. For every nonempty $\om \in E^*_A$ we let $t(\om) = t(\om_{|\om|})$ and $i(\om)=i(\om_1)$. For
every  $\om \in E^*_A$, we let
$$
[\om]:=\{\tau \in E^\mathbb{N}_A:\,\, \tau|_{|\om|}=\om \}.
$$
Given $v \in V$ we denote $$E^n_A(v)=\{\om \in E^n_A: t(\om)=v\}$$ and
$$E_A^\ast(v)=\cup_{n \in \N}E^n_A(v).$$
For each $a\in E$, we let
$$E_a^\infty:=\{\om\in E_A^\N: A_{a\om_1}=1\}.$$
\begin{defn}
A matrix $A:E \times E \ra \{0,1\}$ will be called {\it finitely irreducible} if there exists a finite set $\Lambda \subset E_A^*$ such that for all $i,j\in E$ there exists $\om\in \Lambda$ for which $i\om j\in E_A^*$. If the associated matrix of a GDMS is finitely irreducible, we will call the GDMS finitely irreducible as well. 
\end{defn}
We will be interested in maximal GDMSs.
\begin{defn} A GDMS $\cS$
with an incidence matrix $A$ is called {\it maximal} \index{GDMS!maximal} if it satisfies the following condition:
$$A_{ab}=1 \mbox{ if and only if }t(a)=i(b).$$
\end{defn}
This notion has an easy colloquial description --- a GDMS is maximal when one can compose maps whose domain and codomain components match.

Let $\cS= \big\{ V,E,A, t,i, \{X_v\}_{v \in V}, \{\f_e\}_{e\in E} \big\}$ be a GDMS. For nonempty $\om \in E^*_A$ we define the map coded by $\om$:
\begin{equation}\label{phi-om}
\phi_\om=\phi_{\om_1}\circ\cdots\circ\phi_{\om_n}:X_{t(\om_n)}\to
X_{i(\om_1)} \qquad \mbox{if\ $\om\in E^n_A$.}
\end{equation}
For $\om \in E^{\mathbb N}_A$, the sequence of non-empty compact sets
$\{\f_{\om|_n}(X_{t(\om_n)})\}_{n=1}^\infty$ is decreasing (in the sense of inclusion)
and therefore their intersection is nonempty.
Moreover,
$$
\diam(\f_{\om|_n}(X_{t(\om_n)})) \le s^n\diam(X_{t(\om_n)})\le s^n\max\{\diam(X_v):v\in V\}
$$
for every $n\in\N$, hence
$$
\pi(\om):=\bigcap_{n\in  \N}\f_{\om|_n}(X_{t(\om_n)})
$$
is a singleton. Thus we can now define the coding map \index{coding map}
\begin{equation}
\label{picoding}
\pi:E^{\mathbb N}_A\to \du_{v\in V}X_v:=X,
\end{equation}
the latter being a disjoint union of the sets $X_v$, $v\in V$.
The set
$$
J=J_\cS:=\pi(E^{\mathbb N}_A)
$$
will be called the {\it limit set} \index{limit set} (or {\it attractor}) \index{attractor} of the GDMS $\cS$.

For  $\alpha > 0$, we define the metrics $d_\alpha$ on
$E_A^{\mathbb N}$ by setting
\begin{equation}\label{d-alpha}
d_\alpha(\om,\tau) ={\rm e}^{-\alpha|\om\wedge\tau|}.
\end{equation}
We record that all the metrics $d_{\alpha}$ induce the same topology. Moreover, see \cite[Proposition 4.2]{CTU}, the coding map $\pi:E^{\mathbb N}_A\to \du_{v\in V}X_v$ is H\"older continuous, when $E^{\mathbb N}_A$ is
equipped with any of the metrics $d_\alpha$ as in \eqref{d-alpha} and
$\du_{v\in V}X_v$ is equipped with the direct sum metric.

Let $U$ be an open and connected subset of $\R^n$. A $C^1$ diffeomorphism $\f: U \ra \R^n$ will be called \textit{conformal} if its derivative at every point of $U$ is a similarity map. We will denote the derivative of $\f$ evaluated at the point $z$ by $D \f(z) : \R^n \ra \R^n$  and we denote its operator norm by  $\|D \f (z)\|$. 
It is well known by \textit{Liouville's theorem}, see \cite[Theorem 19.2.1]{MRU}, that for
\begin{itemize}
\item  $n = 1$ the map $\f$ is conformal if and only if it is a $C^1$-diffeomorphism, 
\item $n=2$ the map $\f$ is conformal if and only if it is either holomorphic or antiholomorphic, 
\item $n \geq 3$ the map $\f$ is conformal if and only if it is a  M\"obius transformation.
\end{itemize}
We can now define conformal GDMSs. \footnote{There are several variants for a definition of GDMS, see e.g. \cite{MRU, KU}. The definition we are using is 
slightly more restrictive, but more convenient
 for our applications.}
For the remainder of this section and for the results below, we assume that
the ambient dimension is $n\geq2$. In particular, the distortion estimates
on the enlarged sets $S_v$ use conformality in dimension at least two;
the one-step condition (iv) alone would not imply these estimates in
dimension one. All statements involving $i(\om)$, $t(\om)$, or
$\phi_\om$ refer to nonempty admissible words. The empty word is used only
in symbolic concatenations, including the definition of finite irreducibility.
\begin{defn}
\label{gdmsdef}\label{Carnot-conformal-GDMS}
A graph directed Markov system $\cS= \big\{ V,E,A, t,i, \{X_v\}_{v \in V}, \{\f_e\}_{e\in E} \big\}$ is called {\it conformal} (CGDMS) if the following conditions are satisfied.
\begin{itemize}
\item[(i)] The metric spaces $X_v, v \in V,$ are compact and connected subsets of a fixed Euclidean space $\R^n$ and $X_v=\overline{\Int(X_v)}$ for all $v\in V$.
\item[(ii)] ({\it Open Set Condition} or {\it OSC}). \index{open set condition} For all $a,b\in
E$, $a\ne b$,
$$
\phi_a(\Int(X_{t(a)})) \cap \phi_b(\Int(X_{t(b)}))= \emptyset.
$$
\item[(iii)] \label{cgdmsiii} For every vertex $v\in V$ there exists an open and connected
set $W_v \supset X_v$ such that for every $\om\in E^{\ast}_A$, the map
$\f_\om$ extends to a $C^1$ conformal diffeomorphism of $W_{t(\om)}$ into $W_{i(\om)}$.
\item[(iv)]({\it Bounded Distortion Property} or {\it BDP})  For each $v \in V$ there exists a compact and connected set $S_v$ such that
$X_v \subset \Int(S_v) \subset S_v \subset W_v$  so that $\phi_e (S_{t(e)}) \subset S_{i(e)}$ for all $e \in E$ and
$$\biggl|\frac{\|D\f_e(p)\|}{\|D\f_e(q)\|}-1\biggr|\le L|p-q|^\alpha
 \mbox{ for all }e\in E\mbox{ and }  p,q\in S_{t(e)},$$
where $\alpha>0$ and $L \geq 1$ are two constants depending only on $\cS, S_v$ and $W_v$.
\end{itemize}
\end{defn}
We will use the abbreviation CIFS for conformal IFS.
\begin{remark}\label{remark: conds 3}
Condition (iii) can equivalently be stated at the level of generators:
\begin{enumerate}
\item[(iii)']  For every vertex $v\in V$ there exists an open connected
set $W_v \supset X_v$ such that for every $e\in E$, the map
$\f_e$ extends to a $C^1$ conformal diffeomorphism of $W_{t(e)}$ into $W_{i(e)}$.
\end{enumerate}
Indeed, the extension property for admissible compositions follows by composing
the extensions of the generators. This equivalence holds in every dimension.
Separately, when $n\geq2$, condition (iv) is superfluous since condition
(iii)' implies condition (iv), with $\alpha=1$; see e.g.
\cite{Mauldin_Urbanski_2003, KU}.
\end{remark}

In the assumed dimension $n\geq2$, conformal distortion estimates on compact
subsets of the common extension domains imply that there exists a constant
$K\geq1$, depending on $\cS$ and the chosen sets $S_v,W_v$, such that
\begin{equation}
\label{bdp}
K^{-1}\le\frac{\|D\f_\om(p)\|}{\|D\f_\om(q)\|}\le K
\end{equation}
for every $\om\in E_A^*$ and every pair of points $p,q\in S_{t(\om)}$. 

For $\om \in E^*_A$ we set
$$\|D \f_\om\|_\infty := \|D \f_\om\|_{X_{t(\om)}}.$$
Note that \eqref{bdp} and the chain rule easily imply
that if $\om \in E_A^\ast$ and $\om=\tau \upsilon$ for some $\tau, \upsilon \in E_A^\ast$, then
\begin{equation}\label{quasi-multiplicativity1}
K^{-1} \|D\f_{\tau}\|_\infty \, \|D\f_{\upsilon}\|_\infty \le
\|D\f_\om\|_\infty \le \|D\f_{\tau}\|_\infty \,
\|D\f_{\upsilon}\|_\infty.
\end{equation}
Moreover, there exists a constant $M$, depending only on $\cS$, such that 
for every $\om \in E^*_A$, and every $p,q \in S_{t(\om)}$,
\begin{equation}
\label{boundbyderlip} 
d(\f_{\om}(p), \f_{\om}(q)) \leq  M K\|D\f_\om\|_\infty d(p,q),
\end{equation}
where $d$ is the Euclidean metric on $\R^n$. In particular for every $\om \in E^*_A$
\begin{equation}
\label{boundbyder} 
\diam(\f_\om(X_{t(\om)}))\leq   M K\|D\f_\om\|_\infty \diam(X_{t(\om)}).
\end{equation}

\subsection{Thermodynamic formalism} 

We will now recall some well known facts from the thermodynamic formalism of GDMSs. Let $\mathcal{S}=\{\f_e\}_{e\in E}$  be a finitely irreducible  conformal GDMS. For $t\ge 0$ and $n \in \N$ let
\begin{equation}\label{zn}
Z_n(\cS,t):=Z_{n}(t) := \sum_{\om\in E^n_A} \|D\phi_\om\|^t_\infty.
\end{equation}
Note that \eqref{quasi-multiplicativity1} implies that
\begin{equation}
\label{zmn}
Z_{m+n}(t)\le Z_m(t)Z_n(t),
\end{equation}
and consequently the
sequence $\{\log Z_n(t)\}_{n=1}^\infty$ is subadditive. Therefore, the limit
$$
P_\mathcal{S}(t):=P(t):=\lim_{n \to  \infty}  \frac{ \log Z_n(t)}{n}=\inf_{n \in \N} \frac{ \log Z_n(t)}{n}
$$ 
exists and it 
is called the {\em topological pressure} of the system $\mathcal{S}$ evaluated at the parameter $t$. We also define two special parameters related to topological pressure; 
$$\theta(\cS):=\theta=\inf\{t \geq 0: P(t) <+\infty\}\quad \mbox { and }\quad h(\cS):=h=\inf\{t \geq 0: P(t)\leq 0\}.$$
The parameter $h(\cS)$ is known as \textit{Bowen's parameter}.
A system is called \textit{regular} if $P(t)=0$ for some $t\geq0$.
In that case this zero is unique and equals $h(\cS)$. Finite finitely
irreducible systems are regular, but countably infinite systems need not be:
the definition of $h(\cS)$ does not assert that $P(h(\cS))=0$.

It is well known that $t\mapsto P(t)$ is nonincreasing on $[0,+\infty)$,
strictly decreasing wherever it is finite, and satisfies
$\lim_{t\to+\infty}P(t)=-\infty$. It is convex and continuous on
$(\theta,+\infty)$ and is right-continuous at $\theta$ when
$P(\theta)<+\infty$; see e.g. \cite[19.4.6]{MRU}. These assertions do not
imply that $P$ has a zero. Moreover
\begin{equation}
\label{thetaz}
\theta({\cS}):=\theta=\inf\{t \geq 0: P(t)<\infty\}=\inf\{t \geq 0: Z_1(t)<\infty\},
\end{equation}
and for $t \geq 0$
\begin{equation}
\label{presz1}
P(t)<+\infty \mbox{ if and only if }Z_1(t)<+\infty.
\end{equation}
The proofs of these facts can be found in \cite[Proposition 7.5]{CTU} and \cite[Lemma 3.10]{dimspec}.

Thermodynamic formalism, and topological pressure in particular, plays a fundamental role in the dimension theory of conformal dynamical systems:
\begin{thm}
\label{721}
If $\mathcal{S}$ is a finitely irreducible conformal GDMS, then
$$
h(\fs)= \dim_{\mathcal{H}}(J_\mathcal{S})
= \sup \{\dim_\cH(J_F):  \, F \subset E \, \mbox{finite} \, \}.
$$
\end{thm}
For the proof see \cite[Theorem 7.19]{CTU} or \cite[Theorem 4.2.13]{Mauldin_Urbanski_2003}. 

We close this section with a discussion regarding conformal measures and Perron-Frobenius operators. If $\cS=\{\f_e\}_{e \in E}$ is a finitely irreducible conformal GDMS we define
$$\Fin(\cS):= \{t\geq  0: Z_1(t)<+\infty\}=\left\{t\geq 0:\sum_{e\in E}||D\phi_e||_\infty^t<+\infty\right\}.$$
Gibbs measures are of crucial importance in thermodynamic formalism of countable alphabet symbolic dynamics.
\begin{defn}
\label{gibbsdef} Let $\cS$ be a finitely irreducible conformal GDMS and let $t \in \Fin (\cS)$. A Borel probability measure $\mu$ on $E_A^\N$ is called a $t$-Gibbs state for $\cS$ (or a Gibbs state for the potential $\om \to t \log \|D \phi_{\om_1}(\pi(\sigma (\om)))\|$) if and only if there exists a constant $C_{\mu,t} \geq 1$ such that
\begin{equation}
\label{gibbseq}
C_{\mu,t}^{-1} e^{-P(t)n} \|D \phi_{\om|_n}(\pi(\sigma^n(\om)))\|^t \leq \mu ([\om|_n]) \leq C_{\mu,t} e^{-P(t)n} \|D \phi_{\om|_n}(\pi(\sigma^n(\om)))\|^t,
\end{equation}
for all $\om \in E_A^\N$ and $n \in \N$.
\end{defn}
For $t \in \Fin(\cS)$ the \textit{Perron-Frobenius operator} with respect to $\cS$ and $t$ is defined as
\begin{equation}\label{1j89}
\mathcal{L}_t g(\om)= \sum_{i:\, A_{i \om_1}=1}
g(i \om)\| D\phi_i(\pi (\om))\|^t \quad \mbox{for $g \in C_b(E^\mathbb{N}_A)$ and $\om
  \in E^\mathbb{N}_A$,}
\end{equation}
where $C_b(E^\mathbb{N}_A)$ is the Banach space of real-valued bounded continuous functions on $E_A^\N$. 
It is well known that $\mathcal{L}_t: C_b(E^\mathbb{N}_A) \ra C_b(E^\mathbb{N}_A)$. Moreover, by a straightforward inductive calculation:
\begin{equation}\label{1j89it}
\mathcal{L}^n_t g(\om)= \sum_{\tau \in E_A^n:\, A_{\tau_n \om_1}=1}
g(\tau \om)\| D\phi_\tau(\pi (\om))\|^t \quad \mbox{for $g \in C_b(E^\mathbb{N}_A)$ and $\om
  \in E^\mathbb{N}_A$.}
\end{equation}
We will also denote by $\mathcal{L}_t^*: C_b^*(E^\mathbb{N}_A)\to C_b^*(E^\mathbb{N}_A)$ the dual operator of $\mathcal{L}_t$. The proof of the following theorem can be found in \cite[Theorem 7.4]{CTU}.
\begin{thm} \label{thm-conformal-invariant}
Let $\mathcal{S}=\{\f_e\}_{e\in E}$  be a finitely irreducible conformal GDMS and let $t\in\Fin(\cS)$. 
\begin{enumerate}
\item \label{eq:eigenmes}  There exists a unique Borel probability eigenmeasure $\tilde m_t$ of the conjugate Perron-Frobenius operator $\mathcal{L}_t^*$ and the corresponding eigenvalue is  $e^{P(t)}$.
\item \label{gibssmest} The eigenmeasure $\tilde m_t$ is a $t$-Gibbs state.

\item \label{gibbsmeserg} There exists a unique shift-invariant $t$-Gibbs state $\tilde\mu_t$ which is ergodic and globally equivalent to $\tilde m_t$.
\end{enumerate}
\end{thm}
For all $t \in \Fin(\cS)$ we will denote
\begin{equation}\label{mutdef}
m_t := \tilde m_t\circ \pi^{-1} \  \  {\rm and  }  \  \  \  \mu_t := \tilde \mu_t\circ \pi^{-1}.
\end{equation}
Note that the measures $m_t, \mu_t$ are probability measures supported on $J_{\cS}$. The measures $m_t$ will be called $t$-conformal (with pressure normalization $e^{P(t)}$). If $\cS$ is regular and $h=h(\cS)>0$, then $h\in\Fin(\cS)$ and $P(h)=0$; in this case $m_h$ is called the \textit{conformal measure} of $\cS$. We do not assume that a pressure zero exists for an arbitrary infinite system.

We will conclude this section with a bound for $\L_t^n(\one)$ which will be of paramount importance in the following Sections.
\begin{prop}
\label{symbbound}
Let $\mathcal{S}=\{\f_e\}_{e\in E}$  be a finitely irreducible conformal GDMS and let $t\in\Fin(\cS)$. There exists a constant $M_t \geq 1$ such that
\begin{equation}
\label{ltn1bd}
M_t^{-1} e^{P(t)n} \leq \L_t^n (\one) (\om) \leq M_t e^{P(t)n},
\end{equation}
for all $\om \in E_A^\N$ and $n \in \N$.
\end{prop}
\begin{proof}
The upper bound follows from \cite[Lemma 18.1.1]{MRU}.
For completeness, we give a proof of the lower bound
using the invariant Gibbs measure and bounded distortion.

 We will first show that for all $a \in E$ and $n \in \N$: 
\begin{equation}
\label{lowbddkey}
\sum_{\om \in E_A^n:A_{\om_n a}=1} \sup \{ \|D \phi_\om (\pi (\tau))\|^t: \tau \in [a]\} \geq C_t^{-2}e^{P(t)n}.
\end{equation}
By Theorem \ref{thm-conformal-invariant} \eqref{gibbsmeserg} and Definition \ref{gibbsdef} we know that there exists some $C_t>0$ such that
\begin{equation}
\label{gibbsmut}C_{t}^{-1} e^{-P(t)n} \|D \phi_{\om|_n}(\pi(\sigma^n(\om)))\|^t \leq \tilde\mu_t ([\om|_n]) \leq C_{t} e^{-P(t)n} \|D \phi_{\om|_n}(\pi(\sigma^n(\om)))\|^t,
\end{equation}
for all $\om \in E_A^\N$ and $n \in \N$. Note that \eqref{gibbsmut} and the chain rule imply that
\begin{equation}
\begin{split}
\label{mutab}
\tilde\mu_t ([\alpha \beta]) \leq C_t &\exp (-(|\alpha|+|\beta|)P(t))\, \sup \{\|D \phi_\alpha (\pi(\tau))\|^t: \tau \in [\beta]\}  \\
&\quad \quad\cdot\sup \{\|D \phi_\beta (\pi (\rho))\|^t: \beta \rho \in E_A^\N\},
\end{split}
\end{equation}
for any $\alpha, \beta \in E_A^\ast$ such that $\alpha\beta \in E_A^\ast$. 

Let $a \in E$. We then see that:
\begin{equation*}
\begin{split}
\label{mta2}
C_t^{-1}& e^{-P(t)} \sup \{\|D \phi_a (\pi (\tau)) \|^t: a \tau \in E_A^\N\}  \overset{\eqref{gibbsmut}}{\leq} \tilde\mu_t([a])= \tilde\mu_t(\sigma^{-n} ([a])) =\sum_{\om \in E_A^n: A_{\om_n a}=1} \tilde\mu_t([\om a]) \\&\overset{\eqref{mutab}}{\leq} C_t e^{-(n+1)P(t)} \sup \{\|D \phi_a (\pi (\tau))\|^t:a \tau \in E_A^\N\}\, \sum_{\om \in E_A^n: A_{\om_n a}=1} \sup \{\|D \phi_\om (\pi (\tau))\|^t: \tau \in [a]\}.
\end{split}
\end{equation*}
Thus \eqref{lowbddkey} follows.

We can now prove the lower bound in \eqref{ltn1bd}. Let $\tau \in E_A^\N$ and $n \in \N$. If $\om \in E_A^n$ and $\om \tau \in E_A^\N$ then by the bounded distortion property:
\begin{equation}
\label{bddlowebd}
\|D \phi_\om (\pi (\tau))\| \overset{\eqref{bdp}}{\geq} K^{-1} \sup \{\|D\phi_\om (\pi (\rho))\|: \rho \in [\tau_1]\}.
\end{equation}
Therefore,
\begin{equation*}
\begin{split}
\L_t^n(\one) (\tau)&=\sum_{ \om \in E_A^n:\om \tau \in E_A^\N} \|D \phi_\om (\pi (\tau))\|^t \\
&\overset{\eqref{bddlowebd}}{ \geq} K^{-t} \sum_{\om \in E_A^n:\om \tau \in E_A^\N } \sup \{\|D\phi_\om (\pi (\rho)) \|^t: \rho \in [\tau_1]\} \overset{\eqref{lowbddkey}}{\geq} K^{-t} C_t^{-2} e^{P(t)n}.
\end{split}
\end{equation*}
The proof is complete.
\end{proof}

\subsection{\texorpdfstring{The Radon-Nikodym derivative $\rho_t=\frac{d \mu_t}{ d m_t}$ for maximal CGDMS}{}}
\label{sec:rn}

Most of the results in this section are essentially known; see e.g \cite[Section 6.1]{Mauldin_Urbanski_2003} for the case of CIFSs or \cite{nusspriverduyn, verma} for related treatments in the case of complete metric spaces. For this reason, we skip most of the proofs. The interested reader can find detailed proofs of the statements as they appear in this section in \cite{Erik_thesis}.

In what follows 
$$
\cS= \big\{ V,E,A, t,i, \{X_v\}_{v \in V}, \{\f_e\}_{e\in E} \big\}
$$
will denote a maximal CGDMS, which does not have to satisfy the open set condition. We will assume that the sets $X_v$ are disjoint. This is not an essential restriction because, as it was described in \cite[Remark 4.20]{CTU}, given any GDMS we can use formal lifts to obtain a new GDMS with essentially the same limit set but whose corresponding compact sets are disjoint.

We start by introducing another Perron-Frobenius operator on $C(X)$; recall from Section \ref{subsec:cgdms} that $X= \du_{v \in V} X_v$. For $t \in \Fin(\fs)$, $g \in C(X)$, let
\begin{equation}\label{ft}
 F_t(g)(x) = \sum_{\substack{e\in E\\x\in X_{t(e)}}} \norm{D\phi_e(x)}^t g(\phi_e(x)).
\end{equation}
In formulas below written with indicator functions, a summand is understood
to be zero outside the domain of its branch, without evaluating that branch there.
It is not difficult to show  that $F_t :(C(X), \|\cdot\|_\infty) \ra (C(X), \|\cdot\|_\infty)$ is a bounded linear operator. Moreover, using that $\cS$ is maximal, one can easily derive a formula for the $n$-th iterate of the operator $F_t$:
\begin{equation}\label{perronit}
F_t^{n}(g)(x) = \sum_{\substack{\om\in E_A^n\\x\in X_{t(\om)}}} \norm{D \phi_\om(x)}^t g(\phi_\om(x)).
\end{equation}
\begin{remark} Maximality ensures that iterating $F_t$ produces exactly the admissible compositions prescribed by the incidence matrix.
\end{remark}

The connection between the Perron-Frobenius operator $F_t$ and the symbolic Perron-Frobenius  operator $\mathcal{L}_t$ defined in \eqref{1j89} can be easily obtained. For every $g \in C(X)$ and $n \in \N$:
\begin{equation}
\label{ltnftn}
\L_t^n  (g \circ \pi)=F_t^n (g) \circ \pi.
\end{equation}
See \cite[p. 425]{KU} for the straightforward calculation leading to \eqref{ltnftn}.

Using Proposition \ref{symbbound}, we obtain upper and lower bounds for \(F_t^n(\mathbf1)\) proportional to \(e^{nP(t)}\), with positive comparison constants independent of \(n\).
\begin{prop}\label{transunifbd} 
Let $\fs$ be a finitely irreducible, maximal CGDMS. If $t \in \Fin(\fs)$ then for all $x \in X$ and $n \in \N$:
\begin{equation}\label{eq:ftnunibd}
M_t^{-1}K^{-t}e^{nP(t)} \leq F_t^{(n)}(\one)(x) \leq M_t K^{t}e^{nP(t)},
\end{equation}
where $M_t$ is as in Proposition \ref{symbbound}. 
\end{prop}

\begin{proof}
 Let $x \in X$. Then $x \in X_v$ for some $v \in V$. Let $\tau \in E_A^\N$ such that $i(\tau)=i(\tau_1)=v$. Then, by Proposition \ref{symbbound}
 \begin{equation*}
 \begin{split}
 F_t^n(\one)(x)&=\sum_{\omega \in E_A^n} \norm{D \phi_\omega(x)}^t \chi_{X_{t(\omega)}}(x)=\sum_{\omega \in E_A^n:t(\omega)=v} \norm{D \phi_\omega(x)}^t \\
 &\overset{\eqref{bdp}}{\leq}K ^{t}\sum_{\omega \in E_A^n:A_{\om_n \tau_1}=1} \norm{D \phi_\omega(\pi(\tau))}^t=K^t \L^n_t(\one) (\tau) \overset{\eqref{ltn1bd}}{\leq} K^t M_t e^{n P(t)}.
 \end{split}
 \end{equation*}
 The lower bound follows by a similar argument.
\end{proof} 
By positivity of $F_t$, its iterates satisfy
\[
    \|F_t^n\|_{C(X)\to C(X)}
    =\|F_t^n\mathbf1\|_\infty.
\]
Indeed, $|F_t^ng|\leq F_t^n\mathbf1$ whenever
$\|g\|_\infty\leq1$, and equality of the operator norm
is attained at $g=\mathbf1$.
Proposition~\ref{transunifbd} therefore gives
\[
    M_t^{-1}K^{-t}e^{nP(t)}
    \leq \|F_t^n\|_{C(X)\to C(X)}
    \leq M_tK^t e^{nP(t)}.
\]
Taking $n$-th roots and applying Gelfand's formula yields
\[
    r(F_t)
    =\lim_{n\to\infty}
       \|F_t^n\|_{C(X)\to C(X)}^{1/n}
    =e^{P(t)}.
\]
In order to simplify notation we will also use the following normalized version of $F_t$. For $t \in \Fin(\fs)$ we let
\[ \tilde{F}_t(g)(x) := \lambda_t^{-1} F_t(g)(x) = \lambda_t^{-1}\sum_{e \in E} \norm{D\phi_e(x)}^t g(\phi_e(x)) \chi_{X_{t(e)}}(x),\]
where $\lambda_t = e^{P(t)}$ is the spectral radius of $F_t$. Recalling \eqref{perronit}, we obtain a formula for $\tilde{F}_t^n$ given by
\[ \tilde{F}_t^{n}(g)(x) := \lambda_t^{-n} F_t^{(n)}(g)(x) =\lambda_t^{-n}\sum_{\om \in E_A^n} \norm{D \phi_\om(x)}^t g(\phi_\om(x)) \chi_{X_{t(\om)}}(x).\]
We now recall the definition of almost periodicity.
\begin{defn}[Almost Periodicity]
Suppose that $L$ is a bounded operator on a Banach space $B$, with $L:B \rightarrow B$. Then $L$ is called \textit{almost periodic} if, for every $x \in B$, the orbit $\left( L^n(x)\right)_{n=0}^\infty $ is relatively compact in $B$.
\end{defn}
Arguing as in \cite[Lemma 6.1.1]{Mauldin_Urbanski_2003}, we obtain that:
\begin{prop}[$\tilde{F}_t$ is Almost-Periodic]
\label{prop:ae} Let $\fs$ be a finitely irreducible, maximal CGDMS and let $t \in \Fin(\fs)$. The operator
$ \tilde{F}_t :C(X) \ra C(X)$ is almost periodic.
\end{prop}

Using Propositions \ref{transunifbd} and \ref{prop:ae} and arguing as in \cite[Theorem 6.1.2]{Mauldin_Urbanski_2003}
we obtain the following theorem. 
\begin{thm}
\label{thm:Ftrhot}  
Let $\fs$ be a finitely irreducible, maximal CGDMS and let $t \in \Fin(\fs)$. There exists a unique continuous function $\rho_t: X \rightarrow [0,\infty)$, normalized by $\int_X\rho_t\,dm_t=1$, so that:
\begin{enumerate}
\item \label{eq:thm1exuni} 
$\tilde{F}_t \rho_t = \rho_t$,
and $\int \rho_t dm_t = 1$,
\item \label{eq:thm1bound} $K^{-t} M_t^{-1} \leq \rho_t \leq K^{t}M_t$,
\item \label{eq:thm1conv}  $\{\frac{1}{n}\sum_{k=1}^n\tilde{F}_t^k(\one) \}_{n=1}^\infty$ converges uniformly to $\rho_t$ on $X$,
\item \label{eq:rnrn} $\rho_t|_{J_\fs} = \frac{d \mu_t}{d m_t}$, $m_t$-almost everywhere; thus $\rho_t|_{J_\fs}$ is a continuous representative of the Radon--Nikodym derivative.
\end{enumerate}
\end{thm}

We will also need extensions of the eigenfunctions $\rho_t$ on neighborhoods of $X$. They will be used in Section \ref{sec:derbound} in order to show that the functions $\rho_t$ admit real analytic extensions on $S$, and for technical reasons they will also be useful in the implementation of our method in Section \ref{sec: numerical}. First we need to define an extension of the Perron-Frobenius operator in $S:=\cup_{v \in V} S_v$. We assume that the sets $S_v$ are disjoint. For $t \in \Fin (\cS)$ and $g \in C(S)$, we let
\begin{equation}\label{gt}
 G_t(g)(x) = \sum_{e \in E} \norm{D\phi_e(x)}^t g(\phi_e(x)) \chi_{S_{t(e)}}(x) .
\end{equation}

Similarly to $F_t$, $G_t: (C(S),\|\cdot\|_\infty) \to (C(S),\|\cdot\|_\infty)$ is a bounded linear operator. We can also obtain an analogue of Proposition \ref{transunifbd}.
\begin{prop}\label{transunifbd2}
Let $\fs$ be a finitely irreducible, maximal CGDMS and let $t \in \Fin(\fs)$. Then for all $x \in S$ and $n \in \N$:
\begin{equation}
\label{eq:ftnunibdex}M_t^{-1}K^{-2t}e^{nP(t)} \leq G_t^{(n)}(\one)(x) \leq K^{2t}e^{nP(t)}M_t.
\end{equation}
\end{prop}
\begin{proof} The proof follows from Proposition \ref{transunifbd} and the BDP, \eqref{bdp}. Let $y \in S_v$ and $x \in X_v$ for some $v \in V$. Then,
$$
G_t^{(n)}(\one)(y)= \sum_{\om \in E_A^n(v)} \norm{D \phi_\om(y)}^t  \overset{\eqref{bdp}}{\leq} K^{t} \sum_{\om \in E_A^n(v)} \norm{D \phi_\om(x)}^t=K^{t} F_t^n (\one) (x) \overset{\eqref{eq:ftnunibd}}{\leq}K^{2t}e^{n P(t)}M_t
$$
and similarly
\begin{equation*}
\begin{split}
G_t^{(n)}(\one)(y)&= \sum_{\om \in E_A^n(v)} \norm{D \phi_\om(y)}^t  \overset{\eqref{bdp}}{\geq} K^{-t} \sum_{\om \in E_A^n(v)} \norm{D \phi_\om(x)}^t\\
&=K^{-t} F_t^n (\one) (x) \overset{\eqref{eq:ftnunibd}}{\geq}K^{-2t} M_t^{-1}e^{n P(t)}.
\end{split}
\end{equation*}
The proof is complete
\end{proof}
We also consider the normalized operators
\[ \tilde{G}_t(g)(x) := \lambda_t^{-1} G_t(g)(x) = \lambda_t^{-1}\sum_{e \in E} \norm{D\phi_e(x)}^t g(\phi_e(x)) \chi_{S_{t(e)}}(x),\]
where $\lambda_t=e^{P(t)}$. We now establish almost
periodicity directly on $C(S)$.
\begin{prop}
\label{prop:aeG} Let $\fs$ be a finitely irreducible, maximal CGDMS. If $t \in \Fin(\fs)$  then the operator
$ \tilde{G}_t :C(S) \ra C(S)$ is almost periodic.
\end{prop}
\begin{proof}
Write $T=\widetilde G_t$ and $Q=M_tK^{2t}$.
By positivity and Proposition~\ref{transunifbd2},
\[
    \|T^m\|_{C(S)\to C(S)}
    =\|T^m\mathbf1\|_\infty
    \leq Q,
    \qquad m\geq1.
\]
Thus every orbit is uniformly bounded.
We prove that every orbit is also equicontinuous.

Let $s\in(0,1)$ be the common contraction bound from
Definition~\ref{definition: graph}, and set
$C=\max\{1,MK\}$, where $M$ is the constant in
\eqref{boundbyderlip}.
Since $\|D\phi_\omega\|_\infty\leq s^{|\omega|}$,
that estimate gives
\[
    |\phi_\omega(x)-\phi_\omega(y)|
    \leq Cs^{|\omega|}|x-y|,
    \qquad x,y\in S_{t(\omega)}.
\]

Condition~(iv) of Definition~\ref{gdmsdef}, applied
with the points in both orders, implies
\[
 \left|
 \log\frac{\|D\phi_e(p)\|}{\|D\phi_e(q)\|}
 \right|
 \leq L|p-q|^\alpha,
 \qquad p,q\in S_{t(e)}.
\]
Using the chain rule, invariance of the sets $S_v$,
and the preceding contraction estimate, we obtain,
for every $\omega\in E_A^m$ and
$x,y\in S_{t(\omega)}$,
\[
\begin{aligned}
 \left|
 \log\frac{\|D\phi_\omega(x)\|^t}
          {\|D\phi_\omega(y)\|^t}
 \right|
 &\leq tLC^\alpha
       \sum_{j=0}^{m-1}s^{j\alpha}|x-y|^\alpha\\
 &\leq D_t|x-y|^\alpha,
 \qquad
 D_t=\frac{tLC^\alpha}{1-s^\alpha}.
\end{aligned}
\]

Fix $g\in C(S)$, and define its modulus of continuity by
\[
 \omega_g(r)=
 \sup\{|g(p)-g(q)|:p,q\in S,\ |p-q|\leq r\}.
\]
Since $S$ is compact, $\omega_g(r)\to0$ as $r\to0$.
For $\omega\in E_A^m$, put
\[
    w_\omega(x)
    =e^{-mP(t)}\|D\phi_\omega(x)\|^t.
\]
The logarithmic estimate gives
\[
 |w_\omega(x)-w_\omega(y)|
 \leq w_\omega(y)
       \bigl(e^{D_t|x-y|^\alpha}-1\bigr).
\]
Therefore, for $x,y\in S_v$,
\[
\begin{aligned}
 |T^mg(x)-T^mg(y)|
 &\leq
 \sum_{\omega\in E_A^m(v)}
 |w_\omega(x)-w_\omega(y)|\,\|g\|_\infty\\
 &\quad+
 \sum_{\omega\in E_A^m(v)}
 w_\omega(y)\,
 |g(\phi_\omega(x))-g(\phi_\omega(y))|\\
 &\leq Q\left[
   \|g\|_\infty
       \bigl(e^{D_t|x-y|^\alpha}-1\bigr)
   +\omega_g(Cs^m|x-y|)
 \right]\\
 &\leq Q\left[
   \|g\|_\infty
       \bigl(e^{D_t|x-y|^\alpha}-1\bigr)
   +\omega_g(C|x-y|)
 \right].
\end{aligned}
\]
The last expression tends to zero as $|x-y|\to0$,
uniformly in $m\geq1$.

There are finitely many pairwise disjoint compact
components $S_v$, so this proves equicontinuity on $S$.
Including the single function $T^0g=g$ preserves
equicontinuity and uniform boundedness.
The Arzela--Ascoli theorem now shows that
$\{T^mg:m\geq0\}$ is relatively compact in $C(S)$.
Since $g$ was arbitrary, $T$ is almost periodic.
\end{proof}

We finish this section with a useful extension theorem.
\begin{thm}
\label{thm:Gtrhot}  
Let $\fs$ be a finitely irreducible, maximal CGDMS. If $t \in \Fin(\fs)$ then there exists a unique continuous function $\tilde{\rho}_t: S \rightarrow [0,\infty)$ so that:
\begin{enumerate}
\item \label{eq:thm1exuniext} 
$\tilde{G}_t \tilde{\rho}_t = \tilde{\rho}_t $  and $\int_X \tilde{\rho}_t dm_t = 1$,
\item \label{eq:thm1boundext} $M_t^{-1}K^{-2t} \leq \tilde{\rho}_t \leq M_t K^{2t}$,
\item \label{eq:thm1ext}$\tilde{\rho}_t|_{X}=\rho_t$, where $\rho_t$ is as in Theorem \ref{thm:Ftrhot},
\item \label{eq:thm1convext} $\{\frac{1}{n}\sum_{k=1}^n\tilde{G}_t^k(\one) \}_{n=1}^\infty$ converges uniformly to $\tilde{\rho}_t$ on $S$.
\end{enumerate}
\end{thm}
\begin{proof}
Write $T=\widetilde G_t$, $Q=M_tK^{2t}$, and define
\[
    A_N=\frac1N\sum_{k=1}^N T^k\mathbf1.
\]
The bounds in Proposition~\ref{transunifbd2} and the
equicontinuity established in the preceding proof imply
that $(A_N)$ is relatively compact in $C(S)$.
Moreover,
\[
    TA_N-A_N
    =\frac{T^{N+1}\mathbf1-T\mathbf1}{N}
    \longrightarrow0
    \quad\text{in }C(S).
\]
Consequently every subsequential limit is a fixed point
of $T$ and satisfies
$Q^{-1}\leq u\leq Q$.

Because the branches preserve $X$, restriction commutes
with the transfer operators:
\[
    (T^kg)|_X
    =\widetilde F_t^{\,k}(g|_X),
    \qquad g\in C(S).
\]
Theorem~\ref{thm:Ftrhot} therefore gives
\[
    A_N|_X\longrightarrow\rho_t
    \quad\text{uniformly on }X.
\]
Thus every subsequential limit extends $\rho_t$
and has integral $1$ against $m_t$.

We next prove uniqueness of a fixed-point extension.
Suppose that $h\in C(S)$ satisfies $Th=h$ and $h|_X=0$.
Let $C$ and $s$ be as in the preceding proof, and put
$D=\max_v\operatorname{diam}(S_v)$.
For $x\in S_v$, choose $y\in X_v$.
For every $\omega\in E_A^n(v)$,
$\phi_\omega(y)\in X$ and
\[
    |\phi_\omega(x)-\phi_\omega(y)|
    \leq Cs^nD.
\]
It follows that
\[
    |h(x)|=|T^nh(x)|
    \leq Q\,\omega_h(Cs^nD)\longrightarrow0.
\]
Hence $h=0$ on $S$.

All subsequential limits of $(A_N)$ consequently agree,
so relative compactness implies uniform convergence
of the full sequence to a function $\widetilde\rho_t$.
This proves existence, the stated bounds, and the
restriction identity.
Finally, any normalized nonnegative fixed point of $T$
restricts, by Theorem~\ref{thm:Ftrhot}, to $\rho_t$.
The uniqueness argument above then proves uniqueness
on $S$.
\end{proof}
\section{\texorpdfstring{Derivative bounds for $\rho_t$}{}}
\label{sec:derbound}
In this section we will prove derivative bounds for the eigenfunctions of the Perron-Frobenius operator $F_t$ on maximal CGDMSs. These bounds will play a crucial role in our numerical method. We stress that, as in Section \ref{sec:rn}, the open set condition is not needed for any of the results in this section.

We start by introducing some standard notation. A \textit{multi-index} $\alpha$ is an $n$-tuple of non-negative integers $\alpha_i$. The \textit{length} of $\alpha$ is
\[|\alpha| := \sum_{i=1}^n \alpha_i,\]
and we also denote
$$\alpha ! = \alpha_1! \cdot \alpha_2! \cdots \alpha_n!.$$
For a weakly $|\alpha|$-differentiable function $u$, we define the operator $D^\alpha$ by
\[D^\alpha u = \left( \frac{\partial }{\partial x_1}\right)^{\alpha_1} \cdots \left( \frac{\partial }{ \partial x_n}\right)^{\alpha_n} (u).\]

As in Section \ref{sec:rn},
$$
\cS= \big\{ V,E,A, t,i, \{X_v\}_{v \in V}, \{\f_e\}_{e\in E} \big\}
$$
will denote a maximal CGDMS and we will again assume that the sets $X_v$ are disjoint. Moreover, we will let
$$\eta_\cS=\min_{v \in V} \dist (X_v, \partial W_v).$$

\begin{thm} \label{thm:derivative_bound} 
Let $\fs = \{\phi_e\}_{e \in E}$ be a finitely irreducible, maximal CGDMS in $\R^n, n \geq 2$. Let $t \in \text{Fin}(\fs)$, let  $\rho_t$ be as in Theorem \ref{thm:Ftrhot}, and let $\alpha$ be any multi-index. 
\begin{enumerate}
\item \label{realanal} The eigenfunctions $\rho_t$ admit real analytic extensions on $\Int (S)=\cup_{v\in V}\Int (S_v)$.
\item \label{eq:dern>2} If $\fs$ consists of M\"obius maps then for any $u,s$ such that  $0<u<s<\sqrt{2}-1$,
\begin{equation}\label{eq: derivative Drho}
        |D^\alpha \rho_t (x)| \leq \alpha !  \left( \frac{ n^{1/2}}{u\,\eta_\cS}\right)^{|\alpha|} c(s)^t \rho_t(x), \quad \forall x \in X,
    \end{equation}
    where   $c(s)=(1-s(2+s))^{-1}$.
\item \label{eq:dern=2}  If $n=2$, then 
\begin{equation}
\label{eq: derivative Drhoanaly}
        |D^\alpha \rho_t (x)| \leq \alpha !  \left(\frac{ML}{s\, \eta_\cS}\right)^{|\alpha|} \exp \left(t C_r \left( \frac{L}{L-2} \right)^2 \right) \rho_t(x), \quad \forall x \in X, 
    \end{equation}
    where $r,s, M, L$ can be any numbers such that $r \in (0,1), s \in (0,r), M>1, L>2$ and $$C_r=\log \left( \frac{(1+r)^3}{(1-r)^5}\right).$$
\end{enumerate}
\end{thm}

\begin{proof}
We first assume that the system consists of M\"obius maps. We denote translation by $a\in\R^n$ by $\tau_a(x)=x+a$. For every nonempty admissible word $\omega$, the map $\phi_\omega$ has the form
\begin{equation*}
    \phi_\omega = \tau_{b_\om}  \circ \lambda_{\omega} L_\omega \circ \iota^{\varepsilon_\om} \circ \tau_{-a_\om},
\end{equation*}
where $a_\omega, b_\omega \in\R^n$,$\lambda_\omega>0$, $L_\omega$ is an orthogonal transformation,  $\varepsilon_\om \in \{0,1\}$,  $\iota^{0}=\text{Id}$ and
$$\iota^1 (z)=\iota (z)=\begin{cases} \frac{1}{z},& z \in \C \\
\frac{z}{|z|^2}, &z\in \R^n, n \geq 3. \end{cases}$$ 
Thus,
\[\norm{D \phi_\omega  (z)} =
\begin{cases}
\frac{\lambda_\omega }{|z-a_\om|^2} &\text{ if } \varepsilon_\omega=1,  \\
\lambda_\omega &\text{ if } \varepsilon_\omega = 0 \quad (\mbox{ i.e. } \iota^{\varepsilon_\omega}=\text{Id}).
\end{cases}\]
When $\iota^{\varepsilon_\om}$ is not the identity we have that $a_\omega \not\in W_{t(\om)}$. 

We will first prove statement \eqref{eq:dern>2}. We fix $v \in V$ and $x \in X_v$. For any $\omega \in E_A^*(v)$ we define a function $\rho_\omega :\C^n \rightarrow \C_\infty:=\C \cup \{\infty\}$ given by 
\[
\rho_\omega (z) = \begin{cases}
\frac{|x - a_\omega|^2}{\sum_{j=1}^n\left( z_j - (a_\omega)_j\right)^2} &\text{ if } \varepsilon_\omega=1\\
1, &\text{ if }  \varepsilon_\omega=0.
\end{cases}
\]

For simplicity of notation we let $\eta:=\eta_\cS$. Let $0<u<s<\sqrt{2}-1$ and set $r=s \eta$.
For every \(\omega\in E_A^\ast(v)\), we first show that
\begin{equation}
\label{eq:rhobound}
|\rho_\omega (z)| \leq c(s), \text{ for all } z \in B_{\C^n}(x, r):=\{z \in \C^n: \norm{z-x}<r\}, 
\end{equation}
where $\norm{\cdot}$ denotes the Euclidean norm in $\C^n$. Note that if $\rho_\omega(z)=1$, we have nothing to prove. Therefore we may assume that 
$$ \rho_\omega(z) = \frac{|x - a_\omega|^2}{\sum_{j=1}^n\left( z_j - (a_\omega)_j\right)^2} .$$
Let $z \in B_{\C^n}(x, r)$. Then:
\begin{align*}
    \sum_{j=1}^n (z_j - (a_\omega)_j)^2 &= \sum_{j=1}^n (z_j-x_j+x_j - (a_\omega)_j)^2 
    = \sum_{j=1}^n (z_j-x_j)^2 + \sum_{j=1}^n (x_j-(a_\omega)_j)^2 \\
    &\quad+ 2 \sum_{j=1 }^n (z_j-x_j) ( x_j-(a_\omega)_j) ,
    \end{align*}
    and consequently
    \begin{equation}
    \label{eq:splitsum1}
    \begin{split}
 \norm{ \sum_{j=1}^n (z_j - (a_\omega)_j)^2 }
    & \geq  \norm{\sum_{j=1}^n (x_j-(a_\omega)_j)^2 } - \sum_{j=1}^n |z_j-x_j|^2 - 2 \sum_{j=1 }^n |z_j-x_j| | x_j-(a_\omega)_j| \\
    &\overset{x \in \R^n}{=} |x - a_\omega|^2 - \norm{z-x}^2 - 2 \sum_{j=1 }^n |z_j-x_j| | x_j-(a_\omega)_j|.
    \end{split}
\end{equation}
Since $z \in B_{\C^n}(x, s \eta)$ and $a_\om \notin W_{t(\om)}$
\begin{equation}
\label{|z-y|bd}
\frac{\|z-x\|}{|x-a_\om|}\leq \frac{s \eta}{\eta}=s.
\end{equation}
Using the Cauchy-Schwarz inequality,
\begin{equation}
\begin{split}
\label{eq:sumbound1}\sum_{j=1 }^n |z_j-x_j| | x_j-(a_\omega)_j| &\leq \left(\sum_{j=1 }^n |z_j-x_j|^2\right)^{1/2}\left( \sum_{j=1 }^n|x_j-(a_\omega)_j|^2\right)^{1/2} \\
&=\|z-x\| |x-a_{\om}| \overset{\eqref{|z-y|bd}}{\leq} s |x-a_\om|^2.
\end{split}
\end{equation}
Thus,
\begin{equation}
\begin{split}
\label{eq:sumlbound}
    \norm{\sum_{j=1}^n (z_j-(a_\omega)_j)^2 }&\overset{\eqref{eq:splitsum1} \wedge \eqref{|z-y|bd} \wedge \eqref{eq:sumbound1}}{\geq} |x-a_\omega|^2-s^2|x-a_\omega|^2 -2 s |x  - a_\omega|^2 \\
    &= (1-s(2+s))|x - a_\omega|^2.
\end{split}
\end{equation}
Therefore,
\begin{align*}
    |\rho_\omega(z)| = \frac{|x - a_\omega |^2}{\norm{\sum_{j=1}^n\left( z_j - (a_\omega)_j\right)^2}} \overset{\eqref{eq:sumlbound}}{\leq} (1-s(2+s))^{-1}= c(s).
\end{align*}
Since $B_{\C^n}(x, r)$ is simply connected, the analytic function
$$z \longmapsto \rho_\omega(z), \quad z \in B_{\C^n}(x, r),$$
has an analytic logarithm, since it is nonvanishing; see e.g. \cite[Lemma 6.1.10]{Krantz}. We choose the logarithm $\ell_\omega$ normalized by $\ell_\omega(x)=0$ and define $\rho_\omega^t=\exp(t\ell_\omega)$. On the real ball this agrees with the positive real power. Thus,
$$ z \longmapsto \rho_\omega(z)^t$$
is analytic for $z\in B_{\C^n}(x, r)$. We then let
\begin{equation*}
\begin{split}b_k(z) =\frac{1}{k}\sum_{m=1}^k \sum_{\om \in E_A^m(v)}e^{-P(t)m} \rho_\omega(z)^t \norm{ D\phi_\omega (x)}^t 
:= \frac{1}{k}\sum_{m=1}^k \beta_m(z)
\end{split}
\end{equation*}
for $z \in B_{\C^n}(x, r)$.
Using Proposition \ref{transunifbd} we see that for all $m \in \N$ and $z \in B_{\C^n}(x, r)$,
\begin{equation}
\label{eq:bnyzbound}
\begin{split}
    |\beta_m(z)| &\leq e^{-P(t)m}\sum_{\om \in E_A^m(v)} |\rho_\omega(z)|^t \norm{ D\phi_\omega (x)}^t \overset{\eqref{eq:rhobound}}{\leq} e^{-P(t)m} c(s)^t \sum_{\om \in E_A^m(v)} \norm{ D\phi_\omega (x)}^t\\
    & \leq c(s)^t e^{-P(t)m}F_t^m(\one) (x) \overset{\eqref{eq:ftnunibd}}{\leq} c(s)^t K^tM_t.
    \end{split}
\end{equation}
For each fixed $m$, the summands defining $\beta_m$ are bounded in absolute
value by the summable majorant
$e^{-mP(t)}c(s)^t\|D\phi_\omega(x)\|^t$.
The Weierstrass theorem therefore shows that $\beta_m$, and hence every $b_k$,
is holomorphic on $B_{\C^n}(x,r)$. Moreover,
\[
 |b_k(z)|\leq c(s)^t\frac1k\sum_{m=1}^k e^{-mP(t)}F_t^m(\one)(x).
\]
In particular the sequence is locally uniformly bounded.
Let $\tilde{s} \in (u,s)$ and set
$$\tilde{r}=\tilde{s} \eta.$$
Montel's theorem (see e.g. \cite[Proposition 6]{nara}) gives a subsequence $(b_{m_k})_{k=1}^\infty$ converging locally uniformly on $B_{\C^n}(x,r)$ to a holomorphic function $b$. In particular,
\begin{equation}
\label{eq:unifcovntilde}
b_{m_k}  \to b \mbox{ uniformly on }B_{\C^n}(x, \tilde{r}).
\end{equation}
Therefore, Theorem \ref{thm:Ftrhot} \eqref{eq:thm1conv}, \eqref{eq:bnyzbound} and \eqref{eq:unifcovntilde} imply that
\begin{equation}
\label{eq:bybdxi}
|b(z)| \leq c(s)^t \rho_t(x) \mbox{ for all }z\in  B_{\C^n}(x, \tilde{r}).
\end{equation}
For real $y\in B_{\R^n}(x,r)$, both the inversion case and the similarity
case give
\[
 \rho_\omega(y)^t\|D\phi_\omega(x)\|^t=\|D\phi_\omega(y)\|^t.
\]
Consequently, for $y\in B_{\R^n}(x,r)\cap X_v$,
\[
 b_m(y)=\frac1m\sum_{k=1}^m e^{-kP(t)}F_t^k(\one)(y).
\]
Thus, combining Theorem \ref{thm:Ftrhot} \eqref{eq:thm1conv} and \eqref{eq:unifcovntilde} we deduce that 
\begin{equation}
\label{eq:byrhot}
b =\rho_t \mbox{ in }X_v \cap B_{\C^n}(x, \tilde{r}).
\end{equation}
Recall that the \textit{polydisk metric in $\C^n$} is defined as
$$ \normp{z-w} = \max \left\{ |z_i-w_i|:i = 1 , \ldots n \right\},\quad z,w \in \C^n.$$
A \textit{polydisk in $\C^n$} is a set of the form
$$P(z,r) := \left\{ w \in \C^n : \normp{w-z} <r\right\}, \quad\mbox{ where }z\in \C^n, r>0.$$
It is easy to check that
\begin{equation}\label{eq:polycomp}
    \normp{z-w} \leq \norm{z-w} \leq \sqrt{n} \normp{z-w} .
\end{equation}
Therefore,
$$\overline{P}\left(x,\frac{1}{\sqrt{n}} u \eta \right) \subset \overline{B}_{\C^n} (x,u \eta).$$
Recall that $b$ is holomorphic in $B_{\C^n} (x,\tilde{r})$ which is an open neighborhood of $\overline{B}_{\C^n} (x,u \eta)$. Therefore, if $\alpha$ is any multi-index, applying the Cauchy estimates (see e.g.  \cite[Chapter 1, Proposition 3]{nara}), we see that
\begin{align}\label{derhot}
 |D^\alpha \rho_t (x) | \overset{\eqref{eq:byrhot}}{=}   |D^\alpha b(x)| \leq \alpha !\left( \frac{n^{1/2}}{u \eta}\right)^{|\alpha|} \max_{z \in \partial P(x,\frac{u \eta}{\sqrt{n}}) } |b(z)| \overset{\eqref{eq:bybdxi}}{\leq} \alpha ! \left( \frac{n^{1/2}}{u \eta}\right)^{|\alpha|} c(s)^{t} \rho_t(x) .
\end{align}
Since $v \in V$ and $x \in X_v$ were arbitrary, the proof of statement \ref{eq:dern>2} is complete.

We will now prove statement \ref{eq:dern=2}. We fix $v \in V$ and we define 
\begin{equation}\label{eq:2db}
    b_k(z) = \frac{1}{k}\sum_{n=1}^k e^{-nP(t)}\sum_{\om \in E_A^n(v)} \norm{D\phi_\om (z)}^t 
\end{equation}
for $z \in W_v$ and $k \in \N$. Note that for $z \in X_v$,
\begin{equation}\label{eq:2db2}
    b_k(z)= 
\frac{1}{k}\sum_{n=1}^ke^{-nP(t)} F_t^n(\one)(z).
\end{equation}
Let $\om \in E_A^{\ast}$. Recall that since the maps $\phi_\om$ are conformal we have that either $\norm{D \phi_\om(z) } = |\phi_\om'(z)|$ (when $\phi_\om$ is holomorphic) or $\norm{D \phi_\om(z) } = |(\overline {\phi_\om})'(z)|$ (when $\phi_\om$ is antiholomorphic). By Proposition \ref{transunifbd} 
\begin{equation}\label{eq:2dbnd}
    b_n(z) \overset{\eqref{eq:ftnunibd}}{\leq} K^t M_t, \quad \mbox{ for all }z \in X_v\mbox{ and }n \in \N.
\end{equation}
For $\om \in E_A^\ast$, define
\[\psi_\om =
\begin{cases}
\phi_\om, \text{ if } \phi_\om \text{ is holomorphic} \\
\bar \phi_\om, \text{ if } \phi_\om \text{ is anti-holomorphic.}
\end{cases}\]
Thus $\norm{D \phi_\om(z) } = |\psi_\om'(z)|$. Fix some $\zeta_v \in X_v$ and, without loss of generality, assume that $\zeta_v = 0.$ Given any $\om \in  E_A^\ast (v)$, define 
\[ \rho_\om(z) = \frac{\psi_\om'(z)}{\psi_\om'(0)}, \quad z \in W_v.
\]
To simplify notation we again let $\eta:=\eta_\cS$.
Since $B(0,\eta)$ is simply connected and $\rho_\om$ is nonvanishing and
holomorphic, it has a holomorphic logarithm. Choose the branch satisfying
$\log\rho_\om(0)=0$. Koebe's distortion and rotation estimates
\cite[Theorem 23.1.6]{MRUvol3} give
$$\frac{1-r}{(1+r)^3}\leq|\rho_\om(z)|\leq\frac{1+r}{(1-r)^3}$$
and 
$$ |\arg \rho_\om(z)| \leq 2 \log \left( \frac{1+r}{1-r}\right) $$
on $ \bar B (0,r\eta)$ for $r \in (0,1)$. Therefore $\log \rho_\om = \log |\rho_\om| + i \arg \rho_\om$ is an analytic logarithm for $\rho_\om$ and
\begin{equation}
    |\log \rho_\om (z)| \leq \log \left( \frac{1+r}{(1-r)^3}\right) + 2 \log \left( \frac{1+r}{1-r}\right) := C_r.
\end{equation}
for $z\in \bar B (0,r\eta)$ and $r \in (0,1)$. Therefore we can write $\log \rho_\om$ as a power series
$$ \log \rho_\om = \sum_{m=0}^\infty a_m z^m \text{ in } B(0,r\eta),$$
and by Cauchy estimates we can see that for all $0<s \leq r,$
\begin{equation}\label{eq:ambd}
    |a_m| \leq \frac{C_r}{s^m \eta^m}.
\end{equation}
For real $x,y$ with $|x|+|y|<r\eta$, absolute convergence justifies the rearrangement
\begin{align*}
    \Re ( \log \rho_\om (z)) &= \Re \left( \sum_{m=0}^\infty a_m (x+iy)^m \right) \\
    &= \Re \left( \sum_{m=0}^\infty a_m \sum_{k=0}^m \binom{m}{k}x^k (iy)^{m-k}\right) \\
    & = \Re \left( \sum_{n=0}^\infty \sum_{k=0}^\infty a_{n+k}\binom{n+k}{k}i^nx^k y^n\right) \\
    &= \sum_{n=0}^\infty \sum_{k=0}^\infty \Re \left( a_{n+k}\binom{n+k}{k}i^n \right)x^k y^n \\
    &:= \sum_{n=0}^\infty \sum_{k=0}^\infty c_{k,n} x^k y^n .
\end{align*}
Thus for all $s \leq r$
\begin{equation}
    |c_{k,n}| \leq |a_{n+k}| \binom{n+k}{k} \leq |a_{n+k}|2^{n+k} \overset{\eqref{eq:ambd}}{\leq} \frac{C_r 2^{n+k}}{(s \eta)^{n+k}}.
\end{equation}
For $L>2$, consider the power series
\[
    F_\omega(z,w)
    =\sum_{n,k=0}^{\infty}c_{k,n}z^kw^n,
    \qquad (z,w)\in P(0,s\eta/L).
\]
The coefficient bounds imply normal convergence on this
polydisk, so $F_\omega$ is holomorphic there.
Indeed, for $(z,w)\in P(0,s\eta/L)$,
\begin{equation}
\begin{split}
\label{eq:Fzwbound}
    |F_\omega(z,w)| &\leq \sum_{k,n=0}^\infty |c_{k,n}||z|^k |w|^n \\
    &\leq \sum_{k,n=0}^\infty \frac{C_r 2^{n+k}}{(s\eta)^{n+k}} \frac{s^{n+k}}{L^{n+k}} \eta^{n+k}\\
    & = C_r \sum_{k,n=0}^\infty \left( \frac{2}{L}\right)^{n+k} = C_r \left( \sum_{k=0}^\infty \left( \frac{2}{L}\right)^k\right)^2 \\
    &=  C_r \left(\frac{L}{L-2}\right)^2 := C_1(r,L).
\end{split}
\end{equation}
 In the following we will use the embedding $\iota : \C \rightarrow \C^2$,
$$ \iota(x+iy) = (x+i0,y+i0)$$
for all $x,y \in \R$. To simplify notation, we let
$$A = \iota (A) \text{ if } A \subset \C.$$
Note also that $B(0,r)=\iota(B(0,r))\subset P(0,r)$. 
Hence, 
\begin{equation}
\label{eq:f=reps}
F_\omega = \Re (\log \rho_\om) \mbox{ on }B(0,s\eta/L).
\end{equation}
Let
$$ B_k(z,w) = \frac{1}{k}\sum_{n=1}^k e^{-nP(t)} \sum_{\om \in E_A^n(v)} \norm{D \phi_\om(0)}^t e^{tF_\omega(z,w)}, \quad 
\quad (z,w)\in P(0,s\eta/L),\ k\in\N.$$
For $(x,y) \equiv x+iy=\zeta \in B\left(0, \frac{s \eta }{L}\right)$
\begin{equation}
\label{eq:Bn=bn}
\begin{split}
    B_k(\zeta) = B_k(x,y) &\overset{\eqref{eq:f=reps}}{=} \frac{1}{k}\sum_{n=1}^k e^{-nP(t)} \sum_{\om \in E_A^n(v)} \norm{D \phi_\om(0)}^t e^{t\Re (\log \rho_\om(\zeta))}\\
    &= \frac{1}{k}\sum_{n=1}^k e^{-nP(t)} \sum_{\om \in E_A^n(v)} \norm{D \phi_\om(0)}^t e^{\log |\rho_\om(\zeta)|^t}\\
    &= \frac{1}{k}\sum_{n=1}^k e^{-nP(t)} \sum_{\om \in E_A^n(v)} \norm{D \phi_\om(0)}^t\left| \frac{\psi_\om' (\zeta)}{\psi_\om'(0)} \right|^t\\
    &= \frac{1}{k}\sum_{n=1}^k e^{-nP(t)} \sum_{\om \in E_A^n(v)} \norm{D \phi_\om(\zeta)}^t 
    = b_k(\zeta).
    \end{split}
\end{equation}
Now note that for all $(z,w) \in P(0,s\eta/L)$ 
\begin{equation}
\begin{split}
\label{eq:Bnbound}
    |B_k(z,w)| &= \left| \frac{1}{k}\sum_{n=1}^k e^{-nP(t)} \sum_{\om \in E_A^n(v)} \norm{D \phi_\om(0)}^t e^{tF_\omega(z,w)} \right| \\
    &\leq \frac{1}{k}\sum_{n=1}^k e^{-nP(t)} \sum_{\om \in E_A^n(v)} \norm{D \phi_\om(0)}^t e^{\Re (tF_\omega(z,w))}  \\
    &  \leq \frac{1}{k}\sum_{n=1}^k e^{-nP(t)} \sum_{\om \in E_A^n(v)} \norm{D \phi_\om(0)}^t e^{t|F_\omega(z,w)|} \\
    &\overset{\eqref{eq:Fzwbound}}{\leq} e^{tC_1(r,L)}\frac{1}{k}\sum_{n=1}^k e^{-nP(t)}  \sum_{\om \in E_A^n(v)} \norm{D \phi_\om(0)}^t\\
    &= e^{tC_1(r,L)} b_k(0).
\end{split}
\end{equation}
Thus,
$$|B_n(z,w)| \overset{\eqref{eq:Bnbound} \wedge \eqref{eq:2dbnd}}{\leq} K^t M_t e^{tC_1(r,L)} \quad \mbox{ for }(z,w) \in P(0,s\eta/L).$$
For each fixed word length, the holomorphic summands defining $B_k$
have the summable majorants
$e^{-mP(t)}\|D\phi_\omega(0)\|^t e^{tC_1(r,L)}$.
Normal convergence therefore implies that every $B_k$ is holomorphic on
$P(0,s\eta/L)$. The preceding uniform bound and Montel's theorem give a
subsequence converging locally uniformly there to a holomorphic function
$B$, and hence uniformly on $\bar P(0,s\eta/(ML))$ for every $M>1$. Thus, Theorem \ref{thm:Ftrhot} \eqref{eq:thm1conv} and \eqref{eq:Bn=bn}  imply that
\begin{equation}
\label{eq:brhot} 
B = \rho_t \mbox{ on }B\left(0,\frac{s\eta}{ML}\right) \cap X_v.
\end{equation}
Moreover, Theorem \ref{thm:Ftrhot} \eqref{eq:thm1conv} and \eqref{eq:Bnbound} imply that
\begin{equation}
\label{eq:Bbound}
    |B(z,w)| \leq e^{tC_1(r,L)}\rho_t(0) \mbox{ for all }(z,w) \in \bar P\left(0,\frac{s\eta}{ML}\right).
\end{equation}
By the Cauchy Estimates, if $\alpha$ is any multiindex,
\begin{align*}
    |D^\alpha \rho_t(0)| &\overset{\eqref{eq:brhot}}{=} |D^\alpha B(0)| \leq \frac{\alpha!}{\left( \frac{s\eta}{ML}\right)^{|\alpha|}} \max_{(z,w) \in \partial P\left(0,\frac{s\eta}{ML}\right)} |B(z,w)| \\
    &\overset{\eqref{eq:Bbound}}{\leq} \frac{\alpha!}{(s\eta)^{|\alpha|}} (ML)^{|\alpha|}e^{tC_1(r,L)} \rho_t(0) = \frac{\alpha!}{(s\eta)^{|\alpha|}} (ML)^{|\alpha|}e^{tC_r\left(\frac{L}{L-2}\right)^2} \rho_t(0) .
\end{align*}
The proof of \eqref{eq:dern=2} is complete.

We now prove \eqref{realanal}. Set
\[
 \tilde\eta=\min_{v\in V}\dist(S_v,\partial W_v)>0.
\]
Fix $v\in V$ and $x\in S_v$, and consider the normalized Cesaro averages
\[
 a_k(y)=\frac1k\sum_{m=1}^k e^{-mP(t)}G_t^m(\one)(y),
 \qquad y\in S_v.
\]
By Theorem~\ref{thm:Gtrhot}\,\eqref{eq:thm1convext}, these converge uniformly
on $S_v$ to $\tilde\rho_t$. Repeat the local complexification above with
basepoint $x$, with $\tilde\eta$ in place of $\eta$, and with $G_t$ in place
of $F_t$. Proposition~\ref{transunifbd2} supplies the uniform majorants
needed for normal convergence of the defining series and for Montel's theorem.
For M\"obius maps choose $1/4<\tilde s<s<\sqrt2-1$. In the planar case choose
$s<r<1$, $L>2$, and $M>1$ so that $s/(ML)>1/4$.
Consequently there is a holomorphic function $\tilde R_x$ on
$B_{\mathbb C^n}(x,\tilde\eta/4)$ such that
\[
 \tilde R_x(y)=\tilde\rho_t(y)
 \qquad\text{for every }y\in S_v\cap B_{\mathbb R^n}(x,\tilde\eta/4).
\]
Here the identity follows from the convergence of the averages on $S_v$,
not merely on $X_v$. If $x\in\Int(S_v)$, it identifies $\tilde\rho_t$
with a real analytic function on a real neighborhood of $x$.
Thus $\tilde\rho_t$ is real analytic on $\Int(S)$, and
Theorem~\ref{thm:Gtrhot}\,\eqref{eq:thm1ext} shows that it extends $\rho_t$.
This proves \eqref{realanal} and completes the proof.
\end{proof}
We conclude this section with some remarks.
\begin{remark}\label{rem: extnended}
Using Proposition \ref{transunifbd2}, Theorem \ref{thm:Gtrhot} and replicating the proofs of \eqref{eq:dern>2} and \eqref{eq:dern=2} we obtain derivative bounds for the extensions $\tilde{\rho}_t$ of the eigenfunctions $\rho_t$:
\begin{enumerate}
\item \label{eq:dern>2ex} If $\fs$ consists of M\"obius maps then: 
\begin{equation}\label{eq: derivative Drhoex}
        |D^\alpha \tilde{\rho}_t (x)| \leq \alpha !  \left( \frac{ n^{1/2}}{u\,\tilde{\eta}}\right)^{|\alpha|} c(s)^t \tilde{\rho}_t(x), \quad \forall x \in S,
    \end{equation}
    where  $0<u<s<\sqrt{2}-1$ and $c(s)=(1-s(2+s))^{-1}$.
\item \label{eq:dern=2ex}  If $n=2$, then 
\begin{equation}
        |D^\alpha \tilde{\rho}_t (x)| \leq \alpha !  \left(\frac{ML}{s\, \tilde{\eta}}\right)^{|\alpha|} \exp \left(t \tilde{C}_r \left( \frac{L}{L-2} \right)^2 \right) \tilde{\rho}_t(x), \quad \forall x \in S, 
    \end{equation}
    where $r,s, M, L$ can be any numbers such that $r \in (0,1), s \in (0,r), M>1, L>2$ and $$\tilde{C}_r=\log \left( \frac{(1+r)^3}{(1-r)^5}\right).$$
\end{enumerate}
\end{remark}

\begin{cor}[First and second derivative bounds for M\"obius systems]
\label{cor:mobius-first-second}
Under the hypotheses of Theorem~3.1, suppose that the system
consists of M\"obius maps and that $t\ge 0$. Then
\[
|\nabla\rho_t(x)|
\le \frac{2t}{\eta_{\mathcal S}}\rho_t(x),
\qquad
\|D^2\rho_t(x)\|_{\mathrm{op}}
\le
\frac{2t(2t+1)}{\eta_{\mathcal S}^{\,2}}\rho_t(x),
\qquad x\in X.
\]
Moreover, for the positive extension $\widetilde\rho_t$ considered
in Remark~3.2, these estimates hold on $S$ with
$\eta_{\mathcal S}$ replaced by
\[
\widetilde\eta
=
\min_{v\in V}\operatorname{dist}(S_v,\partial W_v).
\]
\end{cor}

\begin{proof}
For every nonempty admissible word $\omega$, the M\"obius
representation used in the proof of Theorem~3.1 shows that
\[
w_\omega(x):=\|D\phi_\omega(x)\|^t
\]
is either constant or has the form
\[
w_\omega(x)
=
\lambda_\omega^t|x-a_\omega|^{-2t},
\qquad
a_\omega\notin W_{t(\omega)}.
\]
In the nonconstant case, write
$r=|x-a_\omega|$ and $u=(x-a_\omega)/r$. Then
\[
\nabla w_\omega(x)
=
-\frac{2t}{r}u\,w_\omega(x)
\]
and
\[
D^2w_\omega(x)
=
\frac{w_\omega(x)}{r^2}
\bigl(4t(t+1)uu^\top-2tI\bigr).
\]
The matrix in parentheses has eigenvalue $2t(2t+1)$
in the direction $u$ and eigenvalue $-2t$ in every
orthogonal direction. Hence, since $t\ge0$,
\[
|\nabla w_\omega(x)|
\le \frac{2t}{r}w_\omega(x),
\qquad
\|D^2w_\omega(x)\|_{\mathrm{op}}
\le \frac{2t(2t+1)}{r^2}w_\omega(x).
\]
The same inequalities hold trivially for constant weights.

For $x\in X_{t(\omega)}$, we have
$r\ge\eta_{\mathcal S}$. The resulting inequalities are
preserved under positive sums and the normalized Ces\`aro
averages used in the proof of Theorem~3.1. The locally
uniform holomorphic convergence established there implies
local convergence of their derivatives, allowing passage
to the eigenfunction limit.

For the extension to $S$, use the corresponding averages
from Remark~3.2 and the inequality $r\ge\widetilde\eta$.
\end{proof}

\begin{remark} It is straightforward to check that Theorem \ref{thm:derivative_bound} \eqref{eq:dern>2} also holds if $\cS$ consists of \textit{extended M\"obius maps} in $\C$. Recall that a map $f: \C_\infty \ra \C_\infty$ is an extended M\"obius map if $f$ or $\bar{f}$ belongs to the M\"obius group.
\end{remark}

\begin{remark} Although the constants in \eqref{eq: derivative Drho} and \eqref{eq: derivative Drhoanaly} are easily computable, in practice they can be quite large. It might be possible to obtain derivative bounds with better constants following a less universal approach and leveraging more the specifics of each system. In particular, Falk and Nussbaum in \cite{FalkRS_NussbaumRD_2016b}, obtained much better constants for second order derivative bounds in the case of complex continued fractions systems. In \cite{BCLW}, similar sharper bounds were obtained  for $n$-dimensional real continued fractions for any $n\geq 2$ .
\end{remark}
The next result 
allows us to compare eigenfunction values at different points.
\begin{prop}\label{prop: comparing values}
Let \(H\subset\mathbb R^n\) be convex, and let \(v\) be a positive \(C^1\) function on an open neighborhood of \(H\). If
\begin{equation}\label{eq: prop assumption for derivative}
|\nabla v(z)|\leq C_0v(z),\qquad z\in H,
\end{equation}
then
\[
v(y)\leq e^{C_0|y-x|}v(x),\qquad x,y\in H.
\]
\end{prop}
\begin{proof}
    For any $x,y\in H$ define 
    $$
    \phi(s) = (1-s)x+sy,\quad s\in[0,1].
    $$
    Since $\phi(0)=x$ and $\phi(1)=y$, we have
    $$
    \begin{aligned}
  \log \frac{v(y)}{v(x)} &= \log \frac{v(\phi(1))}{v(\phi(0))}=\log{v(\phi(1))}-\log{v(\phi(0))}\\
  &=\int_0^1 \frac{d}{ds} \left( \log{v(\phi(s))}\right)ds
   =\int_0^1 \frac{\nabla v(\phi(s))}{v(\phi(s))}\cdot \phi'(s) ds
  =\int_0^1 \frac{\nabla v(\phi(s))}{v(\phi(s))}\cdot(y-x) ds.
    \end{aligned}
    $$
    Thus, using \eqref{eq: prop assumption for derivative},
    $$
 \left|\log \frac{v(y)}{v(x)}\right|  \le  \int_0^1 \frac{|\nabla v(\phi(s))|}{v(\phi(s))} ds|y-x|\le C_0|y-x|.
    $$
    Exponentiating, we obtain the result. 
\end{proof}

\section{Numerical method}\label{sec: numerical}

In this section, we describe a finite element framework for
obtaining rigorous Hausdorff dimension bounds for finitely
irreducible maximal conformal GDMSs $\mathcal S$, subject to
the quantitative extension, mesh compatibility, and, for
infinite alphabets, tail estimates specified below.
The method is based on the Falk-Nussbaum approach of approximating the eigenfunctions of the Perron-Frobenius operator \cite{Falk}, and consists of the following steps:

\begin{itemize}
\item Discretizing $C(X)$.
\item Approximating the Perron-Frobenius operator.
\item Computing upper and lower bounds for the Hausdorff dimension of the limit set.
\end{itemize}

Before we describe the method, we introduce some notation and supplementary results. 

\subsection{Notation and linear Lagrange  finite elements}
Our numerical estimates apply results from finite element methods.
Suppose we are working on an open, bounded \textit{domain} $\Omega$ in $\R^n$. Throughout the  paper, we will
use the usual notation for the Lebesgue $L^\infty(\Omega)$ and  Sobolev $W^{2,\infty}(\Omega)$ spaces with the corresponding $L^\infty(\Omega)$ norm and $W^{2,\infty}(\Omega)$ semi-norm:
$$
\norm{v}_{L^{\infty}(\Omega)}=\operatorname{esssup}_{x\in\Omega}|v(x)|
$$
and 
\[|v|_{W^{2,\infty}(\Omega)} =\max_{1\le i,j\le n}\norm{\partial^2_{ij} v}_{L^{\infty}(\Omega)}.
\]

 \subsection{\texorpdfstring{Discretizing $C(X)$}{}}\label{sec: discretizing CX}

We use the enlarged sets \(S_v\) and conformal neighborhoods \(W_v\) introduced in Section \ref{sec:prelim}. For each vertex \(v\in V\), choose a compact polyhedral computational domain \(X_v^h\) satisfying
\[
X_v\subset X_v^h\subset S_v\subset W_v,
\]
and set
\[
X^h=\bigsqcup_{v\in V}X_v^h.
\]
We assume that these computational domains satisfy
\[
\phi_e(X_{t(e)}^h)\subset X_{i(e)}^h
\qquad\text{for every }e\in E.
\]
On \(X^h\), we use the positive eigenfunction extension \(\widetilde\rho_t\) constructed in Section \ref{sec:prelim}.

We partition (triangulate) $X^h$ into  simplices, i.e. $X^h=\cup_{\tau}\bar{\tau}$. Let \(\mathcal T_h\) denote this triangulation.
For simplicity we choose a conforming mesh, 
meaning 
the intersection of any two distinct closed simplices is either empty or a common face of both simplices (face, edge, or node).
An example of a two-dimensional conforming triangulation is shown in Figure \ref{fig:circ-mesh}.

\begin{figure}[!htb]  
    \begin{minipage}[t]{0.48\textwidth}
        \centering
        \includegraphics[width=.8\textwidth]{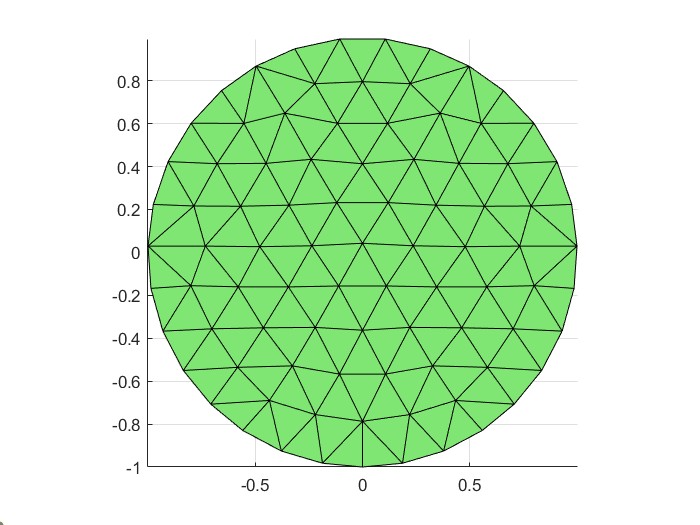}
        \caption{An example of a triangular mesh of $\D$.}
        \label{fig:circ-mesh}
    \end{minipage}
  \hfill
  \begin{minipage}[t]{0.48\textwidth}
    \centering
        \includegraphics[width=.95\textwidth]{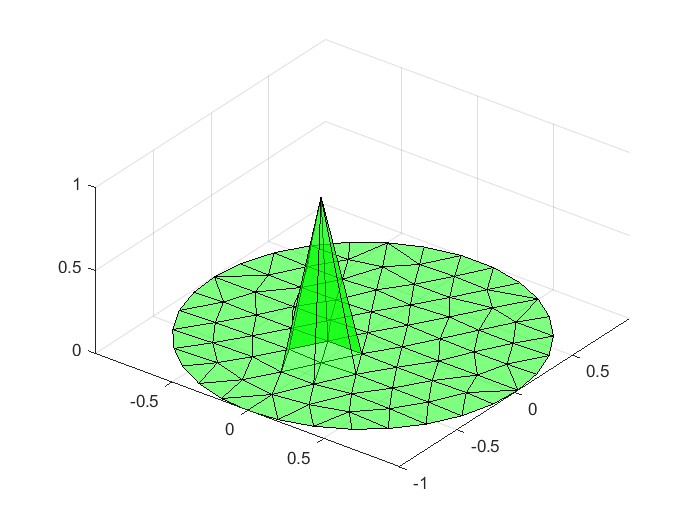}
        \caption{A nodal basis function for $S_h$.}
        \label{fig:nodalbasis}
  \end{minipage}
\end{figure}

\begin{rem}
Instead of triangulation, we could choose any other partition of $X^h$, for example rectangular elements and use bilinear functions as was done in \cite{FalkRS_NussbaumRD_2016b}, which is a valid alternative.  However, triangulation provides more structure that makes the implementation faster and easier.   
\end{rem}

Let $h_\tau=\operatorname{diam}(\tau)$ and
define $h= \max_{\tau}h_\tau.$
 On an element $\tau$ of the mesh, we define $\mathbb{P}_1(\tau)$ as the space of affine polynomials on $\tau$. Furthermore, let $S_h$ be the space of piecewise linear functions on  $X^h$
$$
S_h=\{v\in C^0(X^h)\ :\  v\mid_\tau\in \mathbb{P}_1(\tau) \}.
$$
To use the finite element space $S_h$ for computations, we need some basis functions. Since any element from $S_h$ is uniquely defined by its  values at the nodes of the triangulation $\{x_j\}_{j=1}^N$,  we choose basis functions  $\{\psi_i(x)\}_{i=1}^N$ satisfying
$$
\psi_i(x_j)=\delta_{ij}=
\begin{cases}
1\quad i=j\\
0\quad i\neq j \end{cases}
i,j=1,2,\dots, N,
$$
and define a nodal (Lagrange) interpolation operator $\mathcal{I}_h:C^0(X^h)\to S_h$  by
$$
\mathcal{I}_hv(x)= \sum_{j=1}^N v(x_j)\psi_j(x).
$$  
Since the  nodal interpolant $\mathcal{I}_h:C(X^h)\to S_h$ is invariant on  $S_h$, i.e.
 $\mathcal{I}_h q =q$ for any $q \in S_h$, and 
 \[
\|\mathcal{I}_hv\|_{L^\infty(X^h)}
\le\|v\|_{L^\infty(X^h)},
\]
its operator norm is \(1\), and it preserves positivity.

In addition, on each simplex $\tau$, for every real-valued \(v\in C^2(\overline\tau)\), we have (cf. \cite{Waldron_1998})
\[
\|v-\mathcal{I}_hv\|_{L^\infty(\tau)}
\le
\frac{n}{4(n+1)}h_\tau^2
\sup_{y\in\tau}\|D^2v(y)\|_{\mathrm{op}},
\]
where $\|\cdot\|_{op}$ is the Euclidean operator norm, defined by
\[
\|A\|_{\mathrm{op}}=\sup_{|\xi|=1}|A\xi|.
\]
Since the Hessian \(D^2v(x)\) of a real-valued function is symmetric,
\[
\|D^2v(x)\|_{\mathrm{op}}
=
\sup_{|\xi|=1}
\bigl|\xi^\top D^2v(x)\xi\bigr|
=
\max_{1\le k\le n}|\lambda_k(D^2v(x))|.
\]
Thus it measures the largest absolute second directional derivative. In particular,
\[
\|D^2v(x)\|_{\mathrm{op}}
\le n\max_{i,j}|\partial^2_{ij}v(x)| 
\]
and as a result $\mathcal{I}_h$ satisfies
 \begin{equation}\label{eq: interpolant}
    \|v-\mathcal{I}_hv\|_{L^\infty(\tau)} \leq \frac{n^2}{4(n+1)}h_\tau^2 |v|_{W^{2,\infty}(\tau)}.
 \end{equation}

\subsection{Two-sided estimates for $\tilde{\rho}_t$}
Throughout, \(t\in\operatorname{Fin}(\mathcal S)\) and the constants may depend on \(t\).
Suppose that, on the computational domain $X^h$,
\[
|\nabla\widetilde\rho_t(x)|
\le C_\nabla\widetilde\rho_t(x),
\qquad
\|D^2\widetilde\rho_t(x)\|_{\mathrm{op}}
\le C_H\widetilde\rho_t(x),
\]
with computable constants $C_\nabla$ and $C_H$. Since each simplex is convex, the gradient bound implies
\[
\sup_{y\in\tau}\widetilde\rho_t(y)
\le e^{C_\nabla h_\tau}\widetilde\rho_t(x),
\qquad x\in\tau.
\]
Consequently,
\[
\begin{aligned}
|\widetilde\rho_t(x)-\mathcal{I}_h\widetilde\rho_t(x)|
&\le
\frac{n}{4(n+1)}h_\tau^2
\sup_{y\in\tau}\|D^2\widetilde\rho_t(y)\|_{\mathrm{op}}
\le
\operatorname{err}_\tau\,\widetilde\rho_t(x),
\end{aligned}
\]
where
\begin{equation}\label{eq: error on each element}
\operatorname{err}_\tau
=
\frac{nC_H}{4(n+1)}
e^{C_\nabla h_\tau}h_\tau^2.
\end{equation}
Provided \(\max_\tau\operatorname{err}_\tau<1\), rearranging yields
\begin{equation}\label{eq: 2sided estimates}
\frac{\mathcal{I}_h\widetilde\rho_t(x)}{1+\operatorname{err}_\tau}
\le \widetilde\rho_t(x)
\le
\frac{\mathcal{I}_h\widetilde\rho_t(x)}{1-\operatorname{err}_\tau},
\qquad \forall x\in\tau.
\end{equation}
To obtain \(C_H\) from the derivative estimates of Section \ref{sec:derbound}, use
\[
\|D^2\widetilde\rho_t(x)\|_{\mathrm{op}}
\le
n\max_{i,j}|\partial^2_{ij}\widetilde\rho_t(x)|.
\]
Thus, if only the coordinatewise bound
\[
\max_{i,j}|\partial^2_{ij}\widetilde\rho_t|
\le C_2\widetilde\rho_t
\]
is available, one may take \(C_H=nC_2\). A sharper operator-norm bound can be inserted directly.
Thus, $\mathcal{I}_h\tilde{\rho}_t$ provides upper and lower pointwise bounds for $\tilde{\rho}_t$. For fixed \(t\), the multiplicative correction factors \((1\pm\operatorname{err}_\tau)^{-1}\) converge uniformly to \(1\) at order \(h^2\). 

For M\"obius systems and $t\ge0$,
Corollary~\ref{cor:mobius-first-second} gives the direct bounds
\[
C_\nabla=\frac{2t}{\widetilde\eta},
\qquad
C_H=\frac{2t(2t+1)}{\widetilde\eta^2}.
\]
Consequently, the relative interpolation error in (4.2)
may be chosen as
\[
\operatorname{err}_\tau
=
\frac{n}{4(n+1)}
\frac{2t(2t+1)}{\widetilde\eta^2}
\exp\!\left(
\frac{2t}{\widetilde\eta}h_\tau
\right)h_\tau^2.
\]
For estimates uniform over $0\le t\le t_0$, one may replace
$t$ by $t_0$ in this expression. Since $C_H$ directly bounds
the Hessian operator norm, no additional factor of $n$
is needed.

\subsection{\texorpdfstring{Approximating the Perron-Frobenius operator when the alphabet $E$ is finite.}{}}\label{sec: Approx perron}

Next we want to approximate the  Perron-Frobenius operator 
$F_t:C(X) \ra C(X)$ which was introduced in \eqref{ft}. Recall that
$$ F_t(g)(x) = \sum_{e \in E} \norm{D\phi_e(x)}^t g(\phi_e(x)) \chi_{X_{t(e)}}(x), \quad g \in C(X),
$$
\[
\lambda_t=e^{P(t)},\qquad
G_t\widetilde\rho_t=\lambda_t\widetilde\rho_t
\quad\text{on }S.
\]
Using \eqref{eq: 2sided estimates}, for \(x\in X_v^h\), we have
\begin{equation}\label{eq: 2sided estimates sum}
\frac1{1+\operatorname{err}}
\sum_{t(e)=v}\|D\phi_e(x)\|^t
\mathcal{I}_h\widetilde\rho_t(\phi_e(x))
\le \lambda_t\widetilde\rho_t(x)
\le
\frac1{1-\operatorname{err}}
\sum_{t(e)=v}\|D\phi_e(x)\|^t
\mathcal{I}_h\widetilde\rho_t(\phi_e(x)),
\end{equation}
where 
\[
\operatorname{err}:=
\max_{\tau\in\mathcal T_h}\operatorname{err}_\tau<1.
\]
Let \(v(j)\) denote the unique vertex such that \(x_j\in X_{v(j)}^h\).
For the finite element basis functions \(\psi_k\)  set
\[
(\mathsf M_t)_{jk}
=
\sum_{\substack{e\in E\\t(e)=v(j)}}
\|D\phi_e(x_j)\|^t
\psi_k(\phi_e(x_j)),
\qquad x_j\in X_{v(j)}^h,
\]
and define nonnegative matrices 
\[
A_t=\frac{\mathsf M_t}{1+\operatorname{err}},
\qquad
B_t=\frac{\mathsf M_t}{1-\operatorname{err}}.
\]
Then for
\[
(v_t)_j=\widetilde\rho_t(x_j),
\]
the interpolation inequalities give
\[
A_tv_t\le\lambda_tv_t\le B_tv_t.
\]
One of the technical difficulties of assembling the above matrices is to locate an element $\tau$ that contains $\phi_e(x_j)$. At this point, the structure of the triangulation is particularly useful  as one can use a barycentric point location, which makes the assembly rather efficient.  For example if for the node $x_j$, the image 
$\phi_e(x_j)\in \tau_i$ for some $i$, then we have $\phi_e(x_j)=\lambda_1x_1^i+\cdots+\lambda_{n+1}x_{n+1}^i$, where $x_1^i,\dots,x_{n+1}^i$ are the vertices of the simplex $\tau_i$ and $\lambda_1,\dots, \lambda_{n+1}\geq 0$, $\lambda_1+\cdots+\lambda_{n+1}=1 $ are the barycentric coordinates  of the point $\phi_e(x_j)$. 

Thus, the contribution from this branch enters row $j$
of $A_t$ and $B_t$, in the columns corresponding to
the global indices of $x_1^i,\ldots,x_{n+1}^i$,
with the appropriate barycentric weights and
comparison factors.
 This step can be vectorized for all $e\in E$, making the assembly very efficient.

\subsection{Computing upper and lower bounds of the Hausdorff dimension}

The matrices $A_t, B_t$ consist of non-negative entries and we can use the following key result for such matrices \cite[Lemma 3.2]{FalkRS_NussbaumRD_2016b}.
\begin{lem}
    \label{lemma: monotonicity}
Let $M$ be an $N\times N$ matrix with non-negative entries  and $w$ an $N$-vector with strictly positive components. Then,
$$
\begin{aligned}
\text{if}\quad (Mw)_j &\geq \lambda w_j, \quad j=1,\dots,N,\ \text{then}\ r(M)\geq \lambda,\\
\text{if}\quad (Mw)_j &\leq \lambda w_j, \quad j=1,\dots,N,\ \text{then}\ r(M)\le \lambda,
\end{aligned}
$$
where $r(M)$ denotes the spectral radius of $M$.
\end{lem}
Lemma \ref{lemma: monotonicity} implies
\[
r(A_t)\le\lambda_t\le r(B_t).
\]
Let $t^*=\dim_{\mathcal H}(J_{\mathcal S})=\inf\{t\geq0:P(t)\leq0\}$.
Our goal is to certify parameters $0<\underline t<\overline t$ in
$\Fin(\mathcal S)$ such that
\[
    r(A_{\underline t})>1,\qquad r(B_{\overline t})<1.
\]
The spectral bounds then give
\[
    P(\underline t)>0,\qquad P(\overline t)<0.
\]
Continuity of pressure between these finite-pressure parameters and its
strict decrease show that its unique zero lies in
$(\underline t,\overline t)$. By Bowen's formula this zero is $t^*$.
Thus a successful certification proves both regularity and the rigorous
dimension bound $\underline t<t^*<\overline t$; regularity need not be
assumed in advance.

For regular systems, the numerical search can therefore be organized as a
root-finding problem for $\log\lambda_t=P(t)$. In our computations we used
a variation of the secant method to choose trial parameters, and certified
the final bounds using the strict matrix inequalities above. For an irregular
infinite system there is no pressure zero to locate, and this two-sided
sign-change procedure does not apply. The definition of Bowen's parameter
still gives an upper bound $t^*\leq t$ whenever a certified estimate yields
$P(t)\leq0$; finite subsystems provide lower bounds by Theorem~\ref{721}.
A regular infinite system with $t^*=\theta$ and $P(\theta)=0$ may also lack
a finite-pressure parameter below $t^*$. In that boundary case the same
finite-subsystem approach supplies lower bounds, rather than the two-sided
finite-pressure sign-change procedure.

\subsection{Case of an infinite alphabet}\label{sec: infinite alphabel}

In the case of the infinite alphabet, we consider the truncated finite  alphabet $\tilde{E}\subset E$. The eigenfunction \(\widetilde\rho_t\) and interpolation constants refer to the full infinite system, rather than the truncated subsystem. Define the matrix $\widetilde{\mathsf M}_t$ on the truncated alphabet by
\[
(\widetilde{\mathsf M}_t)_{jk}
=\sum_{\substack{e\in\widetilde E\\t(e)=v(j)}}
  \|D\phi_e(x_j)\|^t\psi_k(\phi_e(x_j)),
\qquad x_j\in X_{v(j)}^h,
\]
and set
\[
\widetilde A_t=\frac{\widetilde {\mathsf M}_t}{1+\operatorname{err}},
\qquad
\widetilde B_t=\frac{\widetilde {\mathsf M}_t}{1-\operatorname{err}}.
\]
Positivity and
\eqref{eq: 2sided estimates}, applied to
$G_t\widetilde\rho_t=\lambda_t\widetilde\rho_t$, give
\begin{equation}\label{eq: truncated operator comparison}
\widetilde A_tv_t\le\lambda_tv_t
\le\widetilde B_tv_t+r_t,
\end{equation}
where the nonnegative tail vector is
\[
(r_t)_j
=\frac{1}{1-\operatorname{err}}
 \sum_{\substack{e\in E\setminus\widetilde E\\t(e)=v(j)}}
 \|D\phi_e(x_j)\|^t
 \mathcal I_h\widetilde\rho_t(\phi_e(x_j)).
\]
In particular, the retained lower matrix requires no tail correction. 

Since $t\in\Fin(\mathcal S)$, bounded distortion on the enlarged sets
implies
\[
\sup_{x\in X_v^h}
\sum_{\substack{e\in E\setminus\widetilde E\\t(e)=v}}
\|D\phi_e(x)\|^t
\le K^t\sum_{e\in E\setminus\widetilde E}
\|D\phi_e\|_\infty^t\longrightarrow0
\]
as the retained alphabets increase to $E$. Thus the tail series converges
uniformly. For a rigorous computation, we require explicit computable bounds
\begin{equation}\label{eq: omitted branch weight bounds}
\sum_{\substack{e\in E\setminus\widetilde E\\t(e)=v(j)}}
\|D\phi_e(x_j)\|^t\le\beta_{t,j},
\qquad j=1,\ldots,N.
\end{equation}
The estimates for $\beta_{t,j}$ depend on the system and the retained alphabet.

To concentrate the upper-tail correction in a single column, choose a
reference node $x_\ell$ and suppose that a computable constant
$R_{t,\ell}$ satisfies
\begin{equation}\label{eq: global nodal comparison for tail}
\widetilde\rho_t(x_k)
\le R_{t,\ell}\widetilde\rho_t(x_\ell),
\qquad k=1,\ldots,N.
\end{equation}
Such a comparison must include all graph components. For a one-vertex
system with convex $X^h$, the gradient bound of the preceding subsections gives
$$
R_{t,\ell}=\exp(C_\nabla\max_k|x_k-x_\ell|).
$$
On disconnected
computational domains, a global comparison requires separate justification.
Nonnegative barycentric interpolation and
\eqref{eq: global nodal comparison for tail} imply
\[
\mathcal I_h\widetilde\rho_t(z)
\le R_{t,\ell}(v_t)_\ell,
\qquad z\in X^h.
\]
Consequently, defining
\[
(c_t)_j=\frac{R_{t,\ell}\beta_{t,j}}{1-\operatorname{err}},
\qquad
\widehat B_t=\widetilde B_t+c_t\mathbf e_\ell^{\mathsf T},
\]
where $\mathbf e_\ell$ is the $\ell$-th coordinate vector, we obtain
\[
\lambda_tv_t\le\widehat B_tv_t.
\]
Thus $(c_t)_j$ is added to entry $(j,\ell)$ of the upper matrix. If a
single bound $C_0\ge\max_j(c_t)_j$ is used, it suffices to add $C_0$ to
every entry in column $\ell$.

\begin{rem}
If no effective global comparison
\eqref{eq: global nodal comparison for tail} is available, a more conservative
upper correction follows directly from positivity:
\[
\mathcal I_h\widetilde\rho_t(z)
\le\max_k(v_t)_k\le\sum_{k=1}^N(v_t)_k.
\]
With $(d_t)_j=\beta_{t,j}/(1-\operatorname{err})$, one may therefore use
\[
\widehat B_t=\widetilde B_t+d_t\mathbf1^{\mathsf T},
\]
where $\mathbf1$ is the all-ones vector. This correction does not require
a comparison between graph components and can be applied as a rank-one
operation without storing a dense matrix.
\end{rem}

For either upper correction, Lemma~\ref{lemma: monotonicity} gives
\[
r(\widetilde A_t)\le\lambda_t\le r(\widehat B_t).
\]
The dimension certificates of the preceding subsection consequently become
\[
r(\widetilde A_{\underline t})>1,
\qquad
r(\widehat B_{\overline t})<1,
\]
for $\underline t,\overline t\in\Fin(\mathcal S)$ with
$0<\underline t<\overline t$. As before, numerical evaluation and rounding
errors must be included when certifying these strict inequalities.

\begin{rem}
The infinite-alphabet case has two approximation errors: interpolation on
the computational mesh and truncation of the alphabet. For a fixed mesh,
the matrix dimension is $N$, independently of $\widetilde E$, although
the retained alphabet affects the entries, sparsity, and assembly cost.
In applications, one chooses computable tail bounds that tend to zero as
the retained alphabets increase to $E$, and balances their contribution
against the interpolation error. The amount of truncation needed depends
on the particular system and the desired dimension enclosure.
\end{rem}

\section{Applications}
\label{sec:app}
In this section, we illustrate the method for various CGDMSs. In particular, we verify that these systems are indeed CGDMSs and  highlight some properties of the general families that these systems belong to. In Section \ref{sec: Hausdorff numerical} we will describe the specific implementation points of our numerical method for these examples.

\subsection{\texorpdfstring{$n$-dimensional continued fractions}
{n-dimensional continued fractions}}
In this section we review $n$-dimensional continued fractions and some of their dynamical properties. We  construct a uniformly contracting CIFS with the same limit set. We use the usual notation for the standard basis in $\R^n$, namely, ${\e}_j=(0,\dots,1,\dots,0)$, with $1$ in the $j$-th position.

\begin{defn}[$n$-dimensional Continued Fractions]
Let $v_{1/2} =(1/2,0,...,0)=\frac{1}{2}{\e}_1$ and let $|\cdot |$ denote the Euclidean norm. The $n$-dimensional continued fraction family, denoted $\cfe$, $n\geq 2$, consists of the maps
\begin{equation}
   \left\{\phi_e:X \rightarrow X | \ e \in \E:=\N \times \Z^{n-1}, \ \phi_e( x) = \frac{x+ e}{| x+e|^2} \right\},
\end{equation}
where 
$$X = \left\{ x \in \R^n : | x -v_{1/2}| \leq \frac{1}{2}\right\}.$$
\end{defn}

\begin{figure}[!htb]  
    \begin{minipage}[t]{0.31\textwidth}
        \centering
        \includegraphics[width = 1.1\textwidth]{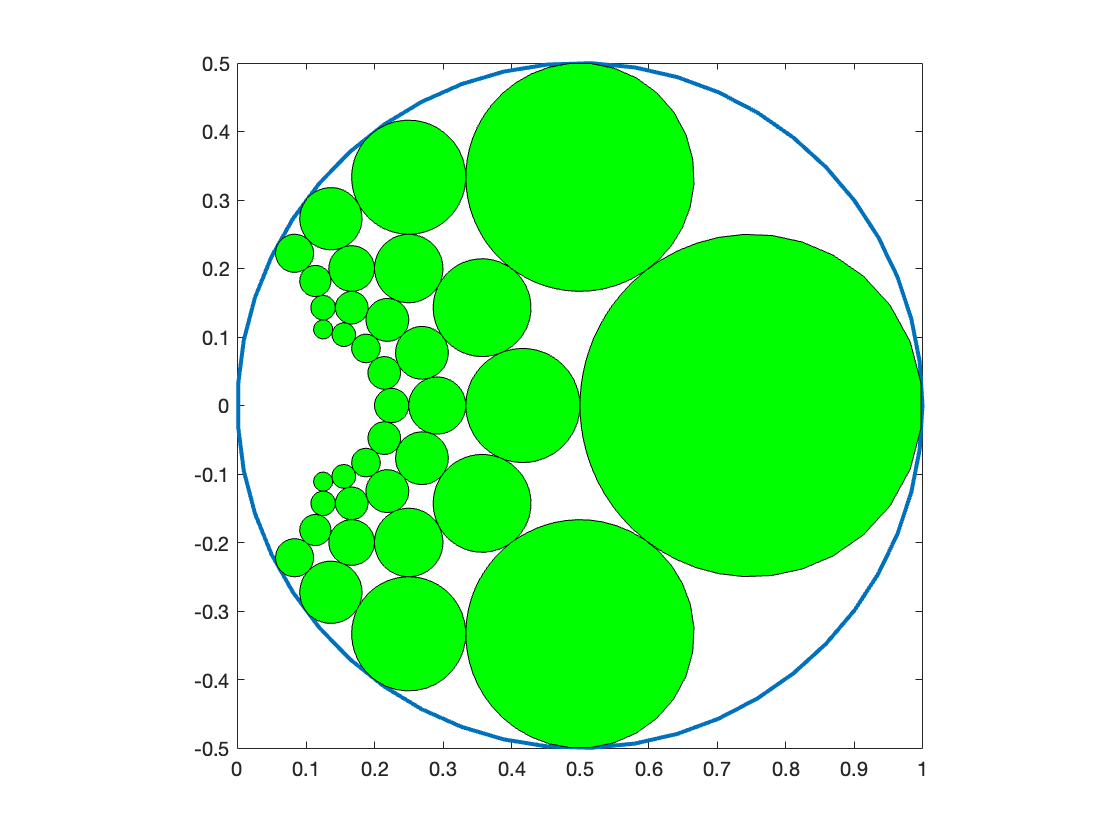}
    \caption{2D cont. frac., first iteration}
    \label{fig:2D_cont_1}
    \end{minipage}
  \hfill
    \begin{minipage}[t]{0.31\textwidth}
        \centering
        \includegraphics[width = 1.1\textwidth]{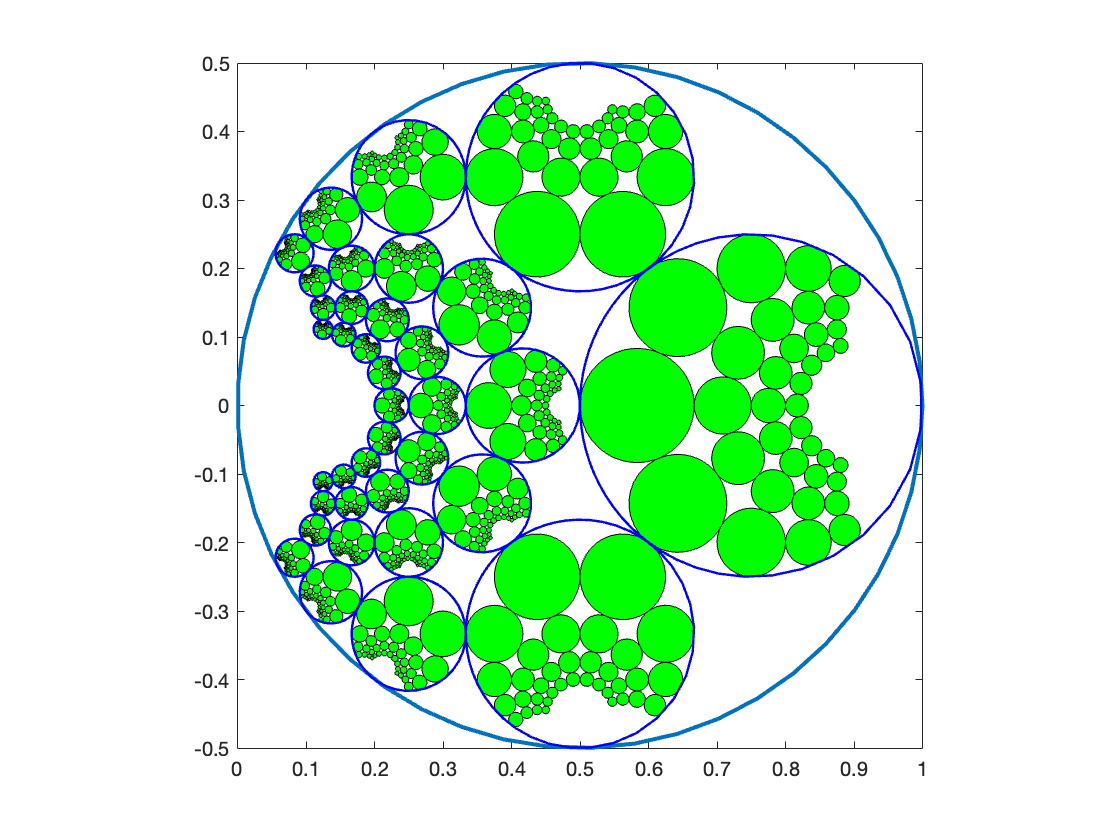}
    \caption{2D cont. frac., second iteration}
    \label{fig:2D_cont_2}
    \end{minipage}
  \hfill
  \begin{minipage}[t]{0.31\textwidth}
    \centering
\includegraphics[width =1.1\textwidth]{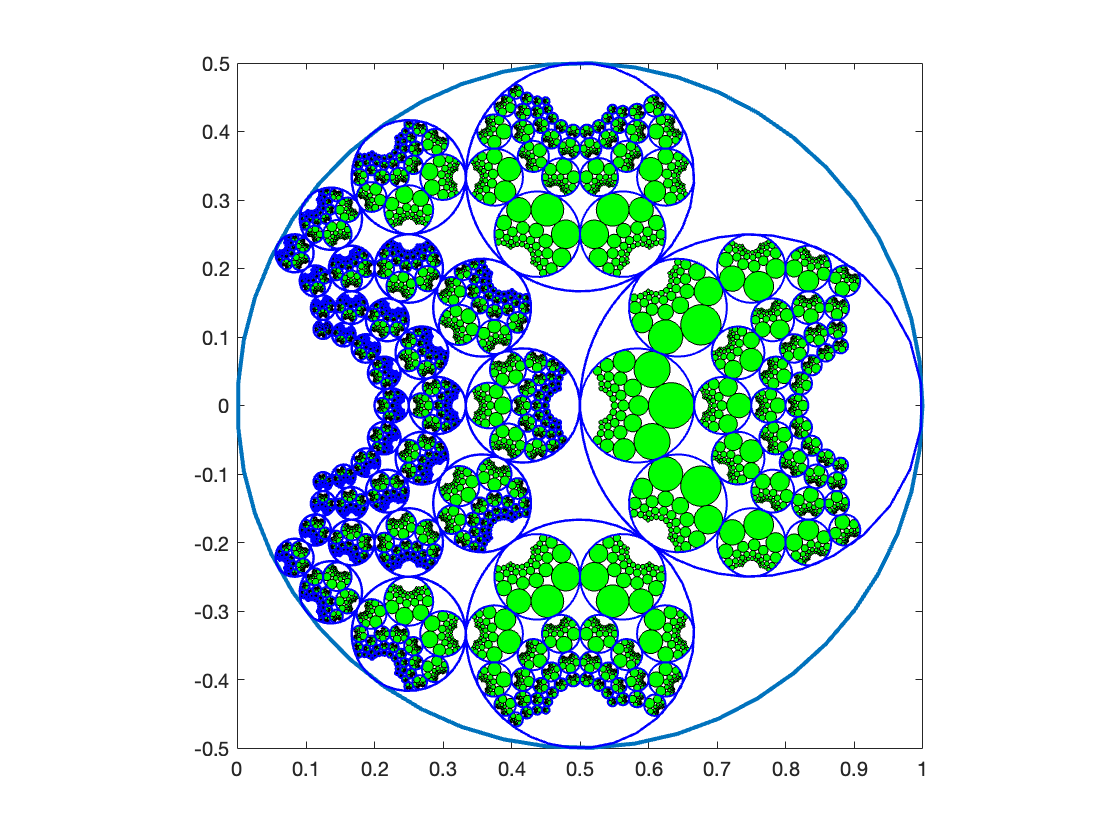}
    \caption{2D cont. frac., third iteration}
    \label{fig:2D_cont_3}
  \end{minipage}
\end{figure}

\begin{figure}[!htb]  
    \begin{minipage}[t]{0.48\textwidth}
        \centering
      \includegraphics[width=0.9\linewidth]{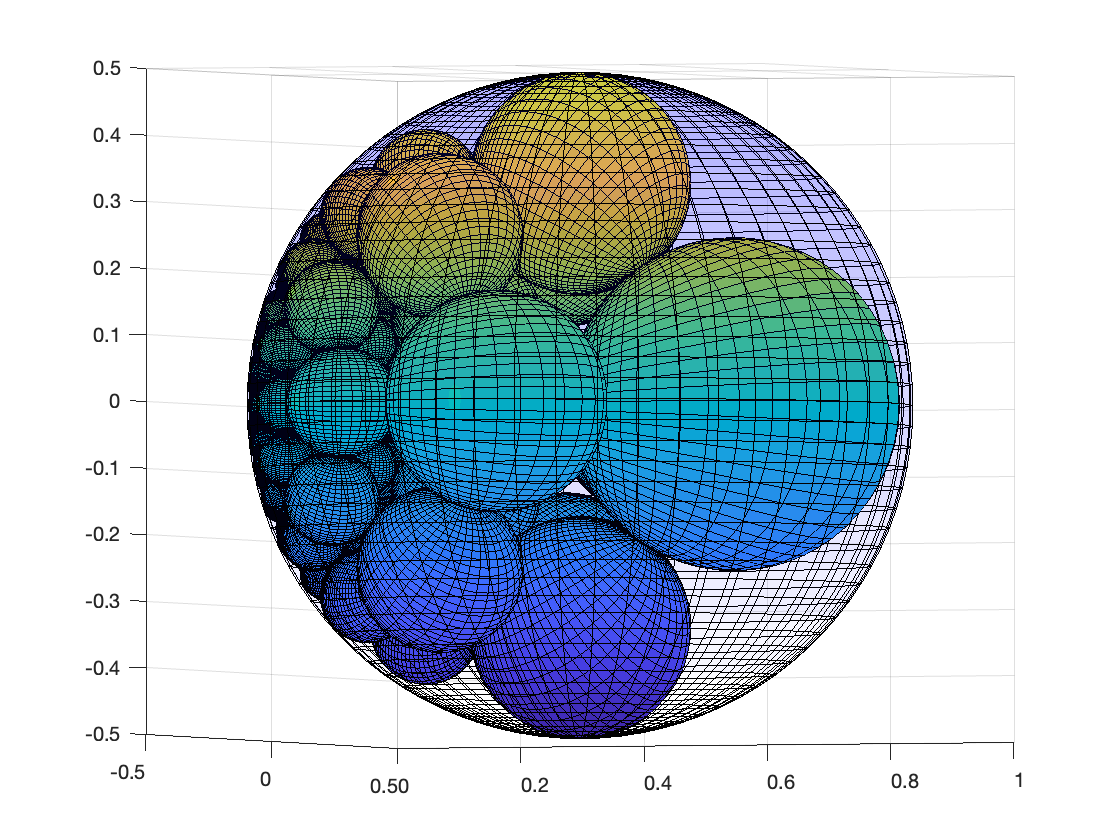}
    \caption{3D continued fraction IFS. First iteration.}
    \label{fig:3D_1}
    \end{minipage}
  \hfill
  \begin{minipage}[t]{0.48\textwidth}
   \centering
      \includegraphics[width=1.0\linewidth]{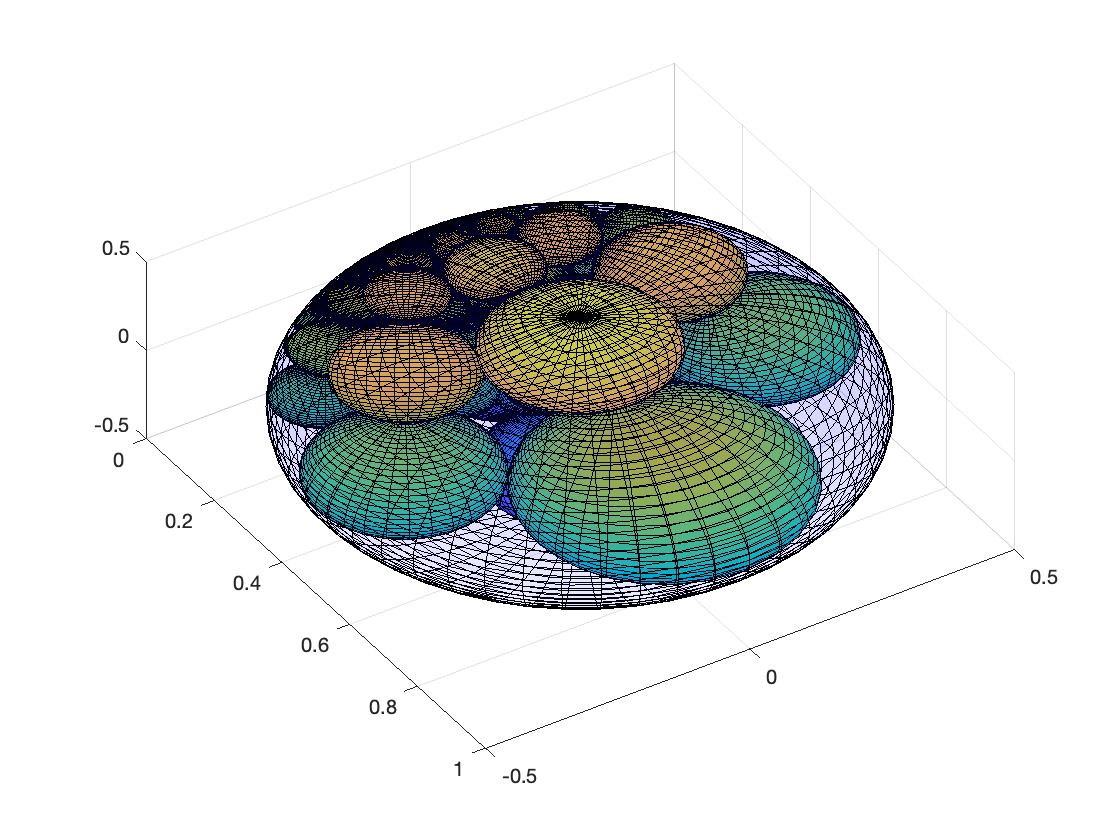}
    \caption{3D continued fraction IFS. First iteration. View from above}
    \label{fig:3D_2}
  \end{minipage}
\end{figure}

First, we show that the above family maps $X$ to itself.
\begin{lem}\label{cfx2x}
For each $e \in \E$, $\phi_e:X \rightarrow X$.
\end{lem}
\begin{proof}
It suffices to show that for all $ x \in X$, $e \in \E$
\begin{equation}\label{eq; desired estimate}
\left|\phi_e(x) - v_{1/2}\right| \leq \frac{1}{2}.
\end{equation}
Squaring, we find
\[
\left|\phi_e(x) - v_{1/2}\right|^2 =\left|\frac{x+e}{|x+e|^2}\right|^2-2\frac{(x+e)\cdot v_{1/2}}{|x+e|^2}+|v_{1/2}|^2=\frac{1-x_1-e_1}{|x+e|^2}+\frac{1}{4}\le \frac{1}{4},
\]
since  $x_1 + e_1 \geq 1$. 
\end{proof}
However before we proceed, we need to address one obstacle. The above family is not a uniformly contracting CIFS, since
$$
\|D\phi_{\mathbf e_1}(0)\|=1.
$$
The next lemma shows that passing to all two-letter compositions gives a uniformly contracting family. We therefore consider the two-step system
\[
    \mathcal S^{(2)}
    =\{\psi_{a,b}:=\phi_a\circ\phi_b:(a,b)\in\E\times\E\}.
\]

\begin{lem}[Two-step contraction]
Every map \(\psi_{a,b}=\phi_a\circ\phi_b\), with
\(a,b\in\E\), maps \(X\) into itself and is Lipschitz
on \(X\) with Lipschitz constant at most \(1/4\).
More precisely,
\[
    \sup_{x\in X}\|D\psi_{a,b}(x)\|\leq\frac14.
\]
\end{lem}

\begin{proof}
Write \(\iota(y)=y/|y|^2\). Since
\(\|D\iota(y)\|=|y|^{-2}\), and the derivatives are
similarities, the chain rule gives, with \(y=x+b\),
\[
    \|D\psi_{a,b}(x)\|
    =\frac{1}{|y|^2|\iota(y)+a|^2}
    =\frac{1}
       {|a|^2\left|y+\dfrac{a}{|a|^2}\right|^2}.
\]
For \(x\in X\), we have \(y_1=x_1+b_1\geq1\) and
\(a_1\geq1\). Therefore,
\[
    |a|\left|y+\frac{a}{|a|^2}\right|
    \geq |a|y_1+\frac{a_1}{|a|}
    \geq |a|+\frac1{|a|}
    \geq2.
\]
It follows that
\[
    \|D\psi_{a,b}(x)\|\leq\frac14.
\] 
Since \(X\) is convex, each \(\psi_{a,b}\) is Lipschitz
on \(X\) with Lipschitz constant at most \(1/4\).
\end{proof}

\begin{remark}
It is easy to see that the limit sets generated by the families \(\{\phi_e\}\) and \(\{\psi_{a,b}\}\) coincide. Indeed, for any sequence \(\omega=(e_1,e_2,\ldots)\in \E^{\mathbb N}\),
the compact sets
\[
    \phi_{e_1}\circ\cdots\circ\phi_{e_m}(X)
\]
are nested, and their diameters are at most
\[
    4^{-\lfloor m/2\rfloor}\operatorname{diam}(X).
\]
They therefore determine a unique limit point. Grouping the
letters into consecutive pairs gives
\[
    \psi_{e_1,e_2}\circ\cdots\circ\psi_{e_{2m-1},e_{2m}}
    =
    \phi_{e_1}\circ\cdots\circ\phi_{e_{2m}},
\]
so both codings determine the same point. Conversely, every
sequence of pairs corresponds to a sequence in \(\E^{\mathbb N}\).
Thus the two limit sets coincide.
\end{remark}

Next we verify the open set condition. 
\begin{lem}
For any $ e_1,e_2 \in \E$ with $e_1 \neq e_2$,
$$ \phi_{e_1}(\Int(X)) \cap \phi_{e_2}(\Int(X)) = \emptyset.$$
\end{lem}

\begin{proof}
To simplify the notation, put \(U=\operatorname{Int}(X)\).
The translated open balls \(U+e\), \(e\in \E\), are pairwise disjoint, because their centers are at least one unit apart and their radii are \(1/2\). Injectivity of \(\iota\) therefore gives
\[
    \phi_a(U)\cap\phi_c(U)=\varnothing
    \qquad(a\neq c).
\]
\end{proof}
The open set condition also holds for the two-step system. Indeed, the self-mapping identity shows that
\(\phi_e(U)\subset U\), where \(U=\operatorname{Int}(X)\).
For distinct pairs \((a,b)\) and \((c,d)\), if \(a\neq c\), then
\[
    \psi_{a,b}(U)\subset\phi_a(U),
    \qquad
    \psi_{c,d}(U)\subset\phi_c(U),
\]
so their images are disjoint. If \(a=c\), then \(b\neq d\), and disjointness follows from
\(\phi_b(U)\cap\phi_d(U)=\varnothing\) and injectivity
of \(\phi_a\).

For the common conformal neighborhood, choose \(0<\delta<1\)
and set
\[
    W=X(\delta)=B(v_{1/2},1/2+\delta).
\]
If \(x\in W\), then \(x_1>-\delta\), so
\[
    (\phi_b(x))_1
    =\frac{x_1+b_1}{|x+b|^2}>0.
\]
Thus \(\phi_b(x)+a\) has first coordinate greater than one.
Using
\[
    |\iota(z)-v_{1/2}|^2
    =\frac14+\frac{1-z_1}{|z|^2},
\]
we obtain
\[
    \psi_{a,b}(W)\subset U\subset W.
\]
Neither inversion encounters its pole on the relevant domain,
so each \(\psi_{a,b}\) extends to a conformal diffeomorphism
of \(W\) onto its image. The remaining bounded distortion condition follows from Remark~\ref{remark: conds 3}. 
Hence \(\mathcal S^{(2)}\) is a uniformly contracting CIFS. The two families have the same limit set. Thus the limit set of \(n\)-dimensional continued fractions admits a CIFS representation.

\begin{remark}[The one-step transfer operator]
Although the CIFS theory is applied to the two-step system,
our numerical implementation uses the one-step transfer
operator
\[
    (F_tg)(x)
    =\sum_{e\in\E}\|D\phi_e(x)\|^t g(\phi_e(x))
    =\sum_{e\in\E}\frac{g(\phi_e(x))}{|x+e|^{2t}}.
\]
For \(2t>n\), the summability condition
\[
    \sum_{e\in\E}\|D\phi_e\|_{\infty,X}^t<\infty
\]
holds, so \(F_t\) is a bounded positive operator on \(C(X)\).

Let $F_t^{(2)}$ denote the transfer operator associated with the two-step CIFS:
\[
    (F_t^{(2)}g)(x)
    =\sum_{a,b\in\E}
       \|D(\phi_a\circ\phi_b)(x)\|^t
       g(\phi_a\circ\phi_b(x)).
\]
The chain rule and absolute convergence give
\[
    F_t^{(2)}=F_t^2.
\]
Define
\[
\lambda_t^{(1)}:=r(F_t),
\qquad
\Lambda_t=r(F_t^{(2)})
=\bigl(\lambda_t^{(1)}\bigr)^2>0.
\]
Let \(\rho_t\) be the normalized positive leading
eigenfunction of \(F_t^{(2)}\),  satisfying
\[
    F_t^{(2)}\rho_t=\Lambda_t\rho_t.
\]
Since \(F_t\) commutes with \(F_t^{(2)}\),
\[
    F_t^{(2)}(F_t\rho_t)
    =F_t(F_t^{(2)}\rho_t)
    =\Lambda_tF_t\rho_t.
\]
The function \(F_t\rho_t\) is strictly positive. Uniqueness,
up to multiplication by a positive scalar, of the positive
leading eigenfunction of \(F_t^{(2)}\) therefore implies
\[
    F_t\rho_t=c_t\rho_t
\]
for some \(c_t>0\). Applying \(F_t\) again gives
\(c_t^2=\Lambda_t\). By the spectral mapping theorem,
\[
    r(F_t^{(2)})=r(F_t)^2,
\]
and hence
\[
    F_t\rho_t=r(F_t)\rho_t.
\]

Thus the one-step and two-step transfer operators have the
same normalized positive leading eigenfunction. In particular,
the derivative estimates obtained from the two-step CIFS apply
to the eigenfunction used in the one-step numerical method.
Moreover,
\[
    r( F_t^{(2)})=1 \quad\Longleftrightarrow\quad r(F_t)=1,
\]
so the two operators have the same spectral-radius-one criterion whenever
the system is regular. In general, if $P_2$ is the pressure of the two-step
CIFS, Bowen's formula reads
\[
    \dim_H(J)=\inf\{t\geq0:P_2(t)\leq0\}.
\]
On the summability domain, $P_2(t)=2\log{\lambda_t^{(1)}}$; thus spectral sign
certificates from either operator give the same dimension bounds without
assuming the existence of a pressure zero.

The tail estimates below concern the omitted branches of
\(F_t\). They remain applicable because they use only the
one-step branch weights, the inclusion \(\phi_e(X)\subset X\),
and the derivative bounds for this common eigenfunction.
\end{remark}

We next establish a tail bound for the original continued fraction branches.

\begin{lem}[Tail bound]\label{lemma: Tail bound continued fractions}
Let \(\rho_t\) be the positive eigenfunction associated with
the two-step CIFS, and choose \(C_0\) such that
\[
    |\nabla\rho_t(x)|\leq C_0\rho_t(x),
    \qquad x\in X.
\]
Suppose that $2t>n$. Then, for every $R\geq1$ and every $x,y\in X$,
\[
\sum_{\substack{e\in \E\\ |e|\geq R+2}}
\frac{\rho_t(\phi_e(x))}{|x+e|^{2t}}
\leq
C_{n,t}\,
e^{C_0}
R^{n-2t}\rho_t(y),
\]
where the constant $C_{n,t}>0$ depends only on $n$ and $t$.
\end{lem}

\begin{proof}
Since $\operatorname{diam}(X)=1$ and $0\in X$, we have $|x|\leq1$ for every
$x\in X$. Thus, if $|e|\geq R+2$ with $R\geq1$, then
\[
|x+e|\geq |e|-|x|\geq |e|-1\geq \frac12|e|,
\]
and consequently
\begin{equation}
\frac1{|x+e|^{2t}}
\leq 2^{2t}\frac1{|e|^{2t}}.
\label{eq:first}
\end{equation}
Next, for any \(u,v\in X\), Proposition~\ref{prop: comparing values} gives,
\[
\rho_t(u)\leq
e^{C_0|u-v|}\rho_t(v).
\]
Since $\operatorname{diam}(X)=1$ and $\phi_e(x)\in X$ by Lemma~\ref{cfx2x},
\[
\rho_t(\phi_e(x))
\leq e^{C_0}\rho_t(y).
\]
Therefore
\begin{equation}
\sum_{\substack{e\in \E\\ |e|\geq R+2}}
\frac{\rho_t(\phi_e(x))}{|x+e|^{2t}}
\leq
e^{C_0}\rho_t(y)
\sum_{\substack{e\in \E\\ |e|\geq R+2}}
\frac1{|x+e|^{2t}}.
\label{eq:reduce}
\end{equation}
It remains to estimate the lattice sum.  For $k\in \mathbb{N}$, put
\[
A_k=\{e\in \E:k\leq |e|<k+1\}.
\]
There is a constant $C_n>0$, depending only on $n$, such that
\[
\#A_k\leq C_n(k+1)^{n-1}.
\]
Indeed, $A_k$ is contained in the lattice points of the annulus
$\{z\in\mathbb R^n:k\leq|z|<k+1\}$, and a unit-cube covering gives this
standard polynomial bound.

Using \eqref{eq:first},
\[
\sum_{\substack{e\in \E\\ |e|\geq R+2}}
\frac1{|x+e|^{2t}}
\leq
2^{2t}\sum_{\substack{e\in \E\\ |e|\geq R+2}}\frac1{|e|^{2t}}\leq
2^{2t}C_n\sum_{k\geq R+1}
\frac{(k+1)^{n-1}}{k^{2t}}.
\]
Since $k\geq R+1\geq2$, we have $(k+1)^{n-1}\leq2^{n-1}k^{n-1}$.
Hence
\[
\sum_{\substack{e\in \E\\ |e|\geq R+2}}
\frac1{|x+e|^{2t}}
\leq
2^{2t+n-1}C_n
\sum_{k\geq R+1}k^{n-1-2t}.
\]
Because $2t>n$, the integral test gives
\[
\sum_{k\geq R+1}k^{n-1-2t}
\leq
\int_R^\infty r^{n-1-2t}\,dr
=
\frac{R^{n-2t}}{2t-n}.
\]
Thus
\[
\sum_{\substack{e\in \E\\ |e|\geq R+2}}
\frac1{|x+e|^{2t}}
\leq
2^{2t+n-1}C_n
\frac{R^{n-2t}}{2t-n}.
\]
Combining this estimate with \eqref{eq:reduce} proves
\[
\sum_{\substack{e\in \E\\ |e|\geq R+2}}
\frac{\rho_t(\phi_e(x))}{|x+e|^{2t}}
\leq
2^{2t+n-1}C_n
e^{C_0}
\frac{R^{n-2t}}{2t-n}\rho_t(y).
\]
The result follows by absorbing the factor
$2^{2t+n-1}C_n/(2t-n)$ into $C_{n,t}$.
\end{proof}
\begin{rem}
Following the lines of more refined analysis from \cite{Falk}, we could obtain sharper tail bounds (cf. \cite{BCLW}). However, the above bounds are sufficient for the purpose of this paper. 
\end{rem}

\subsection{Quadratic perturbations of linear maps (\it{abc-examples})}\label{sec: abs}

In this section we discuss a CIFS in the complex plane which does not
consist of Möbius maps. Suppose that
\[
r\in(0,1),\qquad
X=\bar{B}(0,r):=\{z\in\mathbb C:|z|\leq r\},
\]
and let
\[
\phi_e(z)=a_ez+b_e+c_ez^2,\qquad e\in E\subseteq\mathbb N.
\]
The corresponding (formal) CIFS is denoted by
\[
\mathcal S_{\rm abc}
=\{\phi_e:X\to X\}_{e\in E}.
\]
An arbitrary collection of such maps need not define a CIFS. The maps
may fail to be contractions, their images may intersect, or they may
fail to be injective. Conformality is automatic for holomorphic maps where their derivatives are nonzero,
but the required extension to a common neighborhood must also be
verified. In particular, we verify the following properties:
\begin{enumerate}
    \item the maps $\phi_i$ are contractions on $X$;
    \item there is an open, connected set $W\supset X$ such that each
    $\phi_i$ extends to a uniformly contracting conformal
    diffeomorphism from $W$ onto its image and satisfies
    $\phi_i(W)\subset W$;
    \item the open set condition holds on $X$;
    \item the bounded distortion property holds.
\end{enumerate}

Many of these properties can be verified using computational means,
provided the system satisfies appropriate separation properties. An
investigation of these algorithms is beyond the scope of the paper,
and instead we show how to verify them explicitly in one particular
case.

Consider the CIFS $S_{\rm abc}$ consisting of the maps
\[
\begin{aligned}
\phi_1(z)&=0.25iz+0.1+0.1z^2,\\
\phi_2(z)&=0.2iz-0.1-0.1i+0.05z^2,\\
\phi_3(z)&=0.1z+0.1-0.1i+0.04z^2,
\end{aligned}
\]
defined on
\[
X=\bar{B}(0,0.2).
\]
To show that this system maps $X$ into itself, we use the estimates
\[
|\phi_e(z)|
\leq |a_e||z|+|b_e|+|c_e||z|^2.
\]
For $|z|\leq 0.2$, we obtain
\[
\begin{aligned}
|\phi_1(z)|
&\leq 0.25(0.2)+0.1+0.1(0.2)^2\\
&=0.05+0.1+0.004
=0.154<0.2,
\end{aligned}
\]
and
\[
\begin{aligned}
|\phi_2(z)|
&\leq 0.2(0.2)+|{-0.1-0.1i}|+0.05(0.2)^2\\
&=0.04+\frac{\sqrt{2}}{10}+0.002\\
&<0.184<0.2,
\end{aligned}
\]
while
\[
\begin{aligned}
|\phi_3(z)|
&\leq 0.1(0.2)+|0.1-0.1i|+0.04(0.2)^2\\
&=0.02+\frac{\sqrt{2}}{10}+0.0016\\
&<0.164<0.2.
\end{aligned}
\]
Hence
\[
\phi_e(X)\subset X,\qquad e=1,2,3.
\]

To verify the open set condition, observe that
\[
\phi_1(X)\subset
\bar{B}\left(0.1,\,0.25(0.2)+0.1(0.2)^2\right)
=\bar{B}(0.1,0.054),
\]
\[
\phi_2(X)\subset
\bar{B}\left(-0.1-0.1i,\,
0.2(0.2)+0.05(0.2)^2\right)
=\bar{B}(-0.1-0.1i,0.042),
\]
and
\[
\phi_3(X)\subset
\bar{B}\left(0.1-0.1i,\,
0.1(0.2)+0.04(0.2)^2\right)
=\bar{B}(0.1-0.1i,0.0216).
\]
Moreover,
\[
|b_1-b_2|
=\left|0.1-(-0.1-0.1i)\right|
=\frac{\sqrt{5}}{10}
>0.096=0.054+0.042,
\]
so $\phi_1(X)$ and $\phi_2(X)$ are disjoint. Similarly,
\[
|b_1-b_3|
=\left|0.1-(0.1-0.1i)\right|
=\frac{1}{10}
>0.0756=0.054+0.0216,
\]
and
\[
|b_2-b_3|
=\left|-0.1-0.1i-(0.1-0.1i)\right|
=\frac{1}{5}
>0.0636=0.042+0.0216.
\]
Thus
\[
\phi_i(X)\cap\phi_j(X)=\varnothing,
\qquad i\neq j,
\]
and consequently the open set condition holds.

We next verify the existence of a common neighborhood on which the
maps are uniformly contracting and conformal diffeomorphisms. We take
\[
W=B(0,1.2).
\]
Since $X=\bar{B}(0,0.2)$, we have $X\subset W$ and
\[
\operatorname{dist}(X,\partial W)=1.2-0.2=1.
\]
Consequently, with the notation used in the general construction,
\[
\eta_{\mathcal S}
=\operatorname{dist}(X,\partial W)=1.
\]
The derivatives are
\[
\begin{aligned}
\phi_1'(z)&=0.25i+0.2z,\\
\phi_2'(z)&=0.2i+0.1z,\\
\phi_3'(z)&=0.1+0.08z.
\end{aligned}
\]
For an arbitrary $R>0$,
\[
\sup_{|z|\leq R}|\phi_e'(z)|
\leq |a_e|+2R|c_e|.
\]
Thus the radius at which the corresponding upper bound becomes $1$
is determined by
\[
|a_e|+2R|c_e|=1.
\]
For the three maps this gives
\[
r_1=\frac{1-0.25}{2(0.1)}
=3.75,
\]
\[
r_2=\frac{1-0.2}{2(0.05)}
=8,
\]
and
\[
r_3=\frac{1-0.1}{2(0.04)}
=11.25.
\]
In particular, all three maps are contractions on the larger
neighborhood $W=B(0,1.2)$. More explicitly, for $|z|<1.2$,
\[
|\phi_1'(z)|
\leq 0.25+0.2(1.2)
=0.49<1,
\]
\[
|\phi_2'(z)|
\leq 0.2+0.1(1.2)
=0.32<1,
\]
and
\[
|\phi_3'(z)|
\leq 0.1+0.08(1.2)
=0.196<1.
\]
Hence the maps are uniformly contracting on $W$.

We also verify that $W$ is mapped into itself. For $|z|<1.2$,
\[
\begin{aligned}
|\phi_1(z)|
&\leq 0.25(1.2)+0.1+0.1(1.2)^2\\
&=0.3+0.1+0.144\\
&=0.544<1.2,
\end{aligned}
\]
and
\[
\begin{aligned}
|\phi_2(z)|
&\leq 0.2(1.2)+\sqrt{0.1^2+0.1^2}
   +0.05(1.2)^2\\
&=0.24+\frac{\sqrt{2}}{10}+0.072\\
&<0.454<1.2,
\end{aligned}
\]
while
\[
\begin{aligned}
|\phi_3(z)|
&\leq 0.1(1.2)+\sqrt{0.1^2+0.1^2}
   +0.04(1.2)^2\\
&=0.12+\frac{\sqrt{2}}{10}+0.0576\\
&<0.320<1.2.
\end{aligned}
\]
Therefore
\[
\phi_i(W)\subset W,\qquad i=1,2,3.
\]

It remains to verify that the maps are injective on the common
neighborhood $W$. We have
\[
-i\phi_1'(z)=0.25-0.2iz,
\]
\[
-i\phi_2'(z)=0.2-0.1iz,
\]
and
\[
\phi_3'(z)=0.1+0.08z.
\]
For $|z|<1.2$,
\[
\operatorname{Re}(-i\phi_1'(z))
\geq 0.25-0.2|z|
>0.25-0.24
=0.01>0,
\]
\[
\operatorname{Re}(-i\phi_2'(z))
\geq 0.2-0.1|z|
>0.2-0.12
=0.08>0,
\]
and
\[
\operatorname{Re}(\phi_3'(z))
\geq 0.1-0.08|z|
>0.1-0.096
=0.004>0.
\]
In particular,
\[
\phi_i'(z)\neq0,\qquad z\in W,\quad i=1,2,3.
\]
Moreover, these inequalities imply univalence on $W$. Indeed, $W$ is
convex. If $z_1,z_2\in W$, $z_1\neq z_2$, then
\[
\frac{\phi_1(z_1)-\phi_1(z_2)}{z_1-z_2}
=
\int_0^1
\phi_1'\bigl(z_2+t(z_1-z_2)\bigr)\,dt.
\]
Therefore
\[
\operatorname{Re}\left(
-i\frac{\phi_1(z_1)-\phi_1(z_2)}{z_1-z_2}
\right)>0,
\]
and hence $\phi_1(z_1)\neq\phi_1(z_2)$. The same argument, using
$-i\phi_2'$ and $\phi_3'$, proves that $\phi_2$ and $\phi_3$ are
univalent on $W$.

Thus each $\phi_i$ is a conformal diffeomorphism from $W$ onto its
image $\phi_i(W)$, and the three maps are uniformly contracting and
satisfy
\[
\phi_i(W)\subset W.
\]
This verifies the required common neighborhood condition. 
The bounded distortion property follows from Remark \(\ref{remark: conds 3}\), since the maps extend to conformal diffeomorphisms of the common planar neighborhood \(W\) into itself.

Finally, since the alphabet is finite, no tail bounds are needed.

\begin{figure}[h]
    \centering
    \includegraphics[width = .5\textwidth]{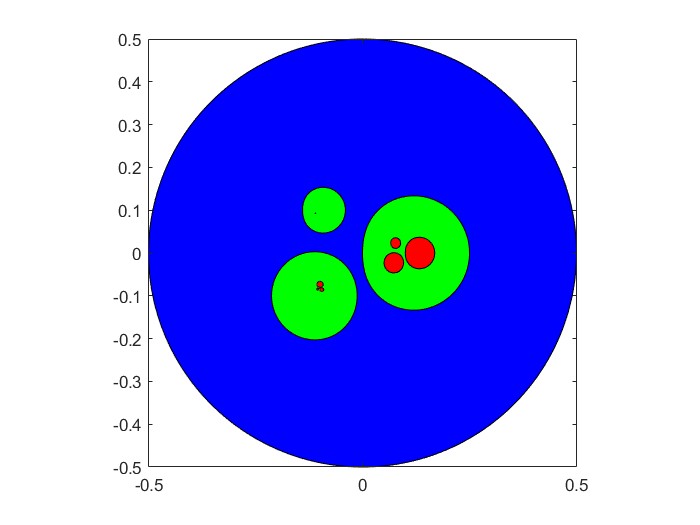}
    \caption{The first (green) and second (red) iterations of a system consisting of quadratic perturbations of linear maps.}
    \label{fig:warsaw-2it}
\end{figure}

\subsection{An application to Schottky-type conformal systems}\label{sec: schottsky justification}

We now provide examples of CGDMSs that are not IFSs and to which our method applies.
These are limit sets associated to
Schottky-type systems in $\R^2$ and $\R^3$.
The two-dimensional systems considered here arise from discrete groups of Möbius transformations.
 For the explicit systems below, we verify uniform contractivity and the remaining conformality conditions after describing their geometry. To describe the general setup of this example, suppose  $B_j$, $j =  1, 2, \dots, q$, are disjoint closed balls (disks)  in $\C$, and consider M\"obius transformations of the form
\[g_j : \hat \C \setminus  B_{j} \rightarrow B_j \text{ defined by } g_j(z) = \frac{\alpha_jz+ \beta_j}{\gamma_jz+\delta_j}.\]
 
 These maps do not form an IFS on a common compact set, since \(g_j\) has a pole in \(B_j\). They can instead be organized as a CGDMS by restricting each \(g_j\) to the balls \(B_i\), \(i\ne j\).
To formulate this system as a maximal CGDMS in the sense of
Definition~\ref{definition: graph}, let
\[
    V=\{1,\ldots,q\},
    \qquad
    X_j={B_j},
\]
and introduce the edge set
\[
    E=\{e_{j,i}:1\leq i,j\leq q,\ i\neq j\}.
\]
The map corresponding to the edge \(e_{j,i}\) is
\[
    \phi_{e_{j,i}}=g_{j,i}
      :=g_j|_{{B_i}}
      \colon {B_i}\longrightarrow {B_j}.
\]
Consequently,
\[
    t(e_{j,i})=i,
    \qquad
    i(e_{j,i})=j.
\]
Since the system is maximal, its edge incidence matrix
\(A\colon E\times E\to\{0,1\}\) is given by
\[
    A_{e_{j,i},e_{k,\ell}}=1
    \quad\Longleftrightarrow\quad
    t(e_{j,i})=i(e_{k,\ell})
    \quad\Longleftrightarrow\quad
    i=k.
\]
Thus,
\begin{equation}\label{eq:schottky-edge-incidence}
    A_{(j,i),(k,\ell)}=\mathbf{1}_{\{i=k\}}.
\end{equation}
In particular, a composition
\[
    g_{j,i}\circ g_{k,\ell}
\]
is admissible precisely when \(i=k\).

For the three-disk example, order the six edges as
\[
    e_{1,2},\ e_{1,3},\ e_{2,1},\ e_{2,3},\
    e_{3,1},\ e_{3,2}.
\]
With respect to this ordering, the edge incidence matrix is
\[
A=
\begin{pmatrix}
0&0&1&1&0&0\\
0&0&0&0&1&1\\
1&1&0&0&0&0\\
0&0&0&0&1&1\\
1&1&0&0&0&0\\
0&0&1&1&0&0
\end{pmatrix}.
\]
The corresponding vertex-transition matrix is
\[
    \widehat A=
    \begin{pmatrix}
        0&1&1\\
        1&0&1\\
        1&1&0
    \end{pmatrix}.
\]
Although \(\widehat A\) is convenient for expressing the
Perron--Frobenius operator in block form, it should not be confused
with the edge incidence matrix \(A\) appearing in
Definition~2.1. Indeed, if a function on
\[
    X=\bigcup_{i=1}^{3}{B_i}
\]
is written as \(f=(f_1,f_2,f_3)\), then the transfer operator takes
the form
\[
    (\mathcal F_t f)_i(z)
       =
       \sum_{\substack{1\leq j\leq 3\\ j\neq i}}
       |g_j'(z)|^t f_j(g_j(z)),
       \qquad z\in{B_i}.
\]
Accordingly, its block representation has the form
\[
    \mathcal F_t=
    \begin{pmatrix}
       0&\mathcal F_{12}&\mathcal F_{13}\\
       \mathcal F_{21}&0&\mathcal F_{23}\\
       \mathcal F_{31}&\mathcal F_{32}&0
    \end{pmatrix},
\]
where 
$$
(\mathcal F_{ij}u)(z)=|g_j'(z)|^t u(g_j(z)),\quad z\in {B_i},\quad i\neq j.
$$
This smaller \(3\times3\) block structure is used in our
numerical implementation.

For the four-ball example, order the twelve edges as
\[
\begin{split}
    &e_{1,2},e_{1,3},e_{1,4},
    e_{2,1},e_{2,3},e_{2,4},\\
    &e_{3,1},e_{3,2},e_{3,4},
    e_{4,1},e_{4,2},e_{4,3}.
\end{split}
\]
The corresponding edge incidence matrix is
\[
A=
\left(
\begin{array}{ccc|ccc|ccc|ccc}
0&0&0&1&1&1&0&0&0&0&0&0\\
0&0&0&0&0&0&1&1&1&0&0&0\\
0&0&0&0&0&0&0&0&0&1&1&1\\ \hline
1&1&1&0&0&0&0&0&0&0&0&0\\
0&0&0&0&0&0&1&1&1&0&0&0\\
0&0&0&0&0&0&0&0&0&1&1&1\\ \hline
1&1&1&0&0&0&0&0&0&0&0&0\\
0&0&0&1&1&1&0&0&0&0&0&0\\
0&0&0&0&0&0&0&0&0&1&1&1\\ \hline
1&1&1&0&0&0&0&0&0&0&0&0\\
0&0&0&1&1&1&0&0&0&0&0&0\\
0&0&0&0&0&0&1&1&1&0&0&0
\end{array}
\right).
\]
Its associated vertex-transition matrix is
\[
    \widehat A=
    \begin{pmatrix}
       0&1&1&1\\
       1&0&1&1\\
       1&1&0&1\\
       1&1&1&0
    \end{pmatrix}.
\]
As in the three-disk example, \(\widehat A\) describes the block
structure of the transfer operator, whereas the \(12\times12\)
matrix \(A\) is the edge incidence matrix prescribed by
Definition~2.1.

 Consider a special well-studied case of three disks of the same radius $r=\frac{1}{\sqrt{3}}$ (cf. Figures \ref{fig:Schottky3}-\ref{fig:Schottky3_mesh}) centered at 
 $$
 c_1=\frac{2}{\sqrt{3}},\quad c_2=-\frac{1}{\sqrt{3}}+i,\quad \text{and}\quad c_3=-\frac{1}{\sqrt{3}}-i,
 $$
 with the corresponding maps
 \begin{align*}
 g_1(z)&=\frac{2}{\sqrt{3}}+\frac{1}{3z-2\sqrt{3}}\\
 g_2(z)&=-\frac{1}{\sqrt{3}}+i+\frac{e^{-2\pi i/3}}{3z+\sqrt{3}-3i}\\
 g_3(z)&=-\frac{1}{\sqrt{3}}-i+\frac{e^{2\pi i/3}}{3z+\sqrt{3}+3i}.
 \end{align*}
The uniform contractivity follows from the form of the maps $g_j$ and from the fact that the distance between the centers of the disks is $2$. As a result, all 3 disks are separated by
$$
2-2r = 2-\frac{2}{\sqrt{3}} \approx 0.8452994616\dots.
$$
Hence,
$$
|g_j'(z)|=\frac{r^2}{|z-c_j|^2}\quad \Rightarrow \quad \sup_{z\in\bar{B_i}}|g_j'(z)|\le \frac{r^2}{(2-r)^2}<1\quad (i\neq j).
$$

 \begin{figure}[!htb]  
    \begin{minipage}[t]{0.48\textwidth}
        \centering
        \includegraphics[width=1.0\textwidth]{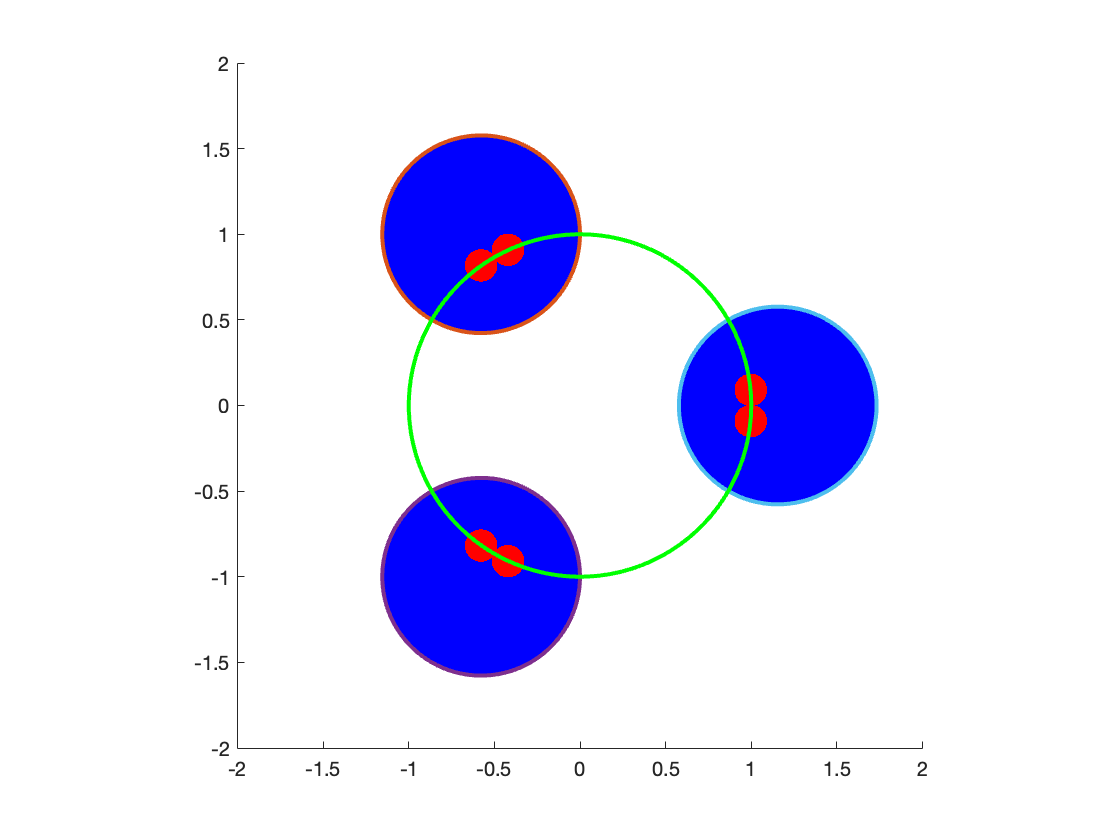}
        \caption{First iteration maps of 3 circle Schottky}
        \label{fig:Schottky3}
    \end{minipage}
  \hfill
  \begin{minipage}[t]{0.48\textwidth}
    \centering
        \includegraphics[width=1.0\textwidth]{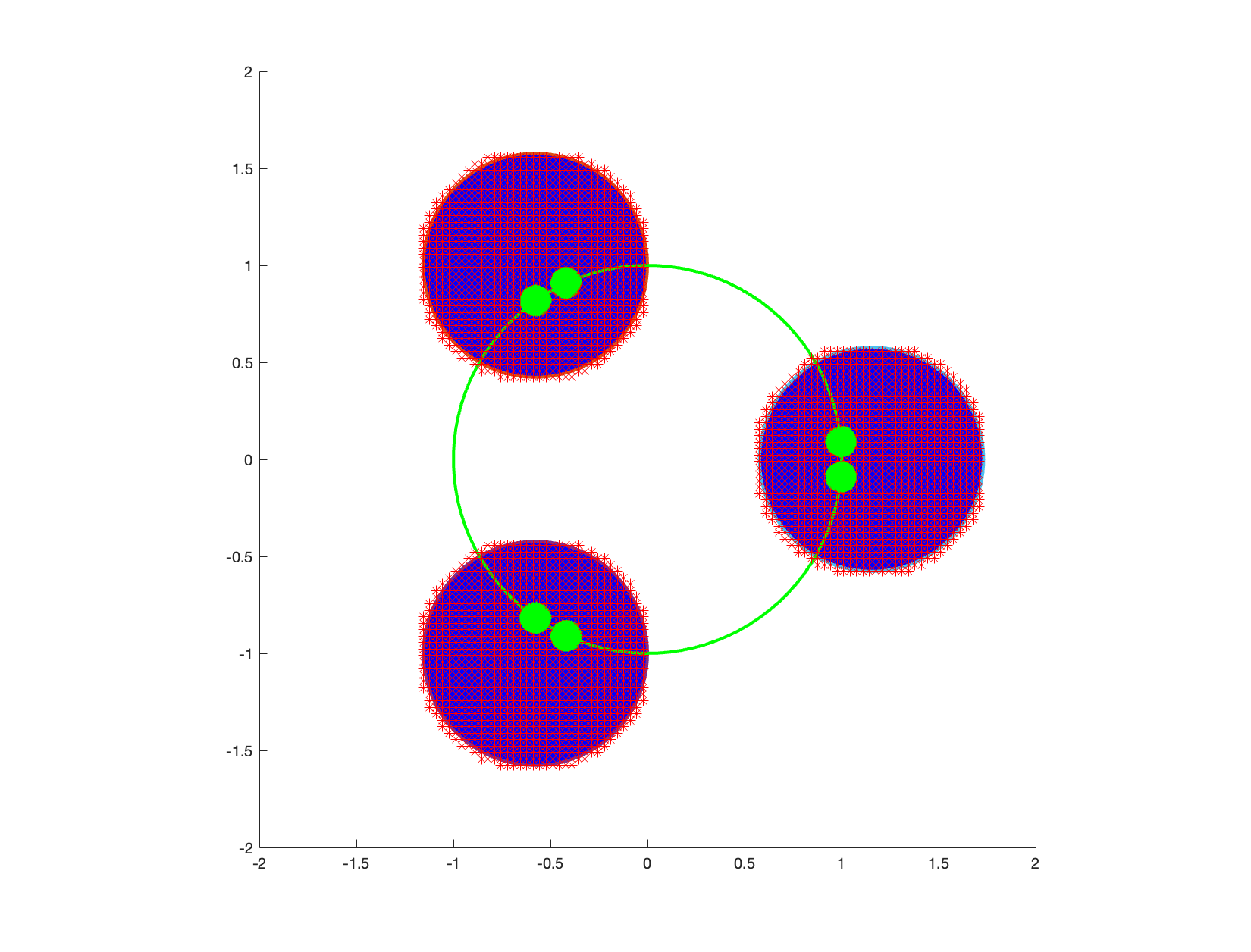}
        \caption{First iteration maps of 3 circle Schottky with mesh}
        \label{fig:Schottky3_mesh}
  \end{minipage}
\end{figure}

We note that the previous example was considered in \cite{mcmullen} and a different numerical approach for calculating its Hausdorff dimension was given in  \cite{Pollicott_Vytnova_2022}.

The previous setting can be easily adapted to $\R^3$; one just needs to replace the M\"obius transformations with inversions in spheres, given by
$$
g_j(x)= c_j+\frac{r^2}{|x-c_j|^2}(x-c_j),
$$
where $c_j$ is the center and $r$ is the radius of the inversion ball. 

 As an example, consider 4 balls of the same radius $r=1/2$ centered at the vertices of the regular tetrahedron 
$$
c_1=\alpha(1,1,1), \quad 
c_2=\alpha(1,-1,-1),\quad 
c_3=\alpha(-1,1,-1)\quad 
\text{and}\quad  
c_4=\alpha(-1,-1,1),
$$
with a scaling factor $\alpha=3/4$, see Figure \ref{fig:Schottky 3D}. One can easily show that with this choice of parameter $\alpha$ and radius $r$ the balls are disjoint. Moreover, the uniform contractivity follows from the properties of the inversion maps and the fact that, by the construction, the distance between the centers of the balls is $2\sqrt{2}\cdot 3/4=3\sqrt{2}/2 $. As a result, all four balls are separated by 
$$
\frac{3\sqrt{2}}{2}-2r = \frac{3\sqrt{2}}{2}-1=1.121320\dots.
$$
Hence
$$
\|Dg_j(x)\|=\frac{r^2}{|x-c_j|^2}\quad \Rightarrow \quad \sup_{x\in\bar{B_i}}\|Dg_j(x)\|\le \frac{(1/2)^2}{(3\sqrt{2}/2-1/2)^2}<1\quad (i\neq j).
$$
\begin{figure}[!htb]  
        \centering
        \includegraphics[width=0.8\textwidth]{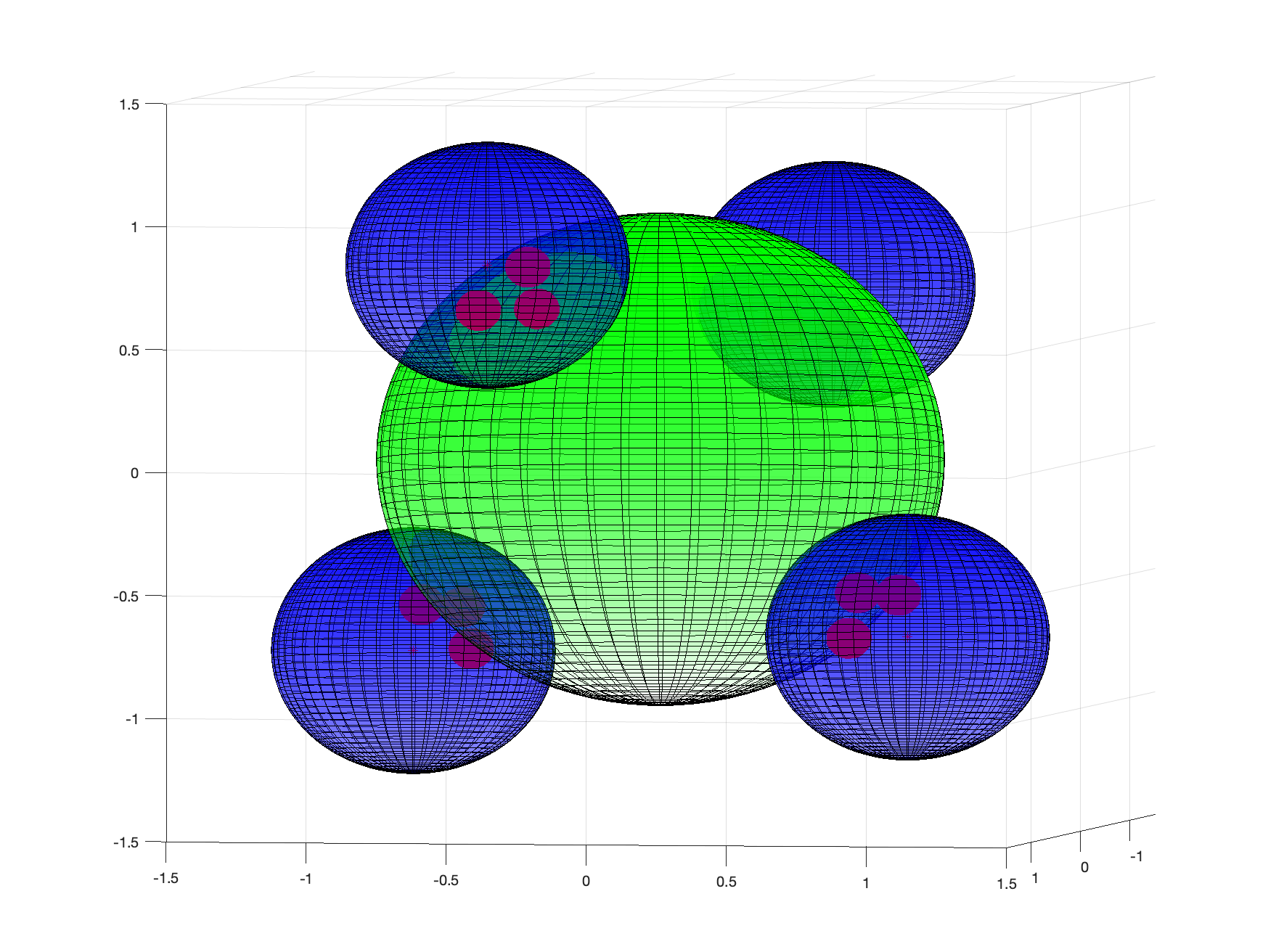}
        \caption{3D Schottky with 4 spheres, first iteration }
        \label{fig:Schottky 3D}
\end{figure}
We verify the common neighborhood condition and the open set
condition for both examples. Let \(c_i\) denote the centers,
let \(r\) be the common radius, and let \(d>2r\) be the
pairwise distance between the centers. Thus,
\[
    (r,d)=
    \begin{cases}
        (1/\sqrt{3},2), & \text{in the three-disk example},\\
        (1/2,3\sqrt{2}/2), & \text{in the four-ball example}.
    \end{cases}
\]
Choose \(R\) such that \(r<R<d/2\), and set
\[
    W_i=B(c_i,R).
\]
In both examples, the maps satisfy
\[
    |g_j(x)-c_j|=\frac{r^2}{|x-c_j|},
    \qquad x\neq c_j.
\]
Consequently, if \(x\in W_i\) and \(i\neq j\), then
\[
    |x-c_j|
    \geq |c_i-c_j|-|x-c_i|
    >d-R,
\]
and hence
\[
    |g_j(x)-c_j|
    <\frac{r^2}{d-R}
    <\frac{r^2}{R}
    <R.
\]
Therefore,
\[
    g_j(W_i)\subset W_j,
    \qquad i\neq j.
\]
Since \(c_j\notin W_i\), each restriction
\(g_j|_{W_i}\) is a conformal diffeomorphism onto its image.
Every admissible composition consequently extends to a
conformal diffeomorphism between the corresponding
neighborhoods, as required by Definition~2.5.

To verify the open set condition, consider distinct edges
\(e_{j,i}\) and \(e_{k,\ell}\). If \(j\neq k\), their images
lie in the disjoint balls \({B_j}\) and
\({B_k}\). If \(j=k\), then \(i\neq\ell\), and
injectivity of \(g_j\), together with disjointness of the
source balls, gives
\[
    g_j\bigl(\operatorname{Int}({B_i})\bigr)
    \cap
    g_j\bigl(\operatorname{Int}({B_\ell})\bigr)
    =\varnothing.
\]
Thus the open set condition holds.

Finally, the edge incidence matrix is finitely irreducible.
Indeed, for any two edges \(e_{j,i}\) and \(e_{k,\ell}\),
they can be concatenated directly if \(i=k\).
If \(i\neq k\), the edge \(e_{i,k}\) is an admissible
connecting edge, since
\[
    A_{e_{j,i},e_{i,k}}
    =
    A_{e_{i,k},e_{k,\ell}}
    =1.
\]
Together with the uniform contraction estimates and
Remark~2.6, these observations show that both examples
define finitely irreducible maximal CGDMSs.

\subsection{The Apollonian gasket}\label{sec: Apollonian theory}
We now focus our attention on one of the most famous fractals, the Apollonian packing. To fully describe the packing as the limit of a conformal IFS, suppose that $k \in \{1,2, \ldots, 6 \}$ and consider the angles
\[ \theta_k = (-1)^k \frac{2 \pi}{3 } \text{ and } \theta_k' = \frac{2 \pi k}{3} \mod 2 \pi.\]
The generators of the system then have a representation via the maps
\begin{align*}
    &f(z) = \frac{(\sqrt{3} -1)z +1}{-z + \sqrt{3}+1}, \ R_{\theta_k}, \text{ and } R_{\theta_k'}\\
\end{align*}
where $R_\theta$ is the standard complex rotation by angle $\theta$. With this notation, 
the infinite family of maps is
 $$\{\phi_{k,n}: k=1,\dots,6 \mbox{ and } n\in \N\}$$ where
\[ \phi_{k,n} = R_{\theta_k'} \circ f^n \circ R_{\theta_k} \circ f .\] 

 \begin{figure}[!htb]  
    \begin{minipage}[t]{0.48\textwidth}
        \centering
        \includegraphics[width=1.0\textwidth]{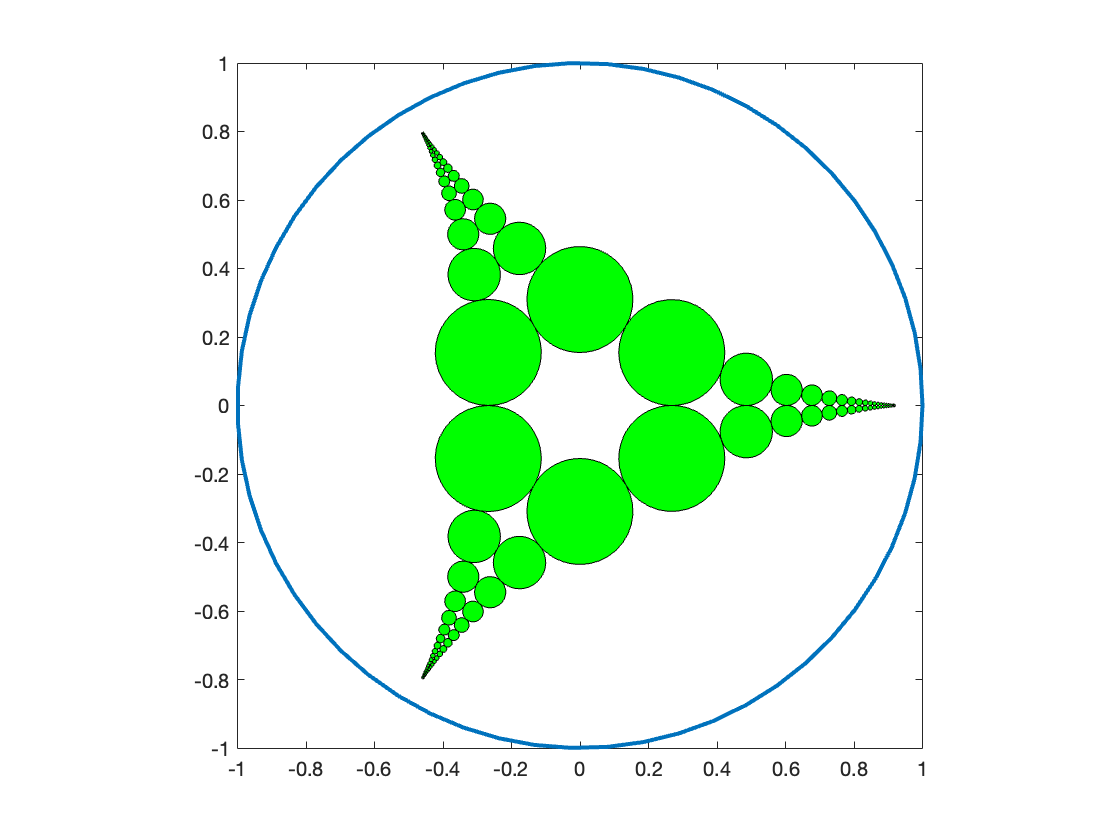}
        \caption{Apollonian. First iteration}
        \label{fig:Apollonian_1}
    \end{minipage}
  \hfill
  \begin{minipage}[t]{0.48\textwidth}
    \centering
        \includegraphics[width=1.0\textwidth]{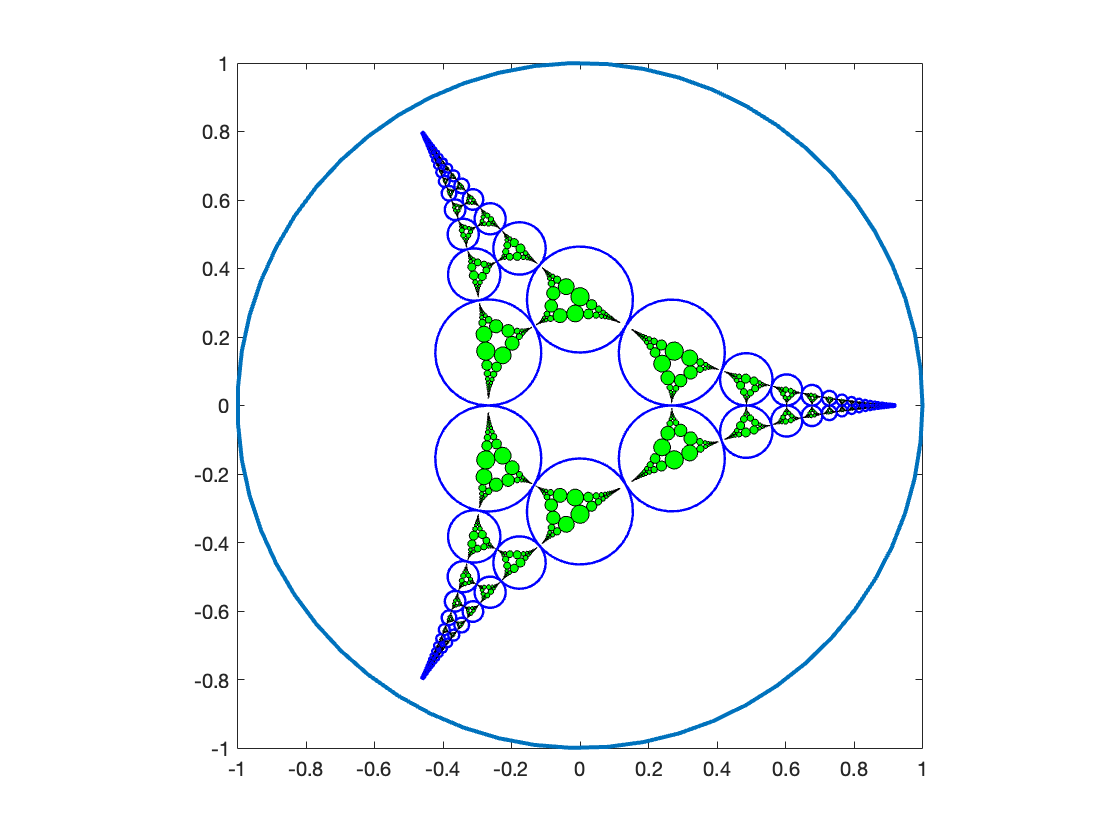}
        \caption{Apollonian. Second iteration}
        \label{fig:Apollonian_2}
  \end{minipage}
\end{figure}

Set,
\[
\D=\overline B(0,1),\qquad
\mathcal A=\{\phi_{k,n}:\D\to \D:
k=1,\ldots,6,\ n\ge1\}.
\]

We regard $\mathcal A$ as a one-vertex system with alphabet
$E=\{1,\ldots,6\}\times\mathbb N$ and full incidence:
every pair of branches may be composed.
Consequently, the system is maximal and finitely
irreducible.

By the induced-system construction of Mauldin and
Urba\'nski \cite[Section~1]{MUapol},
the limit set $J_{\mathcal A}$ equals the Apollonian
gasket with a countable set of cusp points removed.
Consequently,
\[
\dim_H(J_{\mathcal A})
=\dim_H(\text{Apollonian gasket}).
\]
For the rest of this section, we let $\lambda =\sqrt{3}$.
\begin{prop}\label{prop: apollonian common neighborhood}

The disk 
\[
X_\delta=B(0,\delta)=\{z\in\mathbb C:|z|<
\delta\}
\] 
with $\delta=1+\lambda$,
is a common invariant conformal neighborhood for the Apollonian IFS.
\end{prop}
\begin{proof}
We will show that  $\phi_{k,n}(X_\delta)\subset X_\delta$. Writing 
$$
f(z)=\frac{(\lambda-1)z+1}{-z+(\lambda+1)}, 
$$
consider the matrix representation of $f(z)$, $M$ given by 
\[
M = \begin{pmatrix}
\lambda -1 & 1 \\ -1 & \lambda+1
\end{pmatrix}. 
\]
Notice that 
\[M=VJV^{-1} = \begin{pmatrix}
-1 & 1 \\
-1 & 0
\end{pmatrix}
\begin{pmatrix}
\lambda & 1 \\ 
0 & \lambda
\end{pmatrix}
\begin{pmatrix}
0 & -1 \\
1 & -1
\end{pmatrix},\]
as $\lambda$ is the single eigenvalue of multiplicity 2 for $M$. By nilpotence
\[J^n = \left[ \lambda I + \begin{pmatrix}
0 & 1 \\
0 & 0
\end{pmatrix} \right]^n= \lambda^n I + n \lambda^{n-1}\begin{pmatrix}
0 & 1 \\
0 & 0
\end{pmatrix} = \begin{pmatrix}
\lambda^n & n \lambda^{n-1} \\
0 & \lambda^n
\end{pmatrix}\]
and so the matrix representation of $f^n(z)$ is
\[M^n = \lambda^n\begin{pmatrix}
-1 & 1 \\
-1 & 0
\end{pmatrix} \begin{pmatrix}
1 & n/ \lambda \\
0 & 1
\end{pmatrix}\begin{pmatrix}
0 & -1 \\
1 & -1
\end{pmatrix} .\]
Using this representation, and the matrix representation for the rotation
\[
R_{\theta_k} = \begin{pmatrix}
e^{i \theta_k} & 0 \\  0 & 1
\end{pmatrix} ,
\]
we see that the map
\[\phi_{k,n}(z) = R_{\theta_k'} \circ f^n \circ R_{\theta_k} \circ f(z)\]
has the matrix representation $\Phi_{k,n}$ given by
\begin{equation}\label{eq: mobdecomp}
\Phi_{k,n} = \lambda^n \begin{pmatrix}
e^{i \theta_k'} & 0 \\
0 & 1
\end{pmatrix} 
\begin{pmatrix}
-1 & 1 \\
-1 & 0
\end{pmatrix} 
\begin{pmatrix}
1 & n/\lambda \\
0 & 1
\end{pmatrix}
\begin{pmatrix}
0 & -1 \\
1 & -1
\end{pmatrix} 
\begin{pmatrix}
e^{i\theta_k} &0 \\
0 & 1 
\end{pmatrix}
\begin{pmatrix}
\lambda -1 & 1 \\
-1 & \lambda +1
\end{pmatrix} .
\end{equation}
Now we consider the action of each map on $X_\delta$. Start with $f(X_\delta)$. Since $\delta = 1+\lambda$ is a pole of $f(z)$, $f(z)$ maps the ball $X_\delta$ onto the open half-plane to the right of the vertical line
$$
L(t) = -\frac{1}{2(\lambda+1)}+it,\quad t \in \R.
$$
This is easy to see, for
\[f(-1-\lambda)= \frac{-(\lambda - 1)(\lambda+1) +1}{(\lambda+1)+\lambda+1} = -\frac{1}{2(\lambda+1)}\]
and
\[f((1+\lambda) i)=\frac{(\lambda - 1)(1+\lambda) i +1}{-(1+\lambda) i+\lambda+1}=\frac{(2i+1)(1+i)}{2(\lambda+1)} = \frac{-1+3i}{2(\lambda+1)} = -\frac{1}{2(\lambda+1)} + \frac{3i}{2(\lambda+1)} . \]
The equality follows by noticing that the real parts of both these points are equal.

This is followed by a rotation by $2\pi/3$ --- that is, finding the image after $R_{\theta_k}$. By the symmetry of the gasket maps under the complex conjugation, we will only need to consider a rotation by $2 \pi / 3$. Under the rotation $e^{2 \pi i / 3}$, the line $L(t)$ becomes
\[
\left(-\frac{1}{2}+i\frac{\lambda}{2} \right) \left(- \frac{1}{2(\lambda+1)}+it \right) = -\left( \frac{\lambda}{2} + i \frac{1}{2}\right) t +\frac{1}{4(\lambda+1)}-i\frac{\lambda}{4(\lambda+1)}.
\]
Solving for real and imaginary parts to be zero, we see that the new line $\tilde{L}(t)$ passes through the points $\frac{1}{1+\lambda}$ and $-\frac{i}{\lambda(1+\lambda)}$.  

Let
\[
    f_V(z):=\frac{1}{1-z}
\]
denote the M\"obius transformation associated with the matrix
\(V^{-1}\). This transformation maps the line
\(\widetilde L(t)\) onto a circle.
 This is a M\"obius transformation, so it maps the line $\tilde{L}(t)$ into a circle. To compute the center and the radius of this circle, notice that 
$$
\begin{aligned}
f_V(\pm\infty)&=0\\
f_V\left(\frac{1}{1+\lambda}\right)&=\alpha,\quad \text{where}\quad \alpha=\frac{1+\lambda}{\lambda},\\
f_V(-i\beta)&=\frac{1}{1+\beta^2}-i\frac{\beta}{1+\beta^2},\quad \text{where}\quad \beta = \frac{1}{\lambda(1+\lambda)}.
\end{aligned}
$$
Thus, we need to compute the center and radius of the circle passing through three points $(0,0)$, $(1+1/\lambda,0)$, and $(\frac{1}{1+\beta^2},-\frac{\beta}{1+\beta^2})$,
which is equivalent to solving a $3\times 3$ linear system with the matrix 
\[
A =  \begin{pmatrix}
0 & 0 & 1 \\
2\alpha & 0 &  1\\
\frac{2}{1+\beta^2} & -\frac{2\beta}{1+\beta^2} &  1\\
\end{pmatrix} 
\]
and the right-hand side
\[
b =  -\begin{pmatrix}
0 \\
\alpha^2\\
\frac{1}{1+\beta^2} \\
\end{pmatrix} 
\]
Solving, we obtain that the desired center of the circle is $(\frac{1+\lambda}{2\lambda},\frac{\lambda+1}{2})$ and the radius $\rho=\frac{\lambda+1}{\lambda}$.
Thus, after applying $V^{-1}R_{\theta_k}f$,
the image is the interior of the circle with center $(\frac{1+\lambda}{2\lambda},\frac{\lambda+1}{2})$ of  radius $\frac{\lambda+1}{\lambda}$.
Indeed, let \(H=R_{2\pi/3}f(X_\delta)\). The pole \(1\) of \(f_V(z)=(1-z)^{-1}\) lies outside \(\overline H\), because
\[
\operatorname{Re}(e^{-2\pi i/3}\cdot1)
=-\frac12<-\frac1{2(1+\lambda)}.
\]
Consequently, \(f_V(H)\) is the bounded component determined by the calculated circle—its interior.

The image after $\lambda^{-n}J^n$ is simply the translation by $\frac{n}{\lambda}$, corresponding to the matrix 
$$
\begin{pmatrix}
1 & n/\lambda \\
0 & 1
\end{pmatrix},
$$
 so it remains the interior of the translated circle.
Alternatively, this is the map $z\to z+n/\lambda$, which is just a translation by $n/\lambda$ and the  image  is just an open disk centred at $(\frac{1+\lambda}{2\lambda}+\frac{n}{\lambda},\frac{\lambda+1}{2})$ and radius $\rho=\frac{\lambda+1}{\lambda}$. Its boundary is parameterized by
$$
C(t) = \frac{1+\lambda}{2\lambda}+\frac{n}{\lambda}+\frac{\lambda+1}{\lambda}\cos(t)+i\left(\frac{\lambda+1}{2}+\frac{\lambda+1}{\lambda}\sin(t)\right),\quad t\in(0,2\pi).
$$
We could now proceed with the next map $1-\frac{1}{z}$, but we will use a different approach. 
Write \(\phi_{k,n}=h\circ g\), where \(g\) and \(h\) are represented by the matrices below. To prove \(\phi_{k,n}(X_\delta)\subset X_\delta\), it suffices to show that
\[
g(X_\delta)\subset h^{-1}(X_\delta).
\]
Here the map $g(z)$ is the map corresponding to the product of matrices
\[
\begin{pmatrix}
1 & n/\lambda \\
0 & 1
\end{pmatrix}
\begin{pmatrix}
0 & -1 \\
1 & -1
\end{pmatrix} 
\begin{pmatrix}
e^{i\theta_k} &0 \\
0 & 1 
\end{pmatrix}
\begin{pmatrix}
\lambda -1 & 1 \\
-1 & \lambda +1
\end{pmatrix} 
\]
and the map $h(z)$ to the product of matrices
\[
  \begin{pmatrix}
e^{i \theta_k'} & 0 \\
0 & 1
\end{pmatrix} 
\begin{pmatrix}
-1 & 1 \\
-1 & 0
\end{pmatrix} 
\]
Naturally the rotation 
 leaves $X_\delta$ invariant. Since
\[\begin{pmatrix}
-1 & 1 \\
-1 & 0
\end{pmatrix}^{-1}=\begin{pmatrix}
0 & -1 \\
1 & -1
\end{pmatrix},\]
we need to find the image of $B(0,\delta)$ under the M\"obius map $f_V(z)=-\frac{1}{z-1}.$ 
As in the preceding circle calculation, consider the images of the three points
 $(1+\lambda)$, $(-1-\lambda)$, and $i(1+\lambda)$.
$$
\begin{aligned}
f_V(1+\lambda)&=-\frac{1}{\lambda}\\
f_V(-1-\lambda)&=\frac{1}{2+\lambda}\\
f_V(i(1+\lambda))&=\frac{1}{1+\beta^2}+i\frac{\beta}{1+\beta^2},\quad \text{where}\quad \beta =1+\lambda.
\end{aligned}
$$
Thus, we need to compute the center and radius of the circle passing through three points $(-\frac{1}{\lambda},0)$, $(\frac{1}{2+\lambda},0)$, and $(\frac{1}{1+\beta^2},\frac{\beta}{1+\beta^2})$,
which is equivalent to solving a $3\times 3$ linear system with the matrix 
\[
A =  \begin{pmatrix}
-\frac{2}{\lambda} & 0 & 1 \\
\frac{2}{2+\lambda} & 0 &  1\\
\frac{2}{1+\beta^2} & \frac{2\beta}{1+\beta^2} &  1\\
\end{pmatrix} 
\]
and the right-hand side
\[
b =  -\begin{pmatrix}
\frac{1}{\lambda^2} \\
\frac{1}{(2+\lambda)^2}\\
\frac{1}{1+\beta^2} \\
\end{pmatrix} 
\]
with $\beta=1+\lambda$.
Solving, we obtain that the center of the circle is $(-\frac{1}{\lambda(2+\lambda)},0)$ and the radius is $\rho=\frac{1+\lambda}{\lambda(2+\lambda)}$. 
Thus, the preimage \(h^{-1}(X_\delta)\) is the exterior of the circle with center $(-\frac{1}{\lambda(2+\lambda)},0)$ and radius $\frac{1+\lambda}{\lambda(2+\lambda)}$.
Since \(h(1)=0\in X_\delta\), the point \(1\) belongs to \(h^{-1}(X_\delta)\). It lies outside the calculated circle. Therefore \(h^{-1}(X_\delta)\) is indeed its exterior.

To conclude $\phi_{k,n}(X_\delta)\subset X_\delta$, we only need to establish that the distance between centers of the circles $c_1=(-\frac{1}{\lambda(2+\lambda)},0)$ and $c_2=(\frac{1+\lambda}{2\lambda}+\frac{n}{\lambda},\frac{\lambda+1}{2})$ is greater than the sum of the radii $\rho_1=\frac{1+\lambda}{\lambda(2+\lambda)}$ and $\rho_2=\frac{\lambda+1}{\lambda}$.
Magically,
$$
\rho_1+\rho_2=\frac{1+\lambda}{\lambda(2+\lambda)}+\frac{\lambda+1}{\lambda}=\frac{\lambda+1}{\lambda}\cdot\frac{3+\lambda}{2+\lambda}=2
$$
and a direct calculation shows that even for $n=1$,
$$
|c_2-c_1|^2
=\frac{16-2\lambda}{3}
=4+\frac{2(2-\lambda)}3>4
=(\rho_1+\rho_2)^2
$$
and of course the above distance is even greater for $n\geq 2$. 
Consequently, 
\[
\phi_{k,n}(X_\delta)\subset X_\delta.
\]
Since it is a nonconstant Möbius transformation with no pole in \(X_\delta\), its restriction to \(X_\delta\) is a conformal diffeomorphism onto its image.

\end{proof}

\begin{figure}
    \centering
    \includegraphics[width = .5\textwidth]{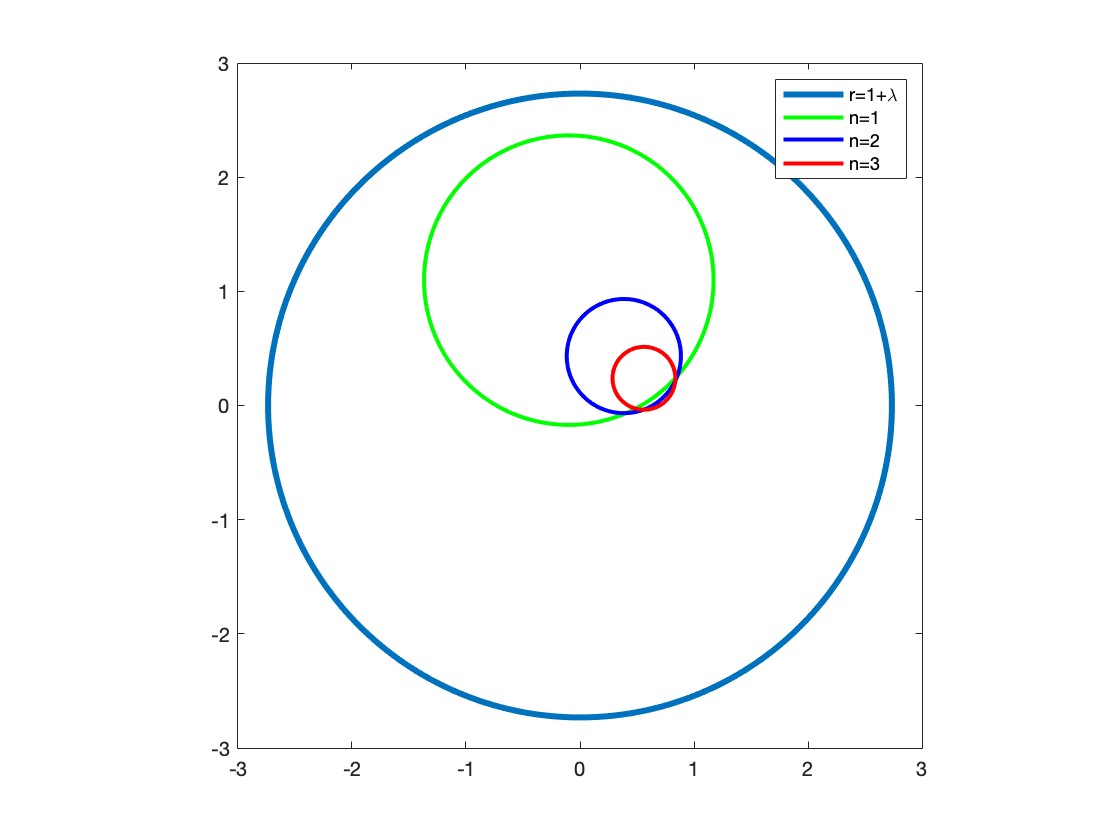}
    \caption{illustration of inclusion, for $n =1,2,3$.}
\end{figure}

\begin{lem}[Open set condition for the Apollonian system]
\label{lem: apollonian osc}
Let $\D=\overline B(0,1)$ and $U=\operatorname{Int}(\D)$.
The branches $\phi_{k,n}$ satisfy
\[
    \phi_{k,n}(U)\subset U
\]
and
\[
    \phi_{k,n}(U)\cap\phi_{\ell,m}(U)=\varnothing
    \qquad\text{whenever }(k,n)\neq(\ell,m).
\]
Consequently, the Apollonian system satisfies the open set
condition with open set $U$.
\end{lem}

\begin{proof}
Write $\lambda=\sqrt3$ and $\zeta=e^{2\pi i/3}$.
For
\[
    f(z)=\frac{(\lambda-1)z+1}{-z+\lambda+1},
\]
direct computation gives
\[
    f(U)=B(b,a),
    \qquad
    b=\frac{2}{2+\lambda},
    \qquad
    a=\frac{\lambda}{2+\lambda}.
\]
Since $b+a=1$, we have $f(U)\subset U$.
Define the three open disks
\[
    U_j=\zeta^j f(U)=B(\zeta^j b,a),
    \qquad j=0,1,2.
\]
Each $U_j$ lies in $U$. Moreover, for distinct
$j,\ell\in\{0,1,2\}$,
\[
    |\zeta^j b-\zeta^\ell b|
    =\sqrt3\,b
    =2a.
\]
Thus the disks $U_0,U_1,U_2$ are pairwise disjoint,
although their boundaries are tangent.

For $\varepsilon\in\{-1,1\}$, put
\[
    V_\varepsilon=\zeta^\varepsilon f(U).
\]
These are the two disks $U_1,U_2$, so they are disjoint
from each other and from $U_0=f(U)$.

We first show that the sets
\[
    f^n(V_\varepsilon),
    \qquad n\geq1,\quad\varepsilon\in\{-1,1\},
\]
are pairwise disjoint. The map $f$ is injective on $U$
and maps $U$ into itself, so every iterate $f^n$ is
injective on $U$.

If the iteration numbers agree and the signs differ,
injectivity gives
\[
    f^n(V_{-1})\cap f^n(V_1)=\varnothing.
\]
If $n<m$, then
\[
    f^{m-n}(V_{\varepsilon'})
    \subset f(U)=U_0.
\]
Since $V_\varepsilon\cap U_0=\varnothing$, injectivity
of $f^n$ implies
\[
    f^n(V_\varepsilon)\cap f^m(V_{\varepsilon'})
    =
    f^n(V_\varepsilon)
    \cap f^n\bigl(f^{m-n}(V_{\varepsilon'})\bigr)
    =\varnothing.
\]
This proves the claimed disjointness. Also, because $n\geq1$,
\[
    f^n(V_\varepsilon)\subset f(U)=U_0.
\]

For each $k\in\{1,\ldots,6\}$, let
\[
    j_k=k\bmod3,\qquad \varepsilon_k=(-1)^k.
\]
The branch formula gives
\[
    \phi_{k,n}(U)
    =\zeta^{j_k}f^n(V_{\varepsilon_k})
    \subset\zeta^{j_k}U_0
    =U_{j_k}\subset U.
\]
If $j_k\neq j_\ell$, the images of $\phi_{k,n}$ and
$\phi_{\ell,m}$ lie in disjoint disks $U_{j_k}$ and
$U_{j_\ell}$.

If $j_k=j_\ell$, apply the common inverse rotation.
The required disjointness follows from the pairwise
disjointness of the sets $f^n(V_\varepsilon)$ proved
above. Indeed, the correspondence
\[
    k\longmapsto(j_k,\varepsilon_k)
\]
is a bijection from $\{1,\ldots,6\}$ onto
$\{0,1,2\}\times\{-1,1\}$. Hence distinct branch indices
with the same outer rotation have different iteration
numbers or different signs.

Therefore all distinct branches have disjoint images
of $U$, proving the open set condition.
\end{proof}
For $n\geq 1$, to show contraction, we use
 the rotated-disk image formula
\[
f(\mathbb D)=\bar{B}(b,a),
\qquad
b=\frac{2}{2+\lambda},
\quad
a=\frac{\lambda}{2+\lambda}.
\]
For \(w=R_{\theta_k}f(z)\), where \(\theta_k=\pm2\pi/3\), this gives
\[
w\in\overline B\!\left(
-\frac b2\pm i\frac{\lambda b}{2},\,a
\right).
\]
Every point in this disk has real part at most its center’s real part plus its radius:
\[
\operatorname{Re}w\le-\frac b2+a
\]
and as a result
\[
1-\operatorname{Re}w
\ge1+\frac b2-a=6-3\lambda.
\]
Together with \(|f'(z)|\le1\) and
\[
|(f^n)'(w)|=\frac{3}{|\lambda+n(1-w)|^2},
\]
this yields
\[
\|\phi_{k,n}'\|_{\infty,\mathbb D}
\le\frac{3}{[\lambda+n(6-3\lambda)]^2}
\le\frac{2+\sqrt3}{8}<1.
\]   
Since $\D$ is convex, this derivative bound gives a common
Lipschitz constant strictly smaller than one.
Together with the self-mapping property, the open set
condition, the common conformal neighborhood, and
Remark~\ref{remark: conds 3}, this verifies that
$\mathcal A$ is a finitely irreducible maximal CIFS.

\subsubsection{Tail Bounds}
 In this section we find tail bounds for the Apollonian gasket. As mentioned with continued fractions, such bounds are necessary for rigorous Hausdorff dimension estimates of infinite systems, though the structure of such bounds will change depending on the system. Generally, an ordering needs to be given on the maps of the system, which in this case is given in its definition.

 Recall that any M\"obius transformation  
\[g(z) = \frac{az+b}{cz+d}\]
has a matrix representation
\[M_g = \begin{pmatrix}
a & b \\ c & d
\end{pmatrix},\]
and the norm of its derivative at $z \in \C$ is given by the formula
\begin{equation}\label{eq:mobius-deriv}
\|Dg(z)\|=|g'(z)|
=\frac{|\det M_g|}{|cz+d|^2},
\qquad cz+d\ne0.  
\end{equation}
As in the previous section, the matrix form for $\phi_{k,n}$, $n=1,2,\dots,$ is
\[
\Phi_{k,n} = \underbrace{\lambda^n \begin{pmatrix}
e^{i \theta_k'} & 0 \\
0 & 1
\end{pmatrix} 
\begin{pmatrix}
-1 & 1 \\
-1 & 0
\end{pmatrix} 
\begin{pmatrix}
1 & n/\lambda \\
0 & 1
\end{pmatrix}
\begin{pmatrix}
0 & -1 \\
1 & -1
\end{pmatrix} }_{ = \ R_{\theta_k'} \circ f^n}
\overbrace{
\begin{pmatrix}
e^{i\theta_k} &0 \\
0 & 1 
\end{pmatrix}
\begin{pmatrix}
\lambda -1 & 1 \\
-1 & \lambda +1
\end{pmatrix}}^{ = \ R_{\theta_k} \circ f} .
\]
Finding tail bounds for the system will amount to applying \eqref{eq:mobius-deriv} and the chain rule. Focusing on the rightmost matrices, note that $R_{\theta_k}$ is just a rotation by $\theta_k$, and thus leaves the derivative unchanged. Taking the determinant,
\[ 
\det  
\begin{pmatrix}
\lambda -1 & 1 \\
-1 & \lambda +1
\end{pmatrix} = \lambda^2 - 1 + 1 = \lambda^2 .
\]
The $c$ and $d$ terms for the map are $ -1$ and $\lambda+1$, respectively, so the derivative will be maximized when 
\[ |-z+\lambda+1|^2\]
is minimized. This is at $z=1$, giving the derivative $\lambda^2 / \lambda^2 =1$. Hence we have that
\[
|\phi_{k,n}'(z)|
=
|(f^n)'(R_{\theta_k}f(z))|\,|f'(z)|
\le |(f^n)'(R_{\theta_k}f(z))|,
\qquad z\in\mathbb D.
\]
We need to find $R_{\theta_k}\circ f (\D)$. Since $f$ is symmetric about the real axis, the points 
\[ f(-1) = \frac{2-\lambda}{2+\lambda} \quad \text{ and } \quad f(1) = 1\]
are antipodal points on $f(\D)$. Thus $f(\D) =\bar{B}(\frac{2}{2+\lambda},\frac{\lambda}{2+\lambda}).$ Without loss of generality, suppose that $\theta_k = \frac{2 \pi }{3}$. Then rotating $f(\D)$ by $e^{2\pi i/3}$ gives 
$$
R_{\theta_k} \circ f(\D) = \bar{B}\left(-\frac{1}{2+\lambda}+\frac{\lambda}{2+\lambda}i, \frac{\lambda}{2+\lambda}\right).
$$
Moving on to the next three maps, note that the final map is just a rotation by $\theta_k'$, and therefore doesn't change the norm of the derivative. Hence we can omit it from our calculations. Furthermore,
\[\lambda^n \begin{pmatrix}
-1 & 1 \\
-1 & 0
\end{pmatrix} 
\begin{pmatrix}
1 & n/\lambda \\
0 & 1
\end{pmatrix}
\begin{pmatrix}
0 & -1 \\
1 & -1
\end{pmatrix}   = \lambda^n\begin{pmatrix}
    - \frac{n}{\lambda} + 1 & \frac{n}{\lambda} \\
    - \frac{n}{\lambda} & \frac{n}{\lambda}+1
\end{pmatrix},\]
implying that 
\[
\det \left( \lambda^n\begin{pmatrix}
    - \frac{n}{\lambda} + 1 & \frac{n}{\lambda} \\
    - \frac{n}{\lambda} & \frac{n}{\lambda}+1
\end{pmatrix}\right) = \det \left( \begin{pmatrix}
\lambda^n & 0 \\
0 & \lambda^n \end{pmatrix}
\begin{pmatrix}
    - \frac{n}{\lambda} + 1 & \frac{n}{\lambda} \\
    - \frac{n}{\lambda} & \frac{n}{\lambda}+1
\end{pmatrix}\right)= \lambda^{2n} .
\]
Referring back to \eqref{eq:mobius-deriv}, this implies that 
\[
\|\phi_{k,n}'\|_{\infty,\mathbb D} \le \left(\frac{\lambda}{n}\right)^2\max_{z \in R_{\theta_k}\circ f(\D)} \frac{1}{\left|z-1-\frac{\lambda}{n}\right|^2}.
\]
Notice that the above maximum occurs at $z\in \bar{B}\left(-\frac{1}{2+\lambda}+\frac{\lambda}{2+\lambda}i, \frac{\lambda}{2+\lambda}\right)$ that minimizes
$\left|z-1-\frac{\lambda}{n}\right|.$

It is well known from basic complex analysis that the minimum of $|z-a|$ on the circle $|z-z_0|=r$ is attained for $$
z=a+\left(1-\frac{r}{|z_0-a|}\right)(z_0-a),
$$
with $a=1+\frac{\lambda}{n}$, $r=\frac{\lambda}{2+\lambda}$, and $z_0=-\frac{1}{2+\lambda}+\frac{\lambda}{2+\lambda}i$, we have
$$
\max_{z \in R_{\theta_k}\circ f(\D)} \frac{1}{\left|z-1-\frac{\lambda}{n}\right|^2}
=\frac{1}{\left(|z_0-a|-r\right)^2}=\frac{(2+\lambda)^2}{\left(|-1+\lambda i-(1+\lambda/n)(2+\lambda)|-\lambda\right)^2}.
$$
Using  that $\lambda=\sqrt{3}$, we compute,
$$
\|\phi_{k,n}'\|_{\infty,\mathbb D}
 \le\frac{\lambda^2}{n^2}\frac{(2+\lambda)^2}{\left(|-1+\lambda i-(1+\lambda/n)(2+\lambda)|-\lambda\right)^2}\le \frac{\lambda^2}{n^2}\frac{(2+\lambda)^2}{\left(|-1+\lambda i-2-\lambda|-\lambda\right)^2}< \frac{3 \times 1.28}{n^2}.
$$

After a simple application of the integral comparison test, one finds that for $t>1/2$ and \(N\in\mathbb N\), 
\[
\begin{aligned}
\sum_{k=1}^{6}\sum_{n=N+1}^{\infty}
\|\phi_{k,n}'\|_{\infty,\mathbb D}^{\,t}
&\le 6(3.84)^t
       \sum_{n=N+1}^{\infty}n^{-2t}\\
&\le 6(3.84)^t\int_N^\infty x^{-2t}\,dx
=\frac{6(3.84)^t}{2t-1}N^{1-2t}
<\frac{6\,4^t}{2t-1}N^{1-2t}.
\end{aligned}
\]


\section{Hausdorff dimension estimates}\label{sec: Hausdorff numerical}

For the examples introduced in the preceding section, we derive
explicit constants for the interpolation and truncation estimates
used in our numerical method. The computational domain is denoted
by \(X^h\) and depends on the example. Its triangulation satisfies
the hypotheses of \eqref{eq: interpolant}, and we set
\[
    h_\tau=\operatorname{diam}(\tau),\qquad
    h=\max_\tau h_\tau.
\]
Numerical dimension enclosures are obtained by bounding the spectral
radius of the transfer operator, including the interpolation and,
for infinite alphabets, truncation errors. Numerical evaluation
errors must also be included when certifying these bounds.

The MATLAB implementations and instructions are available at
\url{https://github.com/dmitriyleykekhman683/Rigourous-Hausdorff-dimesnions.-Examples-}. 

Floating-point computations generate candidate endpoints and
strictly positive nodal test vectors. Independent INTLAB
routines then verify mesh admissibility, the required
analytical and interpolation bounds, applicable tail
comparisons, and strict componentwise inequalities proving
the spectral-radius bounds. Mesh coordinates and test-vector
entries are interpreted as exact stored binary64 values,
and the verification uses outward-rounded interval arithmetic.

\subsection{Two-dimensional continued fractions}\label{sec: 2D cont_frac}

Let
\[
    X=\overline B(v_{1/2},1/2),\qquad v_{1/2}=(1/2,0),
    \qquad E=\{1,2,\ldots\}\times\mathbb Z.
\]
For an alphabet \(A\subset E\), we use the one-step operator
\[
    (F_{A,t}g)(x)
    =\sum_{e\in A}\frac{g(\phi_e(x))}{|x+e|^{2t}},
    \qquad \phi_e(x)=\frac{x+e}{|x+e|^2}.
\]
The preceding section establishes a uniformly contracting CIFS
with branches \(\phi_a\circ\phi_b\), \((a,b)\in A\times A\).
Its transfer operator is \(F_{A,t}^2\), and its normalized
positive leading eigenfunction is also that of \(F_{A,t}\).
We denote this common eigenfunction by \(\rho_{A,t}\) and write
\[
    F_{A,t}\rho_{A,t}=r_A(t)\rho_{A,t}.
\]
Thus the derivative estimates from the two-step CIFS apply to
the eigenfunction used by the one-step numerical implementation.
Writing \(P_{A,2}\) for the pressure of the two-step CIFS, the dimension is
\(\inf\{t\geq0:P_{A,2}(t)\leq0\}\). At summable parameters,
\(P_{A,2}(t)=2\log r_A(t)\). For regular systems this reduces to
\(r_A(t)=1\), equivalently \(r(F_{A,t}^2)=1\); the strict spectral
sign certificates used below establish regularity as well as dimension bounds.

\paragraph{Computational domain and conformal extension.}
Set
\[
    \varepsilon=0.01,\qquad
    S=\overline B(v_{1/2},1/2+\varepsilon),\qquad
    W=B(v_{1/2},1/2+0.99).
\]
The two-step neighborhood argument gives
\[
    (\phi_a\circ\phi_b)(W)\subset\operatorname{Int}(X),
    \qquad
    (\phi_a\circ\phi_b)(S)\subset X.
\]
We choose the computational domain so that
\[
    X\subset X^h\subset S\cap\{x:x_1\geq0\}.
\]
In particular, \(X^h\) contains the entire disk \(X\).
For \(x\in X^h\) and \(e\in E\), the identity
\[
    |\phi_e(x)-v_{1/2}|^2
    =\frac14+\frac{1-x_1-e_1}{|x+e|^2}
    \leq\frac14
\]
shows that \(\phi_e(X^h)\subset X\subset X^h\). Thus
every interpolation point in the one-step assembly lies in the
computational domain. These geometric inclusions are requirements
on the actual mesh, not consequences of its nominal grid spacing.

Let \(\widetilde\rho_{A,t}\) be the positive eigenfunction
extension to \(S\). The one-step eigenfunction identity continues
to hold on \(X^h\). 
Indeed, since \(\phi_b(X^h)\subset X\), for \(x\in X^h\),
\[
\begin{aligned}
r_A(t)^2\widetilde\rho_{A,t}(x)
&=\sum_{b\in A}|x+b|^{-2t}
  (F_{A,t}\rho_{A,t})(\phi_b(x))\\
&=r_A(t)\sum_{b\in A}|x+b|^{-2t}
  \rho_{A,t}(\phi_b(x)).
\end{aligned}
\]
Dividing by \(r_A(t)>0\) proves the required identity on \(X^h\).
 
 For infinite alphabets the defining sum
converges locally uniformly when \(t>1\).
The margin used for derivative estimates is
\[
    \widetilde\eta=\operatorname{dist}(S,\partial W)=0.98.
\]

\paragraph{Uniform interpolation constants.}
For finite alphabets we consider \(0\leq t\leq2\); for the
infinite alphabets below we use \(1<t\leq2\). The upper endpoint
comes from the ambient dimension and does not rely on a numerical
dimension enclosure. Corollary~\ref{cor:mobius-first-second},
applied to the two-step CIFS and its positive eigenfunction
extension, gives
\[
    |\nabla\widetilde\rho_{A,t}(x)|
    \leq\frac{2t}{\widetilde\eta}\widetilde\rho_{A,t}(x),
    \qquad
    \|D^2\widetilde\rho_{A,t}(x)\|_{\mathrm{op}}
    \leq\frac{2t(2t+1)}{\widetilde\eta^2}
           \widetilde\rho_{A,t}(x),\qquad x\in S.
\]
Thus, uniformly over the stated parameter ranges, we may take
\[
    C_0=C_\nabla=\frac4{\widetilde\eta}=\frac{200}{49},
    \qquad
    C_H=\frac{20}{\widetilde\eta^2}=\frac{50000}{2401}.
\]
Here \(C_H\) bounds the Hessian operator norm directly.
Proposition~\ref{prop: comparing values}, applied on a triangle,
gives
\[
    \sup_{y\in\tau}\widetilde\rho_{A,t}(y)
    \leq e^{C_0h_\tau}\widetilde\rho_{A,t}(x),\qquad x\in\tau.
\]
Using the operator-norm interpolation estimate from Section~\ref{sec: discretizing CX},
whose simplex factor is \(1/6\) in two dimensions, we obtain
\[
    |\widetilde\rho_{A,t}(x)-\mathcal I_h\widetilde\rho_{A,t}(x)|
    \leq\operatorname{err}_\tau\widetilde\rho_{A,t}(x),
    \qquad
    \operatorname{err}_\tau
    =\frac{C_H}{6}e^{C_0h_\tau}h_\tau^2
    =\frac{10}{3\widetilde\eta^2}
       e^{(4/\widetilde\eta)h_\tau}h_\tau^2.
\]
For each computation, the code checks the geometric mesh
requirements above and verifies, using the actual maximum
simplex diameter \(h=\max_\tau\operatorname{diam}(\tau)\), that
\[
    \max_\tau\operatorname{err}_\tau<1.
\]
Consequently,
\[
    \frac{\mathcal I_h\widetilde\rho_{A,t}(x)}
         {1+\operatorname{err}_\tau}
    \leq\widetilde\rho_{A,t}(x)
    \leq
    \frac{\mathcal I_h\widetilde\rho_{A,t}(x)}
         {1-\operatorname{err}_\tau},\qquad x\in\tau.
\]
The constants above provide a common choice for all three
examples. At a specified trial parameter \(t\), the sharper
choices \(C_0(t)=2t/\widetilde\eta\) and
\(C_H(t)=2t(2t+1)/\widetilde\eta^2\) may instead be used
consistently in the interpolation and tail estimates.
Alternatively, a uniform cap \(t_0\leq2\) may replace \(2\),
provided that all trial parameters under consideration satisfy
\(t\leq t_0\).

The dimension enclosures reported in
Table~\ref{tableintro} were verified using INTLAB.

\subsubsection{Alphabet with four smallest generators}

Consider
\[
    E_4=\{(1,0),(1,1),(1,-1),(2,0)\}.
\]
For its limit set \(J_{E_4}\), we use the one-step operator
\(F_{E_4,t}\) and the relative interpolation estimate above.
All four branches are included in the assembly, so there is no
truncation error. The dimension enclosure is obtained from
certified spectral-radius bounds at the endpoints of a trial
interval, with the interpolation corrections included.

\begin{table}[htbp]
\centering
\caption{Numerical output for the 2D continued
fraction system with four generators.}
\small
\begin{tabular}{|c|c|c|c|c|}
\hline
\(h_{\max}\) & \(N_{\mathrm{dof}}\) & interp. err. & lower dim. & upper dim.\\
\hline
\(0.0031\ldots\)&220407&\(3.4e-05\ldots\)&\(1.1495608\ldots\)&\(1.1495939\ldots\)\\
\(0.0016\ldots\)&857745&\(9.2e-06\ldots\)&\(1.1495723\ldots\)&\(1.1495820\ldots\)\\
\(0.00083\ldots\)&3404318&\(2.4e-06\ldots\)&\(1.1495751\ldots\)&\(1.1495791\ldots\)\\
\hline
\end{tabular}
\end{table}

\subsubsection{Infinite lattice alphabet}

Now let \(A=E\). For \(R\geq1\), retain the finite alphabet
\[
    E_R=\{e\in E:|e|<R+2\}.
\]
We give an explicit version of the lattice-tail estimate,
valid at all mesh nodes. Set
\[
    q_R=1-\frac{1+\varepsilon}{R+2}>0.
\]
Since \(|x|\leq1+\varepsilon\) on \(X^h\), for every omitted digit \(|e|\geq R+2\) we have
\[
\begin{aligned}
    |x+e|
    &\geq |e|-1-\varepsilon\\
    &=|e|\left(1-\frac{1+\varepsilon}{|e|}\right)
      \geq q_R|e|.
\end{aligned}
\]
Moreover,
\[
    |\phi_e(x)|=\frac1{|x+e|}
       \leq\frac1{R+1-\varepsilon}.
\]
For integer \(k\geq1\), a covering by disjoint unit squares
centered at the lattice points gives
\[
    \#\{e\in E:k\leq|e|<k+1\}
    \leq\pi\left[(k+1+1/\sqrt2)^2-(k-1/\sqrt2)^2\right]
    \leq C_{\mathrm{lat}}(k+1),
\]
where
\[
    C_{\mathrm{lat}}=2\pi(1+\sqrt2).
\]
Therefore, for \(t>1\),
\[
\begin{aligned}
    \sum_{e\in E\setminus E_R}|x+e|^{-2t}
    &\leq 2q_R^{-2t}C_{\mathrm{lat}}
          \sum_{\substack{k\in\mathbb N\\k\geq R+1}}
          k^{1-2t}\\
    &\leq 2q_R^{-2t}C_{\mathrm{lat}}
          \int_R^\infty r^{1-2t}\,dr\\
    &=\frac{q_R^{-2t}C_{\mathrm{lat}}}{t-1}
       R^{2-2t}
     =:A_t(R),\qquad x\in X^h.
\end{aligned}
\]
The eigenfunction comparison can be localized near the origin.
Since \(\phi_e(x)\in X\) and \(S\) is convex, the gradient
bound above, with \(C_0=200/49\), gives
\[
    \rho_{E,t}(\phi_e(x))
    \leq\exp\!\left(\frac{C_0}{R+1-\varepsilon}\right)\rho_{E,t}(0).
\]
It follows that the exact omitted part satisfies
\[
    \sum_{e\in E\setminus E_R}|x+e|^{-2t}
       \rho_{E,t}(\phi_e(x))
    \leq T_t(R)\rho_{E,t}(0),
\]
where
\[
    T_t(R)=A_t(R)\exp\!\left(\frac{C_0}{R+1-\varepsilon}\right).
\]
For the interpolated eigenfunction, each vertex of a triangle
containing \(\phi_e(x)\) is at distance at most
\(h+(R+1-\varepsilon)^{-1}\) from zero. Nonnegative barycentric
coordinates therefore give the corresponding bound
\[
    \sum_{e\in E\setminus E_R}|x+e|^{-2t}
       \mathcal I_h\widetilde\rho_{E,t}(\phi_e(x))\leq
    A_t(R)\exp\!\left(C_0h+\frac{C_0}{R+1-\varepsilon}\right)
       \rho_{E,t}(0).
\]

The matrix correction below uses the bound on the exact omitted
part, namely \(T_t(R)\). The additional factor \(e^{C_0h}\)
in the interpolated-tail estimate is therefore not included in
this correction. The exponent and denominator in \(A_t(R)\)
are evaluated at the actual trial parameter \(t>1\).

To specify the matrix correction, let \(x_j\) be the mesh nodes,
include zero as a node \(x_{\ell_0}=0\), and use column vectors of
nodal values. Thus row \(j\) corresponds to evaluation at \(x_j\).
For each retained branch choose a triangle \(\tau_{e,j}\)
containing \(\phi_e(x_j)\), and let \(\beta_{e,j,\ell}\)
be its barycentric interpolation weights, extended by zero
to the other mesh nodes. Define
\[
    (B_t^-)_{j\ell}
    =\sum_{e\in E_R}
       \frac{|x_j+e|^{-2t}\beta_{e,j,\ell}}
            {1+\operatorname{err}_{\tau_{e,j}}},
\]
and
\[
    (B_t^+)_{j\ell}
    =\sum_{e\in E_R}
       \frac{|x_j+e|^{-2t}\beta_{e,j,\ell}}
            {1-\operatorname{err}_{\tau_{e,j}}}
       +T_t(R)\mathbf1_{\{\ell=\ell_0\}}.
\]
For the strictly positive nodal vector
\(v_j=\widetilde\rho_{E,t}(x_j)\), these matrices satisfy
\[
    B_t^-v\leq r_E(t)v\leq B_t^+v,
    \qquad r(B_t^-)\leq r_E(t)\leq r(B_t^+).
\]
Thus the tail correction is added to the zero-node column in
every evaluation row. A transposed matrix convention gives the
transposed correction. No tail is added to the lower matrix.
The finite-alphabet case uses the same interpolation construction,
with all branches retained and with zero tail correction.

\begin{table}[htbp]
\centering
\caption{Numerical output for the 2D full lattice. }
\small
\begin{tabular}{|c|c|c|c|c|c|c|}
\hline
\(h_{\max}\)&\(N_{\mathrm{dof}}\)&interp. err.&\(R\)&\(T_{t_u}(R)\)&lower dim.&upper dim.\\
\hline
\(0.027\ldots\)&2529&\(0.0028\ldots\)&100&\(0.0067\ldots\)&\(1.8544\ldots\)&\(1.8623\ldots\)\\
\(0.016\)\ldots &9610&\(0.00097\ldots\)&200&\(0.0020\ldots\)&\(1.8554\ldots\)&\(1.8577\ldots\)\\
\(0.0075\ldots\)&37322&\(0.00020\ldots\)&300&\(0.00103\ldots\)&\(1.8556\ldots\)&\(1.8567\ldots\)\\
\hline
\end{tabular}
\end{table}

\subsubsection{Gaussian prime alphabet}

Identify \((a,b)\in\mathbb Z^2\) with \(a+ib\), and let
\[
    E_{\mathrm{prime}}
    =\{(a,b)\in E:a+ib\text{ is a Gaussian prime}\}.
\]
This defines the Gaussian prime alphabet with positive real part.
The two-step CIFS construction and its eigenfunction derivative
bounds apply to this subsystem. We therefore use the same
interpolation constants for \(1<t\leq2\), with the eigenfunction
\(\rho_{E_{\mathrm{prime}},t}\) of this system.

Retain
\[
    E_{\mathrm{prime},R}
       =\{e\in E_{\mathrm{prime}}:|e|<R+2\}.
\]
Because \(E_{\mathrm{prime}}\subset E\), the same lattice bound
\(A_t(R)\) applies to its omitted branches. Consequently,
\[
    \sum_{e\in E_{\mathrm{prime}}\setminus E_{\mathrm{prime},R}}
       |x+e|^{-2t}\rho_{E_{\mathrm{prime}},t}(\phi_e(x))
    \leq T_t(R)\rho_{E_{\mathrm{prime}},t}(0).
\]
The upper and lower matrices are constructed as above, with
the retained alphabet replaced by \(E_{\mathrm{prime},R}\).
Using the full lattice to bound this tail is conservative but
requires no additional estimate for the distribution of Gaussian
primes. Independently certified smaller upper bounds for the
dimension can subsequently be used to improve the derivative
constants and the eigenfunction comparison factors.

\begin{table}[htbp]
\centering
\caption{Numerical output for the 2D Gaussian primes. 
}
\small
\begin{tabular}{|c|c|c|c|c|c|c|}
\hline
\(h_{\max}\)&\(N_{\mathrm{dof}}\)&interp. err.&\(R\)&\(T_{t_u}(R)\)&lower dim.&upper dim.\\
\hline
\(0.027\ldots\)&2529&\(0.0028\ldots\)&200&\(0.0737\ldots\)&\(1.5085\ldots\)&\(1.5606\ldots\)\\
\(0.0176\)\ldots &9610&\(0.0011\ldots\)&1000&\(0.0209\ldots\)&\(1.5097\ldots\)&\(1.5238\ldots\)\\
\(0.0117\ldots\)&37322&\(0.0005\ldots\)&3000&\(0.0077\ldots\)&\(1.5100\ldots\)&\(1.5151\ldots\)\\
\hline
\end{tabular}
\end{table}

\subsection{\texorpdfstring{$3$-dimensional continued fractions}
{3-dimensional continued fractions}}

Let
\[
    X=\overline B(v_{1/2},1/2),\qquad
    v_{1/2}=(1/2,0,0),\qquad
    E=\{1,2,\ldots\}\times\mathbb Z^2.
\]
For an alphabet \(A\subset E\), we use the one-step operator
\[
    (F_{A,t}g)(x)
    =\sum_{e\in A}|x+e|^{-2t}g(\phi_e(x)),
    \qquad \phi_e(x)=\frac{x+e}{|x+e|^2}.
\]
As in Section~\ref{sec: 2D cont_frac}, the derivative estimates
are justified by the two-step CIFS. Its transfer operator is
\(F_{A,t}^2\), and the two operators have the same normalized
positive leading eigenfunction \(\rho_{A,t}\). Writing
\(r_A(t)=r(F_{A,t})\) and denoting the two-step pressure by \(P_{A,2}\),
the dimension is \(\inf\{t\geq0:P_{A,2}(t)\leq0\}\), with
\(P_{A,2}(t)=2\log r_A(t)\) at summable parameters. For regular systems
this reduces to \(r_A(t)=1\), equivalently \(r(F_{A,t}^2)=1\).
The strict spectral sign certificates used below also establish regularity.

\paragraph{Computational domain and parameter range.}
Set
\[
    \varepsilon=0.01,\qquad
    S=\overline B(v_{1/2},1/2+\varepsilon),\qquad
    W=B(v_{1/2},1/2+0.99).
\]
The common-neighborhood argument for continued fractions gives
\[
    (\phi_a\circ\phi_b)(W)\subset\operatorname{Int}(X),
    \qquad (\phi_a\circ\phi_b)(S)\subset X.
\]
We require the tetrahedral mesh to satisfy the hypotheses of \eqref{eq: interpolant}, together with
\[
    X\subset X^h\subset S\cap\{x:x_1\geq0\}.
\]
For every \(x\in X^h\) and \(e\in E\),
\[
    |\phi_e(x)-v_{1/2}|^2
    =\frac14+\frac{1-x_1-e_1}{|x+e|^2}\leq\frac14.
\]
Thus \(\phi_e(X^h)\subset X\subset X^h\), which ensures that the one-step interpolation points lie in the mesh. The positive eigenfunction extension \(\widetilde\rho_{A,t}\) to \(S\) satisfies the one-step eigenfunction identity on \(X^h\), by the same two-step eigenfunction calculation used in Section ~\ref{sec: 2D cont_frac}.

We use \(0\leq t\leq3\) for the finite alphabet and
\(3/2<t\leq3\) for the infinite lattice alphabet. The upper
endpoint follows from the ambient dimension, independently of
any computed enclosure. The lower restriction in the infinite
case guarantees summability of the transfer operator.

\paragraph{Derivative and interpolation constants.}
The margin for the extension estimates is
\[
    \widetilde\eta=\operatorname{dist}(S,\partial W)=0.98.
\]
Corollary~\ref{cor:mobius-first-second}, applied to the two-step
CIFS, gives
\[
    |\nabla\widetilde\rho_{A,t}(x)|
    \leq\frac{2t}{\widetilde\eta}\widetilde\rho_{A,t}(x),
    \qquad
    \|D^2\widetilde\rho_{A,t}(x)\|_{\mathrm{op}}
    \leq\frac{2t(2t+1)}{\widetilde\eta^2}
           \widetilde\rho_{A,t}(x),\qquad x\in S.
\]
Since \(t\leq3\), we may take
\[
    C_0=C_\nabla=\frac6{\widetilde\eta}=\frac{300}{49},
    \qquad
    C_H=\frac{42}{\widetilde\eta^2}=\frac{105000}{2401}.
\]
By Proposition~\ref{prop: comparing values}, for every tetrahedron
\(\tau\) and \(x\in\tau\),
\[
    \sup_{y\in\tau}\widetilde\rho_{A,t}(y)
    \leq e^{C_0h_\tau}\widetilde\rho_{A,t}(x).
\]
The operator-norm interpolation estimate from Section~\ref{sec: discretizing CX} is
\[
    \|\widetilde\rho_{A,t}-\mathcal I_h\widetilde\rho_{A,t}
      \|_{L^\infty(\tau)}
    \leq\frac3{16}h_\tau^2
       \sup_{y\in\tau}
       \|D^2\widetilde\rho_{A,t}(y)\|_{\mathrm{op}}.
\]
Consequently,
\[
    |\widetilde\rho_{A,t}(x)-\mathcal I_h\widetilde\rho_{A,t}(x)|
    \leq\operatorname{err}_\tau\widetilde\rho_{A,t}(x),
    \qquad
    \operatorname{err}_\tau
    =\frac{3C_H}{16}e^{C_0h_\tau}h_\tau^2
    =\frac{63}{8\widetilde\eta^2}
       e^{(6/\widetilde\eta)h_\tau}h_\tau^2.
\]
No additional dimension factor is needed because \(C_H\)
directly bounds the Hessian operator norm.
For each computation, the code checks the geometric mesh
requirements above and verifies, using the actual maximum
simplex diameter \(h=\max_\tau\operatorname{diam}(\tau)\), that
\[
    \max_\tau\operatorname{err}_\tau<1.
\]
Hence
\[
    \frac{\mathcal I_h\widetilde\rho_{A,t}(x)}{1+\operatorname{err}_\tau}
    \leq\widetilde\rho_{A,t}(x)
    \leq
    \frac{\mathcal I_h\widetilde\rho_{A,t}(x)}{1-\operatorname{err}_\tau},
    \qquad x\in\tau.
\]
At a specified trial parameter \(t\), the choices
\(C_0(t)=2t/\widetilde\eta\) and
\(C_H(t)=2t(2t+1)/\widetilde\eta^2\) give tighter bounds.
A smaller uniform cap \(t_0\leq3\) may likewise be used if
all trial parameters under consideration lie below that cap.

The dimension enclosures reported in
Table~\ref{tableintro} were verified using INTLAB.

\subsubsection{Alphabet with five smallest generators}

Consider
\[
    E_5=\{(1,0,0),(1,1,0),(1,-1,0),(1,0,1),(1,0,-1)\},
\]
with limit set \(J_{E_5}\). We assemble the one-step operator
using all five branches and apply the relative interpolation
corrections above. Since the alphabet is finite, no tail
estimate is needed. Certified bounds on the spectral radius
at the endpoints of a trial interval then give a dimension
enclosure.

\begin{table}[htbp]
\centering
\caption{Numerical output  for the 3D continued
fraction system with five generators.}
\small
\begin{tabular}{|c|c|c|c|c|}
\hline
\(h_{\max}\)&\(N_{\mathrm{dof}}\)&interp. err.& lower dim.& upper dim.\\
\hline
\(0.044\ldots\)&\(38783\)&$0.021\ldots$&\(1.4294\ldots\)&\(1.4676\ldots\)\\
\(0.0195\ldots\)&\(443177\)&$0.0035\ldots$&\(1.4451\ldots\)&\(1.4514\ldots\)\\
\(0.0097\ldots\)&\(3276291\)&$0.00082\ldots$&\(1.4475\ldots\)&\(1.4490\ldots\)\\
\hline
\end{tabular}
\end{table}

\subsubsection{Infinite lattice alphabet}

For \(E=\{1,2,\ldots\}\times\mathbb Z^2\), retain
\[
    E_R=\{e\in E:|e|<R+2\},\qquad R\geq1.
\]
We give an explicit lattice-tail estimate on the computational
domain. Set
\[
    q_R=1-\frac{1+\varepsilon}{R+2}>0.
\]
Since \(|x|\leq1+\varepsilon\) for \(x\in X^h\), every
omitted branch, for which \(|e|\geq R+2\), satisfies
\[
\begin{aligned}
    |x+e|
    &\geq |e|-1-\varepsilon\\
    &=|e|\left(1-\frac{1+\varepsilon}{|e|}\right)
      \geq q_R|e|.
\end{aligned}
\]
Moreover,
\[
    |\phi_e(x)|=\frac1{|x+e|}
       \leq\frac1{R+1-\varepsilon}.
\]
For integer \(k\geq1\), let
\[
    A_k=\{e\in E:k\leq|e|<k+1\}.
\]
The disjoint unit cubes centered at the points of \(A_k\)
lie in the spherical shell with radii \(k-\sqrt3/2\)
and \(k+1+\sqrt3/2\). Comparing volumes gives
\[
\begin{aligned}
    \#A_k
    &\leq\frac{4\pi}{3}
       \left[(k+1+\sqrt3/2)^3-(k-\sqrt3/2)^3\right]\\
    &\leq C_{\mathrm{lat},3}(k+1)^2,
\end{aligned}
\]
where one may take
\[
    C_{\mathrm{lat},3}
    =4\pi(1+\sqrt3)\left(1+\frac{\sqrt3}{4}\right)^2.
\]
Indeed, for \(a=k+1+\sqrt3/2\) and \(b=k-\sqrt3/2>0\),
\(a^3-b^3\leq3(a-b)a^2\), and
\(a\leq(1+\sqrt3/4)(k+1)\).
Thus, for \(t>3/2\),
\[
\begin{aligned}
    \sum_{e\in E\setminus E_R}|x+e|^{-2t}
    &\leq q_R^{-2t}C_{\mathrm{lat},3}
       \sum_{\substack{k\in\mathbb N\\k\geq R+1}}
          \frac{(k+1)^2}{k^{2t}}\\
    &\leq 4q_R^{-2t}C_{\mathrm{lat},3}
       \int_R^\infty r^{2-2t}\,dr\\
    &=\frac{4q_R^{-2t}C_{\mathrm{lat},3}}{2t-3}
       R^{3-2t}
     =:A_{3,t}(R).
\end{aligned}
\]
By the gradient bound, with \(C_0=300/49\), and convexity of
\(S\), the eigenfunction
at an omitted image point satisfies
\[
    \rho_{E,t}(\phi_e(x))
    \leq\exp\!\left(\frac{C_0}{R+1-\varepsilon}\right)\rho_{E,t}(0).
\]
Consequently, the exact omitted part is bounded by
\[
    \sum_{e\in E\setminus E_R}|x+e|^{-2t}
       \rho_{E,t}(\phi_e(x))
    \leq T_{3,t}(R)\rho_{E,t}(0),\qquad x\in X^h,
\]
where
\[
    T_{3,t}(R)
    =A_{3,t}(R)\exp\!\left(\frac{C_0}{R+1-\varepsilon}\right).
\]
For the interpolated eigenfunction, the vertices of a
tetrahedron containing \(\phi_e(x)\) are at distance at most
\(h+(R+1-\varepsilon)^{-1}\) from zero. Thus
\[
\begin{split}
    &\sum_{e\in E\setminus E_R}|x+e|^{-2t}
        \mathcal I_h\widetilde\rho_{E,t}(\phi_e(x))\\
    &\qquad\leq A_{3,t}(R)
       \exp\!\left(C_0h+\frac{C_0}{R+1-\varepsilon}\right)
       \rho_{E,t}(0).
\end{split}
\]

The correction below uses the exact-tail coefficient
\(T_{3,t}(R)\), without the extra factor \(e^{C_0h}\)
from the interpolated-tail estimate. The exponent and denominator
in \(A_{3,t}(R)\) are evaluated at the actual trial parameter
\(t>3/2\).

For the matrix bounds we use the same convention as in
Section~\ref{sec: 2D cont_frac}: nodal values form a column
vector, and row \(j\) represents evaluation at node \(x_j\).
Include zero as a mesh node \(x_{\ell_0}=0\). Let
\(B_{t,\mathrm{ret}}^-\) and \(B_{t,\mathrm{ret}}^+\)
be the retained-branch matrices, obtained by dividing each
barycentric contribution by
\(1+\operatorname{err}_{\tau}\) and
\(1-\operatorname{err}_{\tau}\), respectively, where
\(\tau\) contains the corresponding branch image.
Then set
\[
    B_t^-=B_{t,\mathrm{ret}}^-,\qquad
    B_t^+=B_{t,\mathrm{ret}}^+
       +T_{3,t}(R)\mathbf1\mathbf e_{\ell_0}^{\mathsf T},
\]
where \(\mathbf1\) is the all-ones vector and
\(\mathbf e_{\ell_0}\) is the coordinate vector selecting
the zero node. For the positive vector
\(v_j=\widetilde\rho_{E,t}(x_j)\), we have
\[
    B_t^-v\leq r_E(t)v\leq B_t^+v,
    \qquad r(B_t^-)\leq r_E(t)\leq r(B_t^+).
\]
Thus the tail is added to the zero-node column of the upper
matrix, in every evaluation row; no tail correction is added
to the lower matrix. A transposed convention gives the
transposed correction. These bounds, together with controlled numerical evaluation errors, yield rigorous dimension enclosures.

\begin{table}[htbp]
\centering
\caption{Numerical output  for the 3D full lattice. 
}
\small
\begin{tabular}{|c|c|c|c|c|c|c|}
\hline
\(h_{\max}\)&\(N_{\mathrm{dof}}\)&interp. err.&\(R\)& \(T_{3,t_u}(R)\)&lower dim.&upper dim.\\
\hline
\(0.063\ldots\)&17111&\(0.048\ldots\)&30&\(0.059\ldots\)&\(2.551\ldots\)&\(2.674\ldots\)\\
\(0.044\)\ldots &38783&\(0.021\ldots\)&40&\(0.038\ldots\)&\(2.566\ldots\)&\(2.632\ldots\)\\
\hline
\end{tabular}
\end{table}

\begin{rem}
A sharper analysis exploiting the geometry of multidimensional
real continued fraction systems, together with improved
numerical estimates for their Hausdorff dimensions, will be
presented in the forthcoming paper of Brown, Chousionis,
Leykekhman, and Urba\'nski~\cite{BCLW}.
\end{rem}

\subsection{Quadratic perturbations of linear maps}
\label{sec: abc-parameters}

To obtain rigorous upper and lower estimates using
\eqref{eq: 2sided estimates}, where
\[
    \err_\tau=\frac{C_{1,\tau}C_2}{3}h_\tau^2,
\]
we need estimates for $C_{1,\tau}$ and $C_2$ on the computational
domain.

We retain the phase space and conformal neighborhood from
Section~\ref{sec: abs},
\[
    X=\overline B(0,0.2),\qquad W=B(0,1.2),
\]
and set
\[
    S=\overline B(0,0.21).
\]
We require the computational domain to satisfy
\[
    X\subset X^h\subset S.
\]
For $|z|\leq0.21$, the formulas for the three maps give
\[
\begin{aligned}
    |\phi_1(z)|
    &\leq0.25(0.21)+0.1+0.1(0.21)^2
      =0.15691<0.2,\\
    |\phi_2(z)|
    &\leq0.2(0.21)+\frac{\sqrt2}{10}
         +0.05(0.21)^2
      <0.186<0.2,\\
    |\phi_3(z)|
    &\leq0.1(0.21)+\frac{\sqrt2}{10}
         +0.04(0.21)^2
      <0.165<0.2.
\end{aligned}
\]
Thus $\phi_i(S)\subset X$ for all three branches. In particular,
all branch images of mesh points lie in $X^h$.
The margin for the eigenfunction extension is
\[
    \widetilde\eta
    =\operatorname{dist}(S,\partial W)
    =1.2-0.21=0.99.
\]

Let $\widetilde\rho_t$ denote the positive eigenfunction
extension to $S$. By Proposition~\ref{prop: comparing values},
we may take
\[
    C_{1,\tau}=e^{C_0h_\tau},
\]
where $C_0$ and $C_2$ satisfy
\[
    |\nabla\widetilde\rho_t(x)|
       \leq C_0\widetilde\rho_t(x),
    \qquad
    \max_{|\alpha|=2}|D^\alpha\widetilde\rho_t(x)|
       \leq C_2\widetilde\rho_t(x),
    \qquad x\in S.
\]

Since this system does not consist of M\"obius transformations,
we use the extension version of the planar estimate in
Theorem~\ref{thm:derivative_bound}; see
Remark~\ref{rem: extnended}. Namely,
\begin{equation}\label{eq: from thm 3.1}
    |D^\alpha\widetilde\rho_t(x)|
    \leq \alpha!
       \left(\frac{ML}{s\widetilde\eta}\right)^{|\alpha|}
       \exp\left(
          tC_R\left(\frac{L}{L-2}\right)^2
       \right)\widetilde\rho_t(x),
    \qquad x\in S,
\end{equation}
where
\[
    0<s<R<1,\qquad M>1,\qquad L>2,
    \qquad
    C_R=\log\left(\frac{(1+R)^3}{(1-R)^5}\right).
\]
Here $R$ is a relative radius in the normalized disk used
in the distortion estimate.
Taking $R=0.199$ and $s=0.198$, we obtain
\[
    C_{0.199}=1.653935\ldots<1.66.
\]
The derivative constants depend on $t$. To avoid a circular
argument, we first obtain an independent upper bound for
the dimension.

\subsubsection{An a priori choice of $t$.}

On $X=\overline B(0,0.2)$, the derivative bounds from
Section~5.2 give
\[
    \sup_{z\in X}|\phi_1'(z)|
    \leq0.25+0.2(0.2)=0.29,
\]
\[
    \sup_{z\in X}|\phi_2'(z)|
    \leq0.20+0.1(0.2)=0.22,
\]
and
\[
    \sup_{z\in X}|\phi_3'(z)|
    \leq0.10+0.08(0.2)=0.116.
\]
Consequently, for every $q>0$,
\[
    \sum_{i=1}^3\|\phi_i'\|_{\infty,X}^{q}
    \leq0.29^q+0.22^q+0.116^q.
\]
The unique solution of
\[
    0.29^{q_0}+0.22^{q_0}+0.116^{q_0}=1
\]
is
\[
    q_0=0.692347\ldots.
\]
In particular,
\[
    Z_1(0.693)
    \leq0.29^{0.693}+0.22^{0.693}+0.116^{0.693}<1.
\]
Since
\[
    P(t)=\inf_{n\in\N}\frac{\log Z_n(t)}{n},
\]
we have
\[
    P(0.693)\leq\log Z_1(0.693)<0,
\]
and hence
\[
    \dim_H(J_{\rm abc})<0.693.
\]
We therefore choose derivative constants valid uniformly for
\[
    0<t\leq t_0,\qquad t_0=0.693.
\]

For $|\alpha|=1$, evaluation at $L=9.42$ gives
\[
    \inf_{L>2}
    L\exp\left(
       0.693\cdot1.66
       \left(\frac{L}{L-2}\right)^2
    \right)<60.16.
\]
Since \eqref{eq: from thm 3.1} holds for every $M>1$,
we may let $M\downarrow1$. It follows that
\[
    \max_{|\alpha|=1}|D^\alpha\widetilde\rho_t(x)|
    \leq
    \frac{60.16}{0.198\cdot0.99}
    \widetilde\rho_t(x),
    \qquad x\in S.
\]
Passing from coordinate derivatives to the Euclidean
gradient gives
\[
    |\nabla\widetilde\rho_t(x)|
    \leq
    \frac{\sqrt2\,60.16}{0.198\cdot0.99}
    \widetilde\rho_t(x)
    <435\widetilde\rho_t(x).
\]
Thus we may take
\[
    C_0=435,\qquad C_{1,\tau}=e^{435h_\tau}.
\]
For $|\alpha|=2$, we use $\alpha!\leq2$.
Evaluation at $L=6.685$ gives
\[
    \inf_{L>2}
    L^2\exp\left(
       0.693\cdot1.66
       \left(\frac{L}{L-2}\right)^2
    \right)<465.
\]
Again letting $M\downarrow1$ in
\eqref{eq: from thm 3.1}, we obtain
\[
    \max_{|\alpha|=2}|D^\alpha\widetilde\rho_t(x)|
    \leq
    \frac{2\cdot465}{(0.198\cdot0.99)^2}
    \widetilde\rho_t(x)
    <24205\widetilde\rho_t(x).
\]
Consequently, we may choose
\[
    C_2=24205.
\]
Combining these estimates with \eqref{eq: interpolant},
we obtain
\[
    |\widetilde\rho_t(x)
       -\mathcal I_h\widetilde\rho_t(x)|
    \leq\err_\tau\widetilde\rho_t(x),
    \qquad x\in\tau,
\]
where
\[
    \err_\tau
    =\frac{24205}{3}e^{435h_\tau}h_\tau^2.
\]
Provided $\max_\tau\err_\tau<1$, this yields
\[
    \frac{\mathcal I_h\widetilde\rho_t(x)}
         {1+\err_\tau}
    \leq\widetilde\rho_t(x)
    \leq
    \frac{\mathcal I_h\widetilde\rho_t(x)}
         {1-\err_\tau},
    \qquad x\in\tau.
\]
For example, the sufficient mesh condition $h\leq0.004$
gives
\[
    \max_\tau\err_\tau
    \leq\frac{24205}{3}e^{1.74}(0.004)^2
    <0.736<1.
\]
All three branches are retained, so no tail estimate is needed.

\begin{remark}
Once an independent certified computation establishes
$\dim_H(J_{\rm abc})<0.632$, we may replace $t_0=0.693$
by $t_0=0.632$ in the derivative estimates. This gives
smaller constants for subsequent computations.
\end{remark}

Several computational results with the above parameters are given in the table below.

\begin{table}[ht]
\caption{Hausdorff dimension estimates for the abc-example}
\label{table: Hausdorff dimension estimates for abc-example}
\centering
\vspace{.1in}
\begin{tabular}{|c|c|c|c|c|}
\hline
$h_{max}$ &    interp. error & Ndof  & lower Hausdorff dim. & upper Hausdorff dim.\\
\hline
0.00155... & 0.0382...& 140048  & 0.6097...& 0.6548...\\
0.00017... & 0.00025... & 13207022 &  0.63167... & 0.63197...\\
\hline
\end{tabular}
\end{table} 

\subsection{Schottky-type conformal systems}

We derive the constants in the interpolation estimate \eqref{eq: interpolant} for the two graph-directed systems of Section~\ref{sec: schottsky justification}. As in Section~\ref{sec: abc-parameters}, we first obtain an a priori upper bound for the Hausdorff dimension
and then choose derivative bounds valid throughout the corresponding parameter interval. We also account explicitly for the enlargement of the balls needed to accommodate the computational mesh.

\subsubsection{An a priori choice of the parameter}

Let \(q\) be the number of balls, \(r\) their common radius, and
\(d>2r\) their pairwise center distance. For both examples,
\[
    \|Dg_j(x)\|=\frac{r^2}{|x-c_j|^2}.
\]
Consequently, each restricted branch \(g_{j,i}:B_i\to B_j\),
\(i\ne j\), has Lipschitz constant at most
\[
    \kappa=\frac{r^2}{(d-r)^2}<1.
\]
There are \(q(q-1)^m\) admissible edge words of length \(m\):
there are \(q(q-1)\) choices for the first edge and \(q-1\)
choices for each subsequent edge. Thus, for \(t\geq0\),
\[
    Z_m(t)\leq q(q-1)^m\kappa^{mt},
    \qquad
    P(t)\leq\log(q-1)+t\log\kappa.
\]
In particular,
\[
    \dim_H(J)\leq\frac{\log(q-1)}{-\log\kappa}.
\]
This argument uses the admissibility constraints, rather than
treating all \(q(q-1)\) edges as freely composable.

For the three-disk example,
\[
    q=3,\qquad r=\frac1{\sqrt3},\qquad d=2,\qquad
    \kappa=\frac{1/3}{(2-1/\sqrt3)^2}.
\]
Since \(2\kappa^{0.39}<0.990<1\), we obtain
\[
    \dim_H(J_{\mathrm{Sch},2})<0.39.
\]
For the four-ball example,
\[
    q=4,\qquad r=\frac12,\qquad d=\frac{3\sqrt2}{2},\qquad
    \kappa=\frac{1/4}{(3\sqrt2/2-1/2)^2}.
\]
Since \(3\kappa^{0.47}<0.993<1\), we obtain
\[
    \dim_H(J_{\mathrm{Sch},3})<0.47.
\]
Accordingly, the constants below are chosen uniformly for
\(0\leq t\leq t_0\), with \(t_0=0.39\) in two dimensions
and \(t_0=0.47\) in three dimensions.

\subsubsection{Neighborhoods and derivative bounds on the mesh}

Set \(\varepsilon=0.01\) and define
\[
    S_i=\overline{B}(c_i,r+\varepsilon),
    \qquad W_i=B(c_i,r+a),
    \qquad
    a=\begin{cases}
        0.40,& n=2,\\
        0.55,& n=3.
    \end{cases}
\]
In both cases, \(r+a<d/2\), so the neighborhood argument of
Section~5.3 gives \(g_j(W_i)\subset W_j\) for \(i\ne j\).
Furthermore,
\[
    |g_j(x)-c_j|\leq\frac{r^2}{d-r-\varepsilon}<r,
    \qquad x\in S_i,\quad i\ne j,
\]
so \(g_j(S_i)\subset B_j\subset S_j\). We require the
component meshes to satisfy
\[
    B_i\subset X^{h}_i\subset S_i,
    \qquad X^h=\bigcup_{i=1}^{q}X^{h}_i,
\]
with simplices satisfying the hypotheses of (4.1). In particular,
the images of all mesh points under the admissible branches
lie in the corresponding target meshes. Here \(h_\tau\) denotes
the actual diameter of the simplex \(\tau\).

The relevant margin for the eigenfunction extensions on the mesh is
\[
    \widetilde\eta
    =\min_i\operatorname{dist}(S_i,\partial W_i)
    =a-\varepsilon
    =\begin{cases}
        0.39,&n=2,\\
        0.54,&n=3.
    \end{cases}
\]
Thus this margin is determined by the chosen extension
neighborhoods, not by the distance between the original balls.

By Remark~3.2, the positive extensions \(\widetilde\rho_t\)
satisfy
\[
    |D^\alpha\widetilde\rho_t(x)|
    \leq\alpha!
    \left(\frac{\sqrt n}{u\widetilde\eta}\right)^{|\alpha|}
    c(s)^t\widetilde\rho_t(x),
    \qquad x\in\bigcup_iS_i,
\]
where \(0<u<s<\sqrt2-1\) and
\(c(s)=(1-s(2+s))^{-1}\). Choose
\[
    u=0.30,\qquad s=0.31,\qquad c(s)=\frac{10000}{2839}.
\]
Since \(c(s)>1\), the following bounds hold uniformly for
\(0\leq t\leq t_0\):
\[
    |\nabla\widetilde\rho_t(x)|\leq C_0\widetilde\rho_t(x),
    \qquad
    \max_{|\alpha|=2}|D^\alpha\widetilde\rho_t(x)|
    \leq C_2\widetilde\rho_t(x),
\]
with any constants satisfying
\[
    C_0\geq\frac{n}{u\widetilde\eta}c(s)^{t_0},
    \qquad
    C_2\geq\frac{2n}{u^2\widetilde\eta^2}c(s)^{t_0}.
\]
The factor \(n\) in the first bound accounts for passage from
coordinate derivative estimates to the Euclidean gradient norm;
the factor \(2\) in the second bound bounds \(\alpha!\) for
all \(|\alpha|=2\).

\subsubsection{Interpolation constants}

On every simplex \(\tau\), Proposition~\ref{prop: comparing values} gives
\[
    \sup_{x\in\tau}\widetilde\rho_t(x)
    \leq e^{C_0h_\tau}\widetilde\rho_t(y),
    \qquad y\in\tau.
\]
Thus we may take \(C_{1,\tau}=e^{C_0h_\tau}\), and the
relative interpolation estimate becomes
\[
    |\widetilde\rho_t(x)-\mathcal {I}_h\widetilde\rho_t(x)|
    \leq\operatorname{err}_\tau\widetilde\rho_t(x),
    \qquad
    \operatorname{err}_\tau
    =\frac{n^2}{4(n+1)}C_2e^{C_0h_\tau}h_\tau^2.
\]
For \(n=2\), direct evaluation gives
\[
    \frac{2}{0.30\cdot0.39}
       \left(\frac{10000}{2839}\right)^{0.39}<28,
    \qquad
    \frac{4}{(0.30\cdot0.39)^2}
       \left(\frac{10000}{2839}\right)^{0.39}<478.
\]
We may therefore use
\[
    C_0=28,\qquad C_2=478,\qquad
    \operatorname{err}_\tau
       =\frac{478}{3}e^{28h_\tau}h_\tau^2.
\]
For \(n=3\), similarly,
\[
    \frac{3}{0.30\cdot0.54}
       \left(\frac{10000}{2839}\right)^{0.47}<34,
    \qquad
    \frac{6}{(0.30\cdot0.54)^2}
       \left(\frac{10000}{2839}\right)^{0.47}<414.
\]
Thus we may use
\[
    C_0=34,\qquad C_2=414,\qquad
    \operatorname{err}_\tau
       =\frac{9\cdot414}{16}e^{34h_\tau}h_\tau^2.
\]
Provided \(\max_\tau\operatorname{err}_\tau<1\), these choices give
\[
    \frac{\mathcal {I}_h\widetilde\rho_t(x)}{1+\operatorname{err}_\tau}
    \leq\widetilde\rho_t(x)
    \leq\frac{\mathcal {I}_h\widetilde\rho_t(x)}{1-\operatorname{err}_\tau},
    \qquad x\in\tau,
\]
which can be inserted into the matrix bounds of Section~\ref{sec: Approx perron}.
The finite alphabet requires no tail estimate.

\begin{table}[ht]
\caption{Hausdorff dimension estimates for the 2D-Schottky type system}
\label{table_Shottky_2D}
\centering
\vspace{.1in}
\begin{tabular}{|c|c|c|c|c|}
\hline
$h$ &   interp. error & Ndof  & lower Hausdorff dim. & upper Hausdorff dim.\\
\hline
0.00322... & 0.00181...& 858429  & 0.29477...& 0.29632...\\
0.00164 ... & 0.00045... & 3349893 & 0.29535... & 0.29574...\\
0.000854... & 0.000119... & 13271958 &  0.295495... & 0.295597...\\
\hline
\end{tabular}
\end{table} 

\begin{table}[ht]
\caption{Hausdorff dimension estimates for the 3D-Schottky type system}
\label{table_Shottky_3D}
\centering
\vspace{.1in}
\begin{tabular}{|c|c|c|c|c|}
\hline
$h$ &   interp. error & Ndof  & lower Hausdorff dim. & upper Hausdorff dim.\\
\hline
0.0099... & 0.0321...& 12362964  & 0.37950...& 0.402332...\\
0.0074... & 0.0165... & 29616132 & 0.384904... & 0.39666...\\
\hline
\end{tabular}
\end{table} 

\subsection{The Apollonian gasket}
We use the notation and parameters from 
Section \ref{sec: Apollonian theory},  the maps \(\phi_{k,n}\) with
\(k\in\{1,\ldots,6\}\) and \(n\geq1\), and
\(\lambda=\sqrt3\). Proposition~\ref{prop: apollonian common neighborhood} establishes the common invariant conformal neighborhood
\[
    W=B(0,1+\lambda)
\]
of the phase space \(\D=\overline B(0,1)\). 

However, we need to establish the interpolation estimates not  on $\D$ but rather on the computational domain $X^h\supset \D$. 
Assume \(h\leq0.01\), and choose
\[
    h\leq\delta\leq0.01.
\]
Set
\[
    R=1+\delta,\qquad \D_\delta=\overline B(0,1+\delta).
\]
We require the triangulated computational domain $X^h$ to satisfy
\[
    \D\subset X^h\subset \D_\delta,
    \qquad \phi_{k,n}(X^h)\subset X^h
\]
for all branches of the system under consideration. 

To account for evaluation at mesh points outside \(\D\), we repeat
the calculation of Section~\ref{sec: Apollonian theory} on \(\D_\delta\). Recall that
\[
    f(z)=\frac{(\lambda-1)z+1}{-z+\lambda+1},
    \qquad
    \phi_{k,n}=R_{\theta_k'}\circ f^n\circ R_{\theta_k}\circ f.
\]
The image of \(\D_\delta\) under \(f\) is the closed disk with real center \(b_\delta\) and radius
\(a_\delta\), where
\[
    b_\delta=\frac{\lambda+1+(\lambda-1)(1+\delta)^2}{(\lambda+1)^2-(1+\delta)^2},
    \qquad
    a_\delta=\frac{3(1+\delta)}{(\lambda+1)^2-(1+\delta)^2}.
\]
Also,
\[
    \sup_{z\in \D_\delta}|f'(z)|
    =F_\delta:=\frac{3}{(\lambda-\delta)^2}.
\]
After rotation by either \(2\pi/3\) or \(-2\pi/3\), the
image disk has center \(-b_\delta/2\pm i\lambda b_\delta/2\) and
radius \(a_\delta\). Since
\[
    |(f^n)'(w)|=\frac{3}{|n(1-w)+\lambda|^2},
\]
the chain rule gives
\[
    \|\phi_{k,n}'\|_{\infty,\D_\delta}
    \leq\frac{3F_\delta}
    {\left(\sqrt{(n(1+b_\delta/2)+\lambda)^2
                         +3n^2b_\delta^2/4}-na_\delta\right)^2}.
\]
In particular,
\[
    \|\phi_{k,n}'\|_{\infty,\D_\delta}
    \leq\frac{3F_\delta}{(n(1+b_\delta/2-a_\delta)+\lambda)^2}
    \leq\frac{3F_\delta}{(1+b_\delta/2-a_\delta+\lambda)^2}.
\]
For example, for \(\delta_*=0.01\), computations give
\[
a_{\delta_*}<0.471,\qquad
b_{\delta_*}>0.539,\qquad
F_{\delta_*}<1.012.
\]
Consequently,
\[
1+\frac{b_{\delta_*}}2-a_{\delta_*}>0.7985
\]
and
\[
\frac{3F_{\delta_*}}
{(1+b_{\delta_*}/2-a_{\delta_*}+\sqrt3)^2}
<
\frac{3(1.012)}{(0.7985+\sqrt3)^2}
<0.475<\frac12.
\]
Consequently, for every \(n\geq1\) and
\(0<\delta\leq\delta_*\),
\[
    \|\phi_{k,n}'\|_{\infty,\D_\delta}
    \leq
    \|\phi_{k,n}'\|_{\infty,\D_{\delta_*}}
    <\frac12.
\]
 Since the branches map \(\D\) into
\(\D\), comparison with the radial projection onto \(\D\) gives
\[
    |\phi_{k,n}(z)|\leq 1+\tfrac12\delta<1+\delta,
    \qquad z\in \D_\delta\setminus \D.
\]
Thus \(\phi_{k,n}(\D_\delta)\subset \D_\delta\), as required for the eigenfunction extensions in Remark~\ref{rem: extnended}.

For the tail, set
\[
    L_\delta=\sqrt{(1+b_\delta/2)^2+3b_
    \delta^2/4}-a_\delta.
\]
The preceding derivative estimate yields
\[
    \|\phi_{k,n}'\|_{\infty,\D_\delta}
    \leq\frac{K_\delta}{n^2},
    \qquad K_\delta=\frac{3F_\delta}{L_\delta^2}.
\]
For \(\delta_*=0.01\), we have
\(F_{\delta_*}<1.012\) and \(L_{\delta_*}>0.883\), whence
\[
    K_{\delta_*}
    <\frac{3(1.012)}{(0.883)^2}<4.
\]
Since \(\D_\delta\subset \D_{\delta_*}\), it follows that
\[
    \|\phi_{k,n}'\|_{\infty,\D_\delta}
    \leq
    \|\phi_{k,n}'\|_{\infty,\D_{\delta_*}}
    \leq\frac{K_{\delta_*}}{n^2}
    <\frac4{n^2}.
\]
Consequently, for \(t>1/2\) and \(N\geq1\),
\[
    \sum_{k=1}^{6}\sum_{n=N+1}^{\infty}
       \|\phi_{k,n}'\|_{\infty,\D_\delta}^{t}
    \leq 6\,4^t\sum_{n=N+1}^{\infty}n^{-2t}
    \leq\frac{6\,4^t}{2t-1}N^{1-2t}.
\]
Next we address the interpolation estimates.
The margin for the eigenfunction extension on \(\D_\delta\) is
\[
    \widetilde\eta=\operatorname{dist}(\D_\delta,\partial W)
    =\lambda-\delta\geq  \sqrt3-0.01.
\]
For the infinite systems, we consider \(1/2<t\leq2\);
for the finite subsystem, we consider \(0\leq t\leq2\).
The upper endpoint follows since all the limit sets lie in \(\mathbb R^2\).

Let \(\widetilde\rho_t\) denote the positive extension of the
eigenfunction to \(\D_\delta\). Remark~\ref{rem: extnended} gives
\[
    |D^\alpha\widetilde\rho_t(x)|
    \leq\alpha!
       \left(\frac{\sqrt2}{u\widetilde\eta}\right)^{|\alpha|}
       c(s)^t\widetilde\rho_t(x),
    \qquad c(s)=(1-s(2+s))^{-1},
\]
for \(0<u<s<\sqrt2-1\). We choose \(u=0.30\),
\(s=0.31\), so that \(c(s)=10000/2839\). Uniformly for
the parameter ranges above,
\[
    |\nabla\widetilde\rho_t(x)|\leq C_0\widetilde\rho_t(x),
    \qquad
    \max_{|\alpha|=2}|D^\alpha\widetilde\rho_t(x)|
       \leq C_2\widetilde\rho_t(x),
\]
where we may take
\[
    C_0=49,\qquad C_2=186.
\]
Indeed,
\[
    \frac{2}{0.30(\sqrt3-0.01)}
       \left(\frac{10000}{2839}\right)^2<49,
    \qquad
    \frac{4}{[0.30(\sqrt3-0.01)]^2}
       \left(\frac{10000}{2839}\right)^2<186.
\]
The first estimate includes the factor \(\sqrt2\) needed to
pass from coordinate derivative bounds to the Euclidean gradient
norm. The second uses \(\alpha!\leq2\) for \(|\alpha|=2\).

Proposition~\ref{prop: comparing values}, applied within each triangle, gives
\[
    \sup_{y\in\tau}\widetilde\rho_t(y)
    \leq e^{49h_\tau}\widetilde\rho_t(x),\qquad x\in\tau.
\]
Combining this with (4.1), we obtain
\[
    |\widetilde\rho_t(x)-\mathcal {I}_h\widetilde\rho_t(x)|
    \leq\operatorname{err}_\tau\widetilde\rho_t(x),
    \qquad
    \operatorname{err}_\tau=62e^{49h_\tau}h_\tau^2.
\]
Since \(h\leq0.01\), we have
\[
    \max_\tau\operatorname{err}_\tau
    \leq62e^{0.49}10^{-4}<0.011<1.
\]
Consequently,
\[
    \frac{\mathcal {I}_h\widetilde\rho_t(x)}{1+\operatorname{err}_\tau}
    \leq\widetilde\rho_t(x)
    \leq\frac{\mathcal {I}_h\widetilde\rho_t(x)}{1-\operatorname{err}_\tau},
    \qquad x\in\tau.
\]

\paragraph{Weighted tails.}
For a nonnegative piecewise-linear function \(v_h\), define
\[
    (\mathcal T_{t,N}v_h)(x)
    =\sum_{k=1}^{6}\sum_{n=N+1}^{\infty}
       |\phi_{k,n}'(x)|^t v_h(\phi_{k,n}(x)).
\]
The mesh invariance requirement and nonnegativity imply
\[
    0\leq(\mathcal T_{t,N}v_h)(x)
    \leq\frac{6\,4^t}{2t-1}N^{1-2t}
       \|v_h\|_{L^\infty(X^h)},\qquad x\in X^h.
\]
For piecewise-linear functions the norm on the right is the maximum nodal value. In particular, this bound applies to
\(v_h=\mathcal {I}_h\widetilde\rho_t\).
We can also bound the interpolated eigenfunction in terms of \(\rho_t(0)\). Using the convexity of
\(\D_\delta\) and Proposition~\ref{prop: comparing values}, we also obtain
\[
    \|\mathcal {I}_h\widetilde\rho_t\|_{L^\infty(X^h)}
    \leq e^{49R}\rho_t(0).
\]
Let \[
\beta_{t,N}=\frac{6\,4^t}{2t-1}N^{1-2t}.
\]
We use a localized comparison near the three cusp points
\(c_j=e^{2\pi i j/3}\), \(j=0,1,2\).
Choose mesh reference nodes \(p_j=x_{\ell_j}\), and let
\(d_j\geq|p_j-c_j|\) be verified distance bounds.
The omitted branch images satisfy
\[
    |\phi_{k,n}(x)-c_j|\leq\frac{\sqrt3}{n},
    \qquad j=k\bmod3.
\]
With \(w=R_{\theta_k}f(x)\), the earlier disk estimates give \(\Re w<1\). Therefore
\[
f^n(w)-1
=-\frac{\sqrt3(1-w)}{\sqrt3+n(1-w)},
\qquad
|f^n(w)-1|\le\frac{\sqrt3}{n}.
\]
Consequently, every vertex of a triangle containing such
an image is within distance
\(h+\sqrt3/(N+1)+d_j\) of \(p_j\).
The gradient bound and nonnegative barycentric coordinates
therefore give
\[
    \mathcal I_h\widetilde\rho_t(\phi_{k,n}(x))
    \leq
    \exp\!\left(
        49\left[h+\frac{\sqrt3}{N+1}+d_j\right]
    \right)\widetilde\rho_t(p_j).
\]
Let \(m_j\) denote the number of branch families associated
with cusp \(c_j\). Define
\[
    a_j=
    \frac{m_j4^tN^{1-2t}}
         {(2t-1)(1-\operatorname{err})}
    \exp\!\left(
        49\left[h+\frac{\sqrt3}{N+1}+d_j\right]
    \right).
\]
An upper comparison matrix is then
\[
    \widehat B_t
    =
    \widetilde B_t+
    \sum_{j=0}^{2}
    a_j\mathbf1\mathbf e_{\ell_j}^{\mathsf T}.
\]
For the full system \(m_j=2\); for the subsystem with
\(k\in\{1,3,5\}\), \(m_j=1\).

In the numerical tables, \(t_u\) denotes the upper endpoint
of the dimension enclosure in the corresponding row.
The tail diagnostics are evaluated at \(t=t_u\).

\begin{table}[ht]
\caption{Hausdorff dimension estimates for the Apollonian gasket}
\label{table_Apollonian_full}
\centering
\vspace{.1in}
\begin{tabular}{|c|c|c|c|c|c|c|}
\hline
$h_{max}$ & Ndof  &  interp. err. &  $N$ & $\beta_{t_u,N}$ &  lower Haus. dim. & upper Haus dim.\\
\hline
0.0077\ldots &  138977 & 0.0054\ldots& 100 & 0.0114\ldots & 1.3030\ldots& 1.3246\ldots\\
0.0064\ldots &  217090 & 0.00347\ldots& 200 & 0.0042\ldots  & 1.3042\ldots& 1.3103\ldots\\
0.0052\ldots &  308729 & 0.00216\ldots& 300 & 0.0022\ldots  & 1.30479\ldots& 1.30804\ldots\\
\hline
\end{tabular}
\end{table}

\subsubsection{A finite Apollonian subsystem}

Consider
\[
    \mathcal A_{12}
    =\{\phi_{k,n}:k=1,\ldots,6,\ n=1,2\},
\]
with limit set \(J_{\mathcal A_{12}}\). The same neighborhoods,
mesh requirements, and interpolation constants apply to the
positive eigenfunction of this subsystem, uniformly for
\(0\leq t\leq2\). All twelve maps are retained in the
discretized operator, so no tail estimate is required.

\begin{figure}[!htb]  
    \begin{minipage}[t]{0.48\textwidth}
        \centering
        \includegraphics[width = 0.9\textwidth]{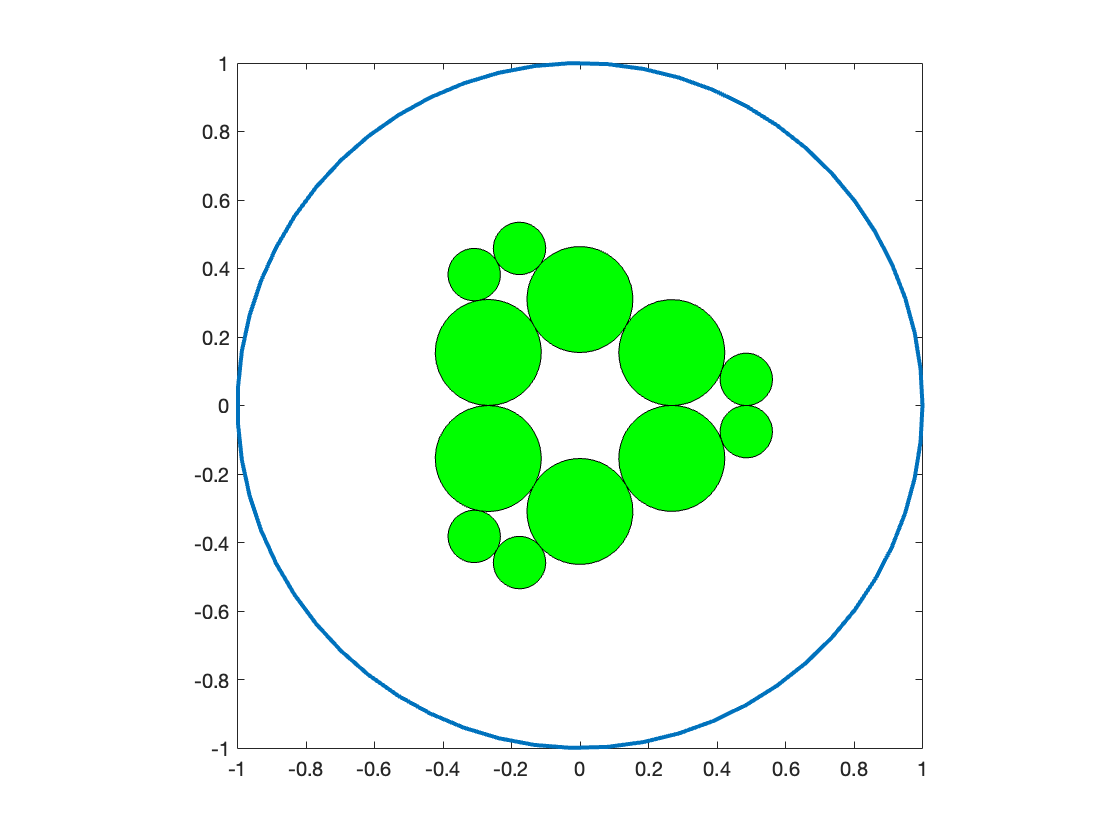}
    \caption{The Apollonian gasket with the first twelve generators, first iteration}
    \label{fig:Apollonian12_1}
    \end{minipage}
  \hfill
  \begin{minipage}[t]{0.48\textwidth}
    \centering
\includegraphics[width = .9\textwidth]{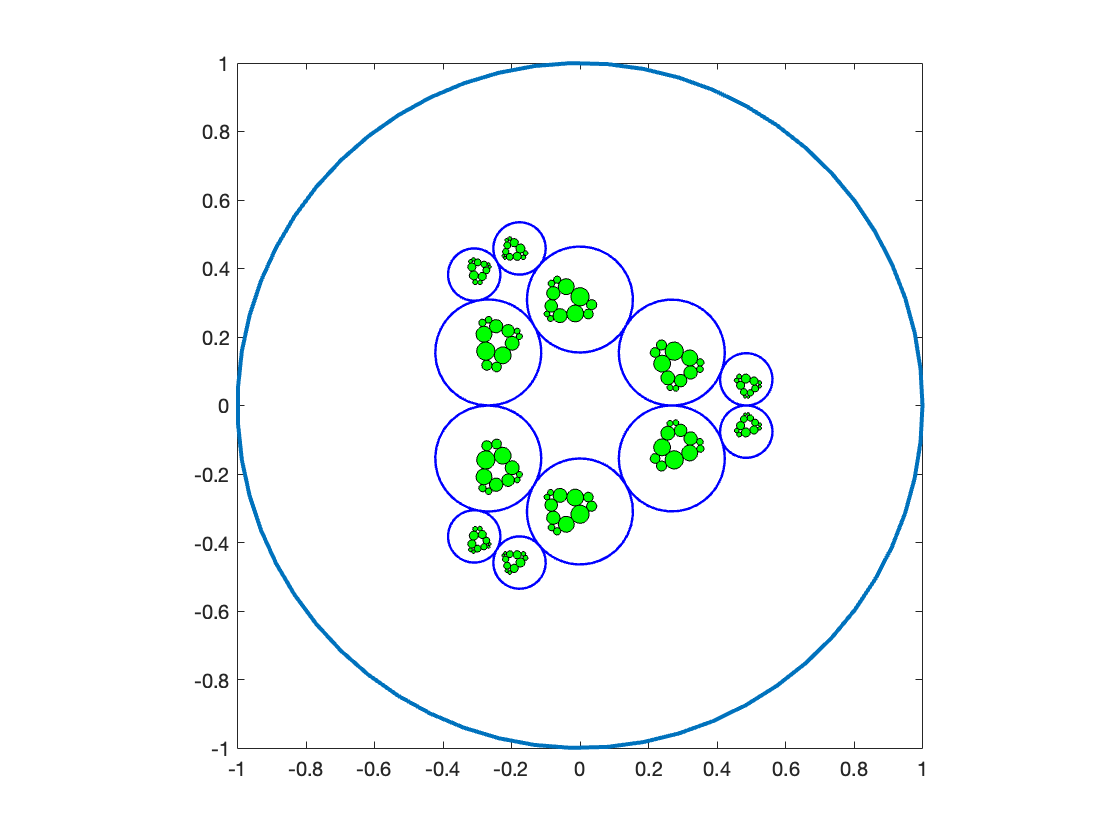}
    \caption{The Apollonian gasket with the first twelve generators, second iteration}
    \label{fig:Apollonian12_2}
  \end{minipage}
\end{figure}

\begin{table}[ht]
\caption{Hausdorff dimension estimates for  the finite Apollonian subsystem}
\label{table_Apollonian_finite}
\centering
\vspace{.1in}
\begin{tabular}{|c|c|c|c|c|}
\hline
$h_{max}$ & Ndof  &  interp. err. &  lower Haus. dim. & upper Haus dim.\\
\hline
0.0077... &  138977 & 0.0054... & 1.11154...& 1.11658...\\
\hline
0.0052... &  308729 & 0.00216...& 1.11305...& 1.11506...\\
\hline
0.0033... &  845407 & 0.00079...& 1.11368...& 1.11442...\\
\hline
\hline
\end{tabular}
\end{table}

\subsubsection{The subsystem with three infinite families of branches}

Let
\[
    \mathcal A_{\mathrm{odd}}
    =\{\phi_{k,n}:k\in\{1,3,5\},\ n\geq1\},
\]
with limit set \(J_{\mathcal A_{\mathrm{odd}}}\).
The same interpolation estimates apply to its eigenfunction.
If we retain the branches with \(n\leq N\), the omitted part
satisfies the following bound for \(1/2<t\leq2\):
\[
    \sum_{k\in\{1,3,5\}}\sum_{n=N+1}^{\infty}
       |\phi_{k,n}'(x)|^t v_h(\phi_{k,n}(x))
    \leq\frac{3\,4^t}{2t-1}N^{1-2t}
       \|v_h\|_{L^\infty(X^h)}.
\]
Thus the scalar tail coefficient is halved, while the same
conservative derivative and interpolation constants can be used.
For this subsystem, define
\[
    \beta^{\mathrm{odd}}_{t,N}
    :=\frac12\beta_{t,N}
    =\frac{3\,4^t}{2t-1}N^{1-2t}.
\]
Once an independent certified dimension upper bound is available
for any of these systems, it can replace the exponent \(2\)
in the derivative estimates to improve the constants.

\begin{figure}[!htb]  
    \begin{minipage}[t]{0.48\textwidth}
        \centering
        \includegraphics[width = 0.9\textwidth]{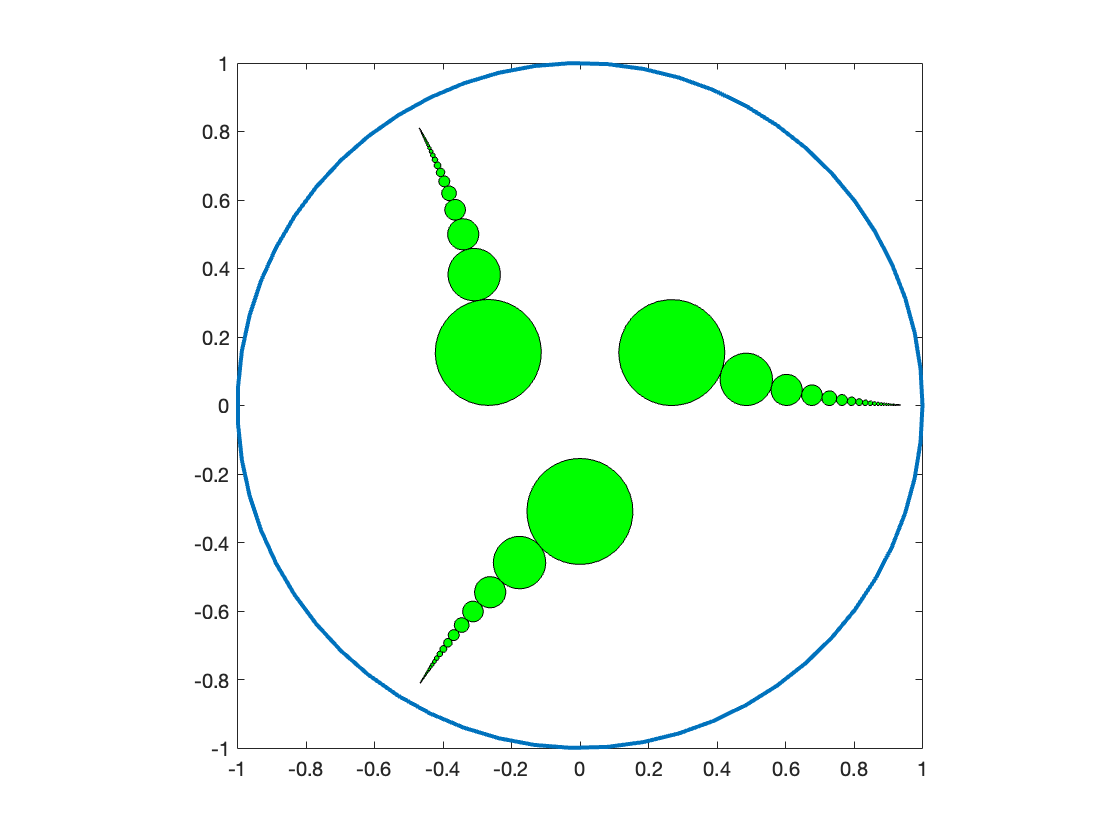}
    \caption{The Apollonian gasket with three infinite families of branches, first iteration}
    \label{fig:Apollonian3_1}
    \end{minipage}
  \hfill
  \begin{minipage}[t]{0.48\textwidth}
    \centering
\includegraphics[width = .9\textwidth]{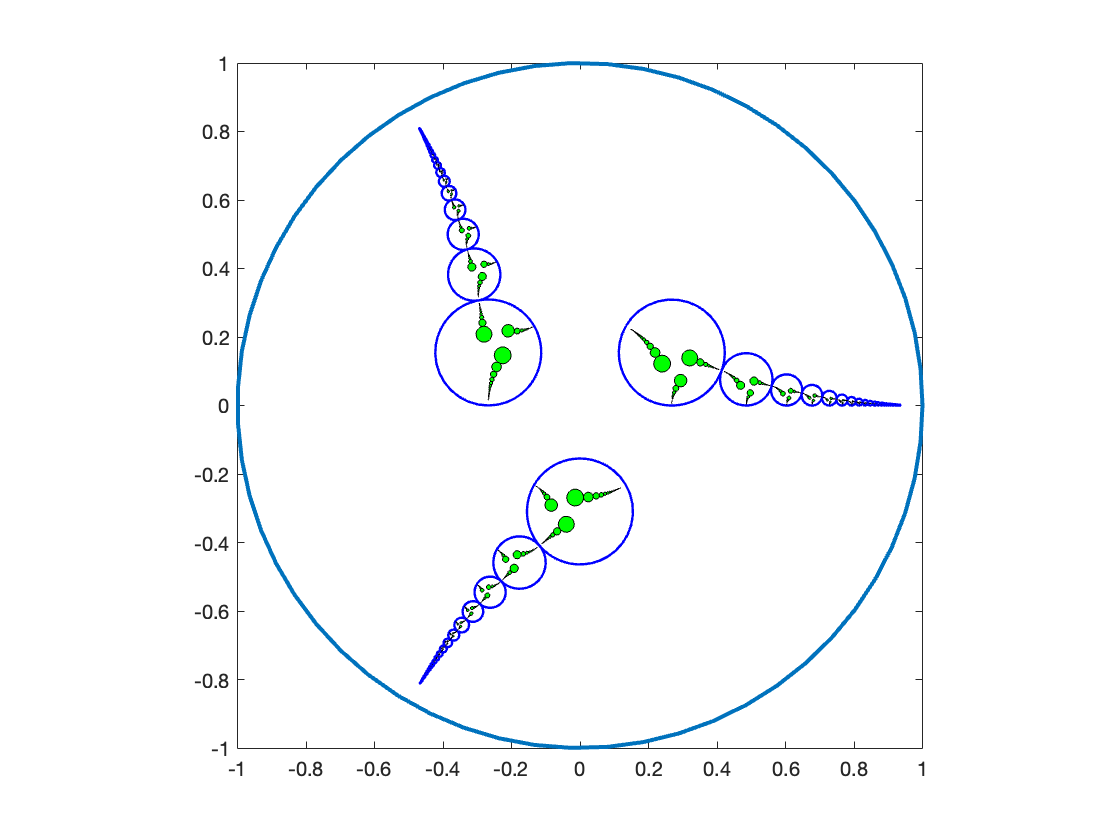}
    \caption{The Apollonian gasket with three infinite families of branches, second iteration}
    \label{fig:Apollonian3_2}
  \end{minipage}
\end{figure}

\begin{table}[ht]
\caption{Hausdorff dimension estimates for the Apollonian three infinite families}
\label{table_Apollonian_three}
\centering
\vspace{.1in}
\begin{tabular}{|c|c|c|c|c|c|c|}
\hline
$h_{max}$ & Ndof  &  interp. err. &  $N$ & $\beta^{\mathrm{odd}}_{t_u,N}$ &  lower Haus. dim. & upper Haus dim.\\
\hline
0.0077... &  138977 & 0.0054...& 100 & 0.0368... & 1.0692...& 1.1228...\\
\hline
0.0064... &  217090 & 0.0034...& 200 & 0.0220...  & 1.0717...& 1.0908...\\
\hline
0.0052... &  308729 & 0.0021...& 600 & 0.0070...  & 1.0732...&  1.0783...\\
\hline
\end{tabular}
\end{table}



\bibliographystyle{plain}
    \bibliography{biblio4.bib}

@article{mcmullen,
  author  = {McMullen, Curtis T.},
  title   = {{Hausdorff} dimension and conformal dynamics. {III}.
             Computation of dimension},
  journal = {American Journal of Mathematics},
  volume  = {120},
  pages   = {691--721},
  year    = {1998}
}

@article {nusspriverduyn,
    AUTHOR = {Nussbaum, Roger D. and Priyadarshi, Amit and Verduyn Lunel,
              Sjoerd},
     TITLE = {Positive operators and {H}ausdorff dimension of invariant
              sets},
   JOURNAL = {Trans. Amer. Math. Soc.},
  FJOURNAL = {Transactions of the American Mathematical Society},
    VOLUME = {364},
      YEAR = {2012},
    NUMBER = {2},
     PAGES = {1029--1066},
      ISSN = {0002-9947,1088-6850},
   MRCLASS = {47B65 (28A78)},
  MRNUMBER = {2846362},
       DOI = {10.1090/S0002-9947-2011-05484-X},
       URL = {https://doi.org/10.1090/S0002-9947-2011-05484-X},
}

@article{DasSim,
  author  = {Das, Tushar and Simmons, David Samuel},
  title   = {Exact dimension functions of the prime continued fraction
             {Cantor} set},
  journal = {Ergodic Theory and Dynamical Systems},
  volume  = {45},
  number  = {6},
  pages   = {1757--1776},
  year    = {2025},
  doi     = {10.1017/etds.2024.57},
  url     = {https://doi.org/10.1017/etds.2024.57}
}

@article {MUapol,
    AUTHOR = {Mauldin, R. Daniel and Urba\'nski, Mariusz},
     TITLE = {Dimension and measures for a curvilinear {S}ierpinski gasket
              or {A}pollonian packing},
   JOURNAL = {Adv. Math.},
  FJOURNAL = {Advances in Mathematics},
    VOLUME = {136},
      YEAR = {1998},
    NUMBER = {1},
     PAGES = {26--38},
      ISSN = {0001-8708,1090-2082},
   MRCLASS = {28A78 (28A80 52C15)},
  MRNUMBER = {1623671},
MRREVIEWER = {Claude\ Tricot},
       DOI = {10.1006/aima.1998.1732},
       URL = {https://doi.org/10.1006/aima.1998.1732},
}

@article {CLU2,
    AUTHOR = {Chousionis, Vasileios and Leykekhman, Dmitriy and
              Urba\'{n}ski, Mariusz},
     TITLE = {On the dimension spectrum of infinite subsystems of continued
              fractions},
   JOURNAL = {Trans. Amer. Math. Soc.},
  FJOURNAL = {Transactions of the American Mathematical Society},
    VOLUME = {373},
      YEAR = {2020},
    NUMBER = {2},
     PAGES = {1009--1042},
      ISSN = {0002-9947,1088-6850},
   MRCLASS = {37D35 (11J70 11K50 28A80 37B10 37C30 37C40)},
  MRNUMBER = {4068257},
MRREVIEWER = {Ksenia\ Fedosova},
       DOI = {10.1090/tran/7984},
       URL = {https://doi.org/10.1090/tran/7984},
}

@article{CLUW_2025,
  author  = {Chousionis, Vasileios and Leykekhman, Dmitriy and
             Urba{\'n}ski, Mariusz and Wendt, Erik},
  title   = {The dimension spectrum of the infinitely generated
             {Apollonian} gasket},
  journal = {Ergodic Theory and Dynamical Systems},
  volume  = {46},
  number  = {11},
  pages   = {2702--2725},
  year    = {2026},
  doi     = {10.1017/etds.2026.10307},
  url     = {https://doi.org/10.1017/etds.2026.10307}
}

@unpublished{BCLW,
  author = {Brown, Jacob and Chousionis, Vasileios and
            Leykekhman, Dmitriy and Urba{\'n}ski, Mariusz},
  title = {On the dimension spectrum of $n$-dimensional continued fractions},
  note   = {In preparation}
}

@article{poli,
  author  = {Jenkinson, Oliver and Pollicott, Mark},
  title   = {Rigorous effective bounds on the {Hausdorff} dimension of
             continued fraction {Cantor} sets: A hundred decimal digits
             for the dimension of {$E_2$}},
  journal = {Advances in Mathematics},
  volume  = {325},
  pages   = {87--115},
  year    = {2018},
  doi     = {10.1016/j.aim.2017.11.028},
  url     = {https://doi.org/10.1016/j.aim.2017.11.028}
}

@article{kontor,
  author  = {Bourgain, Jean and Kontorovich, Alex},
  title   = {On {Zaremba}'s conjecture},
  journal = {Annals of Mathematics},
  volume  = {180},
  pages   = {137--196},
  year    = {2014}
}

@article{Falk,
  author  = {Falk, Richard S. and Nussbaum, Roger D.},
  title   = {{$C^m$} eigenfunctions of {Perron--Frobenius} operators and
             a new approach to numerical computation of {Hausdorff}
             dimension: applications in {$\mathbb{R}^1$}},
  journal = {Journal of Fractal Geometry},
  volume  = {5},
  number  = {3},
  pages   = {279--337},
  year    = {2018},
  doi     = {10.4171/JFG/62},
  url     = {https://doi.org/10.4171/JFG/62}
}

@article {baifinch,
    AUTHOR = {Bai, Zai-Qiao and Finch, Steven R.},
     TITLE = {Precise calculation of {H}ausdorff dimension of {A}pollonian
              gasket},
   JOURNAL = {Fractals},
  FJOURNAL = {Fractals. Complex Geometry, Patterns, and Scaling in Nature
              and Society},
    VOLUME = {26},
      YEAR = {2018},
    NUMBER = {4},
     PAGES = {1850050, 9},
      ISSN = {0218-348X,1793-6543},
   MRCLASS = {28A80 (30F35 37D35)},
  MRNUMBER = {3858491},
       DOI = {10.1142/S0218348X18500500},
       URL = {https://doi.org/10.1142/S0218348X18500500},
}

@article {boyd,
    AUTHOR = {Boyd, David W.},
     TITLE = {The residual set dimension of the {A}pollonian packing},
   JOURNAL = {Mathematika},
  FJOURNAL = {Mathematika. A Journal of Pure and Applied Mathematics},
    VOLUME = {20},
      YEAR = {1973},
     PAGES = {170--174},
      ISSN = {0025-5793},
   MRCLASS = {52A45},
  MRNUMBER = {493763},
MRREVIEWER = {H.\ Groemer},
       DOI = {10.1112/S0025579300004745},
       URL = {https://doi.org/10.1112/S0025579300004745},
}

@article {polur,
    AUTHOR = {Pollicott, Mark and Urba\'nski, Mariusz},
     TITLE = {Asymptotic counting in conformal dynamical systems},
   JOURNAL = {Mem. Amer. Math. Soc.},
  FJOURNAL = {Memoirs of the American Mathematical Society},
    VOLUME = {271},
      YEAR = {2021},
    NUMBER = {1327},
     PAGES = {v+139},
      ISSN = {0065-9266,1947-6221},
      ISBN = {978-1-4704-6577-3; 978-1-4704-6632-9},
   MRCLASS = {37-02 (37D35 37F35 52C26 60F05)},
  MRNUMBER = {4288186},
MRREVIEWER = {Claudio\ Bonanno},
       DOI = {10.1090/memo/1327},
       URL = {https://doi.org/10.1090/memo/1327},
}

@book {Mauldin_Urbanski_2003,
    AUTHOR = {Mauldin, R. Daniel and Urba\'{n}ski, Mariusz},
     TITLE = {Graph directed {M}arkov systems},
    SERIES = {Cambridge Tracts in Mathematics},
    VOLUME = {148},
      NOTE = {Geometry and dynamics of limit sets},
 PUBLISHER = {Cambridge University Press, Cambridge},
      YEAR = {2003},
     PAGES = {xii+281},
      ISBN = {0-521-82538-5},
   MRCLASS = {37C45 (28A75 28A78 28A80 37B10 37C30 37F30 82B05)},
  MRNUMBER = {2003772},
MRREVIEWER = {Marc\ Kesseb\"{o}hmer},
       DOI = {10.1017/CBO9780511543050},
       URL ={https://doi.org/10.1017/CBO9780511543050},
}

@book {Krantz,
    AUTHOR = {Krantz, Steven G.},
     TITLE = {Function theory of several complex variables},
      NOTE = {Reprint of the 1992 edition},
 PUBLISHER = {AMS Chelsea Publishing, Providence, RI},
      YEAR = {2001},
     PAGES = {xvi+564},
      ISBN = {0-8218-2724-3},
   MRCLASS = {32-01},
  MRNUMBER = {1846625},
MRREVIEWER = {Harold\ P.\ Boas},
       DOI = {10.1090/chel/340},
       URL = {https://doi.org/10.1090/chel/340},
}

@book {KU,
    AUTHOR = {Kotus, Janina and Urba\'{n}ski, Mariusz},
     TITLE = {Meromorphic dynamics. {V}ol. {I}. {A}bstract ergodic theory,
              geometry, graph directed {M}arkov systems, and conformal
              measures},
    SERIES = {New Mathematical Monographs},
    VOLUME = {46},
 PUBLISHER = {Cambridge University Press, Cambridge},
      YEAR = {2023},
     PAGES = {xxv+482},
      ISBN = {978-1-009-21591-6; 978-1-009-21605-0},
   MRCLASS = {37-02 (28Axx 37Axx 37Bxx 37Fxx)},
  MRNUMBER = {4530063},
}

@book {nara,
    AUTHOR = {Narasimhan, Raghavan},
     TITLE = {Several complex variables},
    SERIES = {Chicago Lectures in Mathematics},
 PUBLISHER = {University of Chicago Press, Chicago, Ill.-London},
      YEAR = {1971},
     PAGES = {x+174},
   MRCLASS = {32-01},
  MRNUMBER = {342725},
MRREVIEWER = {J.\ Kajiwara},
}

@book {MRUvol3,
    AUTHOR = {Urba\'{n}ski, Mariusz and Roy, Mario and Munday, Sara},
     TITLE = {Non-invertible dynamical systems. {V}ol. 3. {A}nalytic
              endomorphisms of the {R}iemann sphere},
    SERIES = {De Gruyter Expositions in Mathematics},
    VOLUME = {69.3},
 PUBLISHER = {De Gruyter, Berlin},
      YEAR = {2023},
     PAGES = {viii--xxvii and 926--1334},
      ISBN = {978-3-11-076984-5; 978-3-11-076987-6; 978-3-11-076989-0},
   MRCLASS = {37-01 (30Cxx 37Dxx 37Fxx)},
  MRNUMBER = {4653003},
}

@book {MRU,
    AUTHOR = {Urba\'{n}ski, Mariusz and Roy, Mario and Munday, Sara},
     TITLE = {Non-invertible dynamical systems. {V}ol. 2. {F}iner
              thermodynamic formalism---distance expanding maps and
              countable state subshifts of finite type, conformal {GDMS}s,
              {L}asota-{Y}orke maps and fractal geometry},
    SERIES = {De Gruyter Expositions in Mathematics},
    VOLUME = {69.2},
 PUBLISHER = {De Gruyter, Berlin},
      YEAR = {2022},
     PAGES = {viii--xxix+432--921},
      ISBN = {978-3-11-070061-9; 978-3-11-070269-9; 978-3-11-070273-6},
   MRCLASS = {37-01 (28A80 37Axx 37D35 37Fxx)},
  MRNUMBER = {4486444},
}

@incollection{bu2,
	Author = {Bumby, Richard T.},
	Booktitle = {Number theory ({N}ew {Y}ork, 1983--84)},
	Doi = {10.1007/BFb0074599},
	Mrclass = {11K55 (11K50 28A05)},
	Mrnumber = {803348},
	Mrreviewer = {F. Schweiger},
	Pages = {1--8},
	Publisher = {Springer, Berlin},
	Series = {Lecture Notes in Math.},
	Title = {Hausdorff dimension of sets arising in number theory},
	Url = {https://doi.org/10.1007/BFb0074599},
	Volume = {1135},
	Year = {1985}}

@article{bu1,
	Author = {Bumby, Richard T.},
	Doi = {10.1515/crll.1982.331.192},
	Fjournal = {Journal f\"ur die Reine und Angewandte Mathematik},
	Issn = {0075-4102},
	Journal = {J. Reine Angew. Math.},
	Mrclass = {10K50 (28A75)},
	Mrnumber = {647383},
	Mrreviewer = {V. Drobot},
	Pages = {192--206},
	Title = {Hausdorff dimensions of {C}antor sets},
	Url = {https://doi.org/10.1515/crll.1982.331.192},
	Volume = {331},
	Year = {1982}}

@article{jenktexan,
	Author = {Jenkinson, Oliver},
	Doi = {10.1142/S0219493704000900},
	Fjournal = {Stochastics and Dynamics},
	Issn = {0219-4937},
	Journal = {Stoch. Dyn.},
	Mrclass = {28A80 (11K50 11K55 37C45 37F35)},
	Mrnumber = {2069367},
	Mrreviewer = {J\"org M. Thuswaldner},
	Number = {1},
	Pages = {63--76},
	Title = {On the density of {H}ausdorff dimensions of bounded type continued fraction sets: the {T}exan conjecture},
	Url = {https://doi.org/10.1142/S0219493704000900},
	Volume = {4},
	Year = {2004}}

@article{cu1,
	Author = {Cusick, Thomas W.},
	Doi = {10.1112/S0025579300009050},
	Fjournal = {Mathematika. A Journal of Pure and Applied Mathematics},
	Issn = {0025-5793},
	Journal = {Mathematika},
	Mrclass = {10F20},
	Mrnumber = {0472721},
	Mrreviewer = {K. E. Hirst},
	Number = {2},
	Pages = {166--172},
	Title = {Continuants with bounded digits},
	Url = {https://doi.org/10.1112/S0025579300009050},
	Volume = {24},
	Year = {1977}}

@article{cu2,
	Author = {Cusick, Thomas W.},
	Doi = {10.1112/S002557930000930X},
	Fjournal = {Mathematika. A Journal of Pure and Applied Mathematics},
	Issn = {0025-5793},
	Journal = {Mathematika},
	Mrclass = {10F20},
	Mrnumber = {0498413},
	Mrreviewer = {K. E. Hirst},
	Number = {1},
	Pages = {107--109},
	Title = {Continuants with bounded digits. {II}},
	Url = {https://doi.org/10.1112/S002557930000930X},
	Volume = {25},
	Year = {1978}}

@article{cu3,
	Author = {Cusick, Thomas W.},
	Doi = {10.1007/BF01304191},
	Fjournal = {Monatshefte f\"ur Mathematik},
	Issn = {0026-9255},
	Journal = {Monatsh. Math.},
	Mrclass = {11J70 (11K50)},
	Mrnumber = {781688},
	Mrreviewer = {K. E. Hirst},
	Number = {2},
	Pages = {105--109},
	Title = {Continuants with bounded digits. {III}},
	Url = {https://doi.org/10.1007/BF01304191},
	Volume = {99},
	Year = {1985}}

@article {CTU,
    AUTHOR = {Chousionis, Vasileios and Tyson, Jeremy and Urba\'{n}ski,
              Mariusz},
     TITLE = {Conformal graph directed {M}arkov systems on {C}arnot groups},
   JOURNAL = {Mem. Amer. Math. Soc.},
  FJOURNAL = {Memoirs of the American Mathematical Society},
    VOLUME = {266},
      YEAR = {2020},
    NUMBER = {1291},
     PAGES = {viii+155},
      ISSN = {0065-9266,1947-6221},
      ISBN = {978-1-4704-4215-6; 978-1-4704-6245-1},
   MRCLASS = {22E25 (11J70 28A78 37B10 37C30 37D35 37F35 53C17)},
  MRNUMBER = {4126256},
MRREVIEWER = {Stefan\ Groote},
       DOI = {10.1090/memo/1291},
       URL = {https://doi.org/10.1090/memo/1291},
}

@article{JenkinsonO_PollicottM_2001,
	Author = {Jenkinson, Oliver and Pollicott, Mark},
	Fjournal = {Ergodic Theory and Dynamical Systems},
	Issn = {0143-3857},
	Journal = {Ergodic Theory Dynam. Systems},
	Mrclass = {37C30 (11K50 28A80 37C45 37M25)},
	Mrnumber = {1855840},
	Mrreviewer = {Dieter H. Mayer},
	Number = {5},
	Pages = {1429--1445},
	Title = {Computing the dimension of dynamically defined sets: {$E_2$} and bounded continued fractions},
	Url = {https://doi.org/10.1017/S0143385701001687},
	Volume = {21},
	Year = {2001}}

@article{FN_high_order,
  author  = {Falk, Richard S. and Nussbaum, Roger D.},
  title   = {Hidden positivity and a new approach to numerical computation
             of {Hausdorff} dimension: higher order methods},
  journal = {Journal of Fractal Geometry},
  volume  = {9},
  number  = {1/2},
  pages   = {23--72},
  year    = {2022},
  doi     = {10.4171/JFG/111},
  url     = {https://doi.org/10.4171/JFG/111}
}

@article{good,
	Author = {Good, Irving J.},
	Journal = {Proc. Cambridge Philos. Soc.},
	Mrclass = {27.2X},
	Mrnumber = {0004878},
	Mrreviewer = {P. Erd\"os},
	Pages = {199--228},
	Title = {The fractional dimensional theory of continued fractions},
	Volume = {37},
	Year = {1941}}

@article{hentex,
	Author = {Hensley, D.},
	Doi = {10.1006/jnth.1996.0058},
	Fjournal = {Journal of Number Theory},
	Issn = {0022-314X},
	Journal = {J. Number Theory},
	Mrclass = {11K55 (28A78)},
	Mrnumber = {1387719},
	Mrreviewer = {Oto Strauch},
	Number = {1},
	Pages = {9--45},
	Title = {A polynomial time algorithm for the {H}ausdorff dimension of continued fraction {C}antor sets},
	Url = {https://doi.org/10.1006/jnth.1996.0058},
	Volume = {58},
	Year = {1996}}

@article {huang,
    AUTHOR = {Huang, ShinnYih},
     TITLE = {An improvement to {Z}aremba's conjecture},
   JOURNAL = {Geom. Funct. Anal.},
  FJOURNAL = {Geometric and Functional Analysis},
    VOLUME = {25},
      YEAR = {2015},
    NUMBER = {3},
     PAGES = {860--914},
      ISSN = {1016-443X,1420-8970},
   MRCLASS = {11K50},
  MRNUMBER = {3361774},
MRREVIEWER = {Christoph\ Baxa},
       DOI = {10.1007/s00039-015-0327-6},
       URL = {https://doi.org/10.1007/s00039-015-0327-6},
}

@incollection {polcont,
    AUTHOR = {Jenkinson, Oliver and Pollicott, Mark},
     TITLE = {Rigorous dimension estimates for {C}antor sets arising in
              {Z}aremba theory},
 BOOKTITLE = {Dynamics: topology and numbers},
    SERIES = {Contemp. Math.},
    VOLUME = {744},
     PAGES = {83--107},
 PUBLISHER = {Amer. Math. Soc., [Providence], RI},
      YEAR = {2020},
      ISBN = {978-1-4704-5100-4},
   MRCLASS = {11K55 (11A55 11K50 37C30 37F35)},
  MRNUMBER = {4062559},
       DOI = {10.1090/conm/744/14980},
       URL = {https://doi.org/10.1090/conm/744/14980},
}

@article{bowen,
 Author = {Bowen, Rufus},
	Coden = {PMIHA6},
	Fjournal = {Institut des Hautes \'Etudes Scientifiques. Publications Math\'ematiques},
	Issn = {0073-8301},
	Journal = {Inst.\ Hautes \'Etudes Sci.\ Publ.\ Math.},
	Mrclass = {57S05 (30C20 32G15 58F11)},
	Mrnumber = {556580},
	Mrreviewer = {L. A. Bunimovich},
	Volume = {50},
	Pages = {11--25},
	Title = {Hausdorff dimension of quasicircles},
	Url = {http://www.numdam.org/item?id=PMIHES_1979__50__11_0},
	Year = {1979}}

@article {MUCIFS,
    AUTHOR = {Mauldin, R. Daniel and Urba\'nski, Mariusz},
     TITLE = {Dimensions and measures in infinite iterated function systems},
   JOURNAL = {Proc. London Math. Soc. (3)},
  FJOURNAL = {Proceedings of the London Mathematical Society. Third Series},
    VOLUME = {73},
      YEAR = {1996},
    NUMBER = {1},
     PAGES = {105--154},
      ISSN = {0024-6115,1460-244X},
   MRCLASS = {28A80 (28A78 58F03 58F11)},
  MRNUMBER = {1387085},
MRREVIEWER = {Claude\ Tricot},
       DOI = {10.1112/plms/s3-73.1.105},
       URL = {https://doi.org/10.1112/plms/s3-73.1.105},
}

@article {dimspec,
    AUTHOR = {Chousionis, Vasileios and Leykekhman, Dmitriy and
              Urba\'{n}ski, Mariusz},
     TITLE = {The dimension spectrum of conformal graph directed {M}arkov
              systems},
   JOURNAL = {Selecta Math. (N.S.)},
  FJOURNAL = {Selecta Mathematica. New Series},
    VOLUME = {25},
      YEAR = {2019},
    NUMBER = {3},
     PAGES = {Paper No. 40, 74},
      ISSN = {1022-1824,1420-9020},
   MRCLASS = {37B10 (11J70 11K50 28A80 37C30 37C40 37D35)},
  MRNUMBER = {3960790},
MRREVIEWER = {Poj\ Lertchoosakul},
       DOI = {10.1007/s00029-019-0487-6},
       URL = {https://doi.org/10.1007/s00029-019-0487-6},
}

@article{hen1,
	Author = {Hensley, Doug},
	Doi = {10.1016/0022-314X(89)90005-X},
	Fjournal = {Journal of Number Theory},
	Issn = {0022-314X},
	Journal = {J. Number Theory},
	Mrclass = {11K50 (11K55)},
	Mrnumber = {1034198},
	Mrreviewer = {Richard T. Bumby},
	Number = {2},
	Pages = {182--198},
	Title = {The {H}ausdorff dimensions of some continued fraction {C}antor sets},
	Url = {https://doi.org/10.1016/0022-314X(89)90005-X},
	Volume = {33},
	Year = {1989}}

@article{hen2,
    AUTHOR = {Hensley, Doug},
     TITLE = {Continued fractions, {C}antor sets, {H}ausdorff dimension, and
              transfer operators and their analytic extension},
   JOURNAL = {Discrete Contin. Dyn. Syst.},
  FJOURNAL = {Discrete and Continuous Dynamical Systems. Series A},
    VOLUME = {32},
      YEAR = {2012},
    NUMBER = {7},
     PAGES = {2417--2436},
      ISSN = {1078-0947},
   MRCLASS = {11K55 (11J70 11K50 11Y60 37A45)},
  MRNUMBER = {2900553},
MRREVIEWER = {Radhakrishnan Nair},
       URL = {https://doi.org/10.3934/dcds.2012.32.2417}}

@book{henbook,
    AUTHOR = {Hensley, Doug},
     TITLE = {Continued fractions},
 PUBLISHER = {World Scientific Publishing Co. Pte. Ltd., Hackensack, NJ},
      YEAR = {2006},
     PAGES = {xiv+245},
      ISBN = {981-256-477-2},
   MRCLASS = {11A55 (11J70 28D05 30B70 37A45 37F10 40A15)},
  MRNUMBER = {2351741},
MRREVIEWER = {Oto Strauch},
       URL = {https://doi.org/10.1142/9789812774682},
}

@article {hutch,
    AUTHOR = {Hutchinson, John E.},
     TITLE = {Fractals and self-similarity},
   JOURNAL = {Indiana Univ. Math. J.},
  FJOURNAL = {Indiana University Mathematics Journal},
    VOLUME = {30},
      YEAR = {1981},
    NUMBER = {5},
     PAGES = {713--747},
      ISSN = {0022-2518,1943-5258},
   MRCLASS = {49F20 (00A69 28A12 58C27)},
  MRNUMBER = {625600},
MRREVIEWER = {F.\ J.\ Almgren, Jr.},
       DOI = {10.1512/iumj.1981.30.30055},
       URL = {https://doi.org/10.1512/iumj.1981.30.30055},
}

@article{jarnik,
    author = {Jarnik, Vojt\v{e}ch},
    title = {Zur metrischen Theorie der diophantischen Approximationen (On the metric theory of Diophantine approximations)},
    journal = {Prace Matematyczno-Fizyczne},
    volume = {36},
    year = {1928-1929}
}

@unpublished{verma,
  author        = {Verma, Saurabh},
  title         = {{Hausdorff} dimension and infinitesimal similitudes on
                   complete metric spaces},
  note          = {Preprint, arXiv:2101.07520},
  year          = {2021},
  eprint        = {2101.07520},
  archivePrefix = {arXiv},
  primaryClass  = {math.DS},
  doi           = {10.48550/arXiv.2101.07520},
  url           = {https://arxiv.org/abs/2101.07520}
}

@article{FalkRS_NussbaumRD_2016b,
  author  = {Falk, Richard S. and Nussbaum, Roger D.},
  title   = {A new approach to numerical computation of {Hausdorff}
             dimension of iterated function systems: applications to
             complex continued fractions},
  journal = {Integral Equations and Operator Theory},
  volume  = {90},
  number  = {5},
  pages   = {Paper No. 61},
  year    = {2018},
  doi     = {10.1007/s00020-018-2485-z},
  url     = {https://doi.org/10.1007/s00020-018-2485-z}
}

@article {Vytnova_Wormell_2025,
    AUTHOR = {Vytnova, Polina L. and Wormell, Caroline L.},
     TITLE = {Hausdorff dimension of the {A}pollonian gasket},
   JOURNAL = {Invent. Math.},
  FJOURNAL = {Inventiones Mathematicae},
    VOLUME = {239},
      YEAR = {2025},
    NUMBER = {3},
     PAGES = {909--946},
      ISSN = {0020-9910,1432-1297},
   MRCLASS = {52C26},
  MRNUMBER = {4861129},
       DOI = {10.1007/s00222-024-01311-y},
       URL = {https://doi.org/10.1007/s00222-024-01311-y},
}

@article {Babenko_1978,
    AUTHOR = {Babenko, K. I. and Jur'ev, S. P.},
     TITLE = {Discretization of a problem of {G}auss},
   JOURNAL = {Dokl. Akad. Nauk SSSR},
  FJOURNAL = {Doklady Akademii Nauk SSSR},
    VOLUME = {240},
      YEAR = {1978},
    NUMBER = {6},
     PAGES = {1273--1276},
      ISSN = {0002-3264},
   MRCLASS = {65D20 (10A35)},
  MRNUMBER = {499751},
}

@article {Pollicott_Vytnova_2022,
    AUTHOR = {Pollicott, M. and Vytnova, P.},
     TITLE = {Hausdorff dimension estimates applied to {L}agrange and
              {M}arkov spectra, {Z}aremba theory, and limit sets of
              {F}uchsian groups},
   JOURNAL = {Trans. Amer. Math. Soc. Ser. B},
  FJOURNAL = {Transactions of the American Mathematical Society. Series B},
    VOLUME = {9},
      YEAR = {2022},
     PAGES = {1102--1159},
      ISSN = {2330-0000},
   MRCLASS = {37C30 (11K55 30F35 37F32)},
  MRNUMBER = {4527964},
MRREVIEWER = {Christian\ Frederik\ Wei\ss},
       DOI = {10.1090/btran/109},
       URL = {https://doi.org/10.1090/btran/109},
}

@phdthesis{Erik_thesis,
  title        = {Rigorous Hausdorff dimension estimates for conformal fractals},
  author       = {Erik Wendt},
  year         = 2025,
  month        = {June},
  address      = {Storrs, CT},
  school       = {University of Connecticut},
  type         = {PhD thesis}
}

@book {Ulam_1960,
    AUTHOR = {Ulam, S. M.},
     TITLE = {A collection of mathematical problems},
    SERIES = {Interscience Tracts in Pure and Applied Mathematics},
    VOLUME = {no. 8},
 PUBLISHER = {Interscience Publishers, New York-London},
      YEAR = {1960},
     PAGES = {xiii+150},
   MRCLASS = {00.00},
  MRNUMBER = {120127},
MRREVIEWER = {A.\ Dvoretzky},
}

@article{LiTY_1976,
title = {Finite approximation for the {F}robenius-{P}erron operator. A solution to {U}lam's conjecture},
journal = {Journal of Approximation Theory},
volume = {17},
number = {2},
pages = {177-186},
year = {1976},
issn = {0021-9045},
doi = {10.1016/0021-9045(76)90037-X},
url = {https://www.sciencedirect.com/science/article/pii/002190457690037X},
author = {Tien-Yien Li}
}

@article{klus2015numerical,
title = {On the numerical approximation of the Perron-Frobenius and Koopman operator},
journal = {Journal of Computational Dynamics},
volume = {3},
number = {1},
pages = {51--79},
year = {2016},
issn = {2158-2491},
doi = {10.3934/jcd.2016003},
url = {https://www.aimsciences.org/article/id/99a4878a-5263-40bf-b4e8-bbf7c2730a2a},
author = {Stefan  Klus and P{\'e}ter Koltai and Christof  Sch{\"u}tte}
}

@article{arbabi2017ergodic,
  title={Ergodic theory, dynamic mode decomposition, and computation of spectral properties of the Koopman operator},
  author={Arbabi, Hassan and Mezic, Igor},
  journal={SIAM Journal on Applied Dynamical Systems},
  volume={16},
  number={4},
  pages={2096--2126},
  year={2017},
  publisher={SIAM}
}

@article{colbrook2024rigorous,
  title={Rigorous data-driven computation of spectral properties of Koopman operators for dynamical systems},
  author={Colbrook, Matthew J and Townsend, Alex},
  journal={Communications on Pure and Applied Mathematics},
  volume={77},
  number={1},
  pages={221--283},
  year={2024},
  publisher={Wiley Online Library}
}

@unpublished{bandtlow2020lagrange,
  author        = {Bandtlow, Oscar F. and Slipantschuk, Julia},
  title         = {{Lagrange} approximation of transfer operators associated
                   with holomorphic data},
  note          = {Preprint, arXiv:2004.03534},
  year          = {2020},
  eprint        = {2004.03534},
  archivePrefix = {arXiv},
  primaryClass  = {math.DS},
  doi           = {10.48550/arXiv.2004.03534},
  url           = {https://arxiv.org/abs/2004.03534}
}

@article{ding1996finite,
  title={Finite approximations of Frobenius-Perron operators. A solution of Ulam's conjecture to multi-dimensional transformations},
  author={Ding, Jiu and Zhou, Aihui},
  journal={Physica D: Nonlinear Phenomena},
  volume={92},
  number={1-2},
  pages={61--68},
  year={1996},
  publisher={Elsevier}
}

@article{dellnitz1999approximation,
  title={On the approximation of complicated dynamical behavior},
  author={Dellnitz, Michael and Junge, Oliver},
  journal={SIAM Journal on Numerical Analysis},
  volume={36},
  number={2},
  pages={491--515},
  year={1999},
  publisher={SIAM}
}

@article{junge2009discretization,
  title={Discretization of the Frobenius--Perron operator using a sparse Haar tensor basis: the sparse Ulam method},
  author={Junge, Oliver and Koltai, P{\'e}ter},
  journal={SIAM Journal on Numerical Analysis},
  volume={47},
  number={5},
  pages={3464--3485},
  year={2009},
  publisher={SIAM}
}

@article{Waldron_1998,
 ISSN = {00361429},
 URL = {http://www.jstor.org/stable/2587128},
 author = {Shayne Waldron},
 journal = {SIAM Journal on Numerical Analysis},
 number = {3},
 pages = {1191--1200},
 publisher = {Society for Industrial and Applied Mathematics},
 title = {The Error in Linear Interpolation at the Vertices of a Simplex},
 urldate = {2026-09-03},
 volume = {35},
 year = {1998},
  doi = {10.1137/S0036142996313154}
}
\end{document}